\documentclass{amsart}
\usepackage{amsmath}    % 数学符号
\usepackage{amssymb}   % 数学符号
\usepackage{amsfonts}
\usepackage{array}      % 表格扩展功能
\usepackage{float}
\usepackage{booktabs}  % 用于更好的表格线
\usepackage{tabularx}
\usepackage{enumitem}
\usepackage{mathabx}
\usepackage{graphicx}
\usepackage{dynkin-diagrams}
\usepackage{color, xcolor} % 导入颜色包
\usepackage{verbatim}
\usepackage{psfrag}
\usepackage{amsthm} % 导入定理环境包
\usepackage{amscd}    % 专门用于 CD 交换图环境
\usepackage{tikz-cd}
\usepackage{wasysym}
\usepackage{xy}% 交换图支持（用于 \xymatrix）
\usepackage[colorlinks=true,
            linkcolor=red!50!black,
            citecolor=blue,
            urlcolor=teal, 
            pdftitle={Quasi-admissible, raisable nilpotent orbits and covering Barbasch-Vogan duality},
            pdfauthor={Yi-Yang Zhang}]{hyperref}
\usepackage[margin=3cm]{geometry}  % 所有边距设为 2cm
\theoremstyle{plain} % 默认样式：粗体标题，斜体正文

\usepackage[shortalphabetic]{amsrefs}
 \makeatletter
\def\append@label@year@{%
    \safe@set\@tempcnta\bib@year
    \edef\bib@citeyear{\the\@tempcnta}%
    \ifnum\bib@citeyear>9
      \append@to@stem{%
          \ifx\bib@year\@empty
          \else
            \@xp\year@short \bib@citeyear \@nil
          \fi
      }%
    \fi
}
\makeatother
\newtheorem{theorem}{Theorem}[section] % [section] 表示按章节编号
\newtheorem{lemma}[theorem]{Lemma} % [theorem] 表示与定理共享计数器
\newtheorem{proposition}[theorem]{Proposition}

\theoremstyle{definition} % 定义样式：粗体标题，正体正文

\theoremstyle{remark} % 备注样式：斜体标题，正体正文

\numberwithin{equation}{section}

\begin{document}

\title{Quasi-admissible, raisable nilpotent orbits and covering Barbasch-Vogan duality}
\author[Zhang]{Yi-Yang Zhang}
\address{School of Mathematical Sciences, Zhejiang University, Hangzhou 310058, China}
\email{3220102591@zju.edu.cn}

\date{}

%wavefront set

\begin{abstract}
For simply-connected Lie groups of type E over \( p \)-adic local field \( F \), we determine the degree of the cover required for a given $F$-split nilpotent orbit to be quasi-admissible or raisable, respectively. Combining this result with the previously computed data for other types by Gao-Liu-Tsai, we prove that all \(F\)-split nilpotent orbits whose geometry type contained in the image of the covering Barbasch-Vogan duality map $d_{\mathrm{BV},G}^{(n)}$ of almost-simple Lie groups $G$ in each Cartan type are always $\overline{G}^{(n)}$-quasi-admissible.
\end{abstract}

\maketitle

\setcounter{tocdepth}{2}  % 0: chapter, 1: section, 2: subsection, 3: subsubsection
\tableofcontents

\section{Introduction}
Let \( F \) be a \( p \)-adic local field of characteristic zero, and let \(\mathbf{G}\) be a connected split linear reductive group over \(F\). Denote by \(F^{\mathrm{alg}}\) the algebraic closure of \(F\). We consider the group \(G:=\mathbf{G}(F)\) or its Brylinski--Deligne covers \cite{BD01}
\[
\begin{tikzcd}
    \mu_n \arrow[r,hook]& \overline{G}^{(n)} \arrow[r,two heads]& G ,
\end{tikzcd}
\]
where we assume that $F$ contains the full group $\mu_n$ of $n$-th roots of unity. When the number \(n\) is understood, we also write \(\overline{G}:=\overline{G}^{(n)}\). Denote by \(\operatorname{Irr}_{\mathrm{gen}}(\overline{G})\) the set of equivalence classes of irreducible genuine representations of \(\overline{G}\), where a genuine representation of \(\overline{G}\) is one such that \( \mu_n \) acts via a fixed embedding \( \mu_n \hookrightarrow \mathbb{C}^\times \). Let \(\mathcal{N}(G)\) denote the partially ordered set of nilpotent orbits in \(\mathfrak{g}_F=\operatorname{Lie}(G)\) under the conjugation action of \(G\), the partial order being given by the closure ordering with respect to the usual topology on \(\mathfrak{g}_F\) induced from that of \(F\). There is also the set \(\mathcal{N}(\mathbf{G})\) of geometric (stable) nilpotent orbits, equipped with the partial order coming from the Zariski topology. For every \(\pi \in \operatorname{Irr}_{\mathrm{gen}}(\overline{G})\), one can write its character expansion at the identity as a distribution on \(\mathfrak{g}\)
\begin{equation}\label{eq:dist}
    \Theta_\pi \circ \exp = \sum_{\mathcal{O} \in \mathcal{N}(G)}c_\mathcal{O}(\pi) \widehat{v_{\mathcal{O}}},
\end{equation}
according to \cites{MW87,Varma14,Prakash15}. Here \(c_\mathcal{O}(\pi) \in \mathbb{C}\) and \(\widehat{v_{\mathcal{O}}}\) denotes the Fourier transform of the orbital integral over \(\mathcal{O}\). Define

\[
\mathcal{N}(\pi):=\{\mathcal{O} \in \mathcal{N}(G) : c_\mathcal{O}(\pi) \neq 0\}
\]
and let
\[
\text{WF}(\pi) \subset \mathcal{N}(\pi)
\]
be the subset consisting of all maximal elements in $\mathcal{N}(\pi)$. 
One also consider
\[
\text{WF}^{\text{geo}}(\pi) \subset \mathcal{N}(\pi) \otimes F^{\text{alg}} \subset \mathcal{N}(\mathbf{G})
\]
consisting of maximal elements in \( \mathcal{N}(\pi) \otimes F^{\text{alg}} \). Determining \(\operatorname{WF}(\pi)\) and \(\operatorname{WF}^{\mathrm{geo}}(\pi)\) is of great interest. For instance, one can use these sets to define the Gelfand--Kirillov dimension of \(\pi\) as
\[
d_{\mathrm{GK}}(\pi):=\frac{1}{2}\max\{\dim\mathcal{O}:\mathcal{O} \in \operatorname{WF}(\pi)\},
\]
which satisfies the asymptotic relation
\[
\dim\pi^K \approx c_\pi \cdot \operatorname{vol}(K)^{-d_{\mathrm{GK}}(\pi)},
\]
as certain open compact congruence subgroups \(K \subset \overline{G}\) shrink to the identity; here \(c_\pi \in \mathbb{C}\) depends only on \(\pi\) (see \cite{Savin94} for the linear case, which is also applicable to covering setting by \eqref{eq:dist}).

However, determining \(\operatorname{WF}(\pi)\) for a general \(\pi \in \operatorname{Irr}_{\mathrm{gen}}(\overline{G})\) is difficult. Some existing works provide sufficient conditions for a nilpotent orbit to lie in the complement
\[
\mathcal{N}(G)-\bigcup_{\pi \in \text{Irr}_{\text{gen}}(\overline{G})} \text{WF}(\pi).
\]
These conditions are expressed in terms of the notions of \textit{admissibility} \cites{Duflo82,Nevins99,Nevins02,Torasso97}, \textit{quasi-admissibility} \cite{GGS21,GLT25}, and \textit{raisability} \cites{JLS16,GLT25}. Their relations can be summarized as follows:
\begin{itemize}
    \item A \(\overline{G}\)-admissible orbit is always \(\overline{G}\)-quasi-admissible.
    \item If an orbit \(\mathcal{O}\) is either non-\(\overline{G}\)-quasi-admissible or \(\overline{G}\)-raisable, then \(\mathcal{O}\) does not belong to \(\operatorname{WF}(\pi)\) for any \(\pi \in \operatorname{Irr}_{\mathrm{gen}}(\overline{G})\).
\end{itemize}

\subsection{Main results}
In \cite{GLT25}, Gao-Liu-Tsai established explicit criteria on the degree \(n\) and on \(F\)-split nilpotent orbits for quasi-admissibility and raisability. Their calculations cover all simple Lie groups except those of type E. The first motivation of the present paper is to complete the computation begun in loc. cit. 

In Section~2 we provide the necessary preliminaries for our calculations and proofs. In Section~3 we complete the determination of quasi-admissibility and raisability for type E simply-connected groups, yielding our first result.

\begin{theorem}
    Let \(\overline{G}^{(n)}\) be the \(n\)-fold cover of a simply-connected exceptional group of type \(E_r\) with \(6 \leq r \leq 8\). Conditions on \(n\) for the quasi-admissibility and raisability of an \(F\)-split nilpotent orbit are given in Tables \ref{tab:e6}--\ref{tab:e8c}, whenever our method applies.
\end{theorem}

As an application we consider the covering Barbasch--Vogan duality defined in \cite{GLLS26}, which is expected to provide an upper bound for \(\operatorname{WF}^{\mathrm{geo}}(\pi)\). Let \(\overline{G}^\vee\) be the complex Langlands dual group of \(\overline{G}^{(n)}\). The covering Barbasch--Vogan duality is a map
\[
d_{\mathrm{BV},G}^{(n)}: \mathcal{N}(\overline{G}^\vee) \rightarrow \mathcal{N}(\mathbf{G}).
\]
Denote by \(\mathrm{AZ}\) the Aubert--Zelevinsky involution and by \(\phi_\pi\) the hypothetical \(L\)-parameter of \(\pi\). Gao-Liu-Lo-Shahidi proposed the following conjecture.

    \begin{itemize}
        \item For every \(\pi \in \operatorname{Irr}_{\mathrm{gen}}(\overline{G}^{(n)})\), one has  
        \[
          \operatorname{WF}^{\mathrm{geo}}(\operatorname{AZ}(\pi)) \leq d_{\mathrm{BV}}^{(n)}(\mathcal{O}(\phi_{\pi})),
        \]
        the equality is attained for some \(\pi\) in the tempered \(L\)-packet.
    \end{itemize}

This conjecture generalises independent expectations formulated in the linear setting in \cite{HLLS24} and \cite{CK25}. In view of this conjecture, one expects that every \(F\)-split nilpotent orbit whose geometry type lying in the image of \(d_{\mathrm{BV}}^{(n)}\) should be \(\overline{G}^{(n)}\)-quasi-admissible. Our second motivation is to verify this expectation.

In Section~4 we prove that, for classical groups, every \(F\)-split nilpotent orbit whose geometry type lying in the image of the covering Barbasch--Vogan duality map \(d_{\mathrm{BV},G}^{(n)}\) is always \(\overline{G}^{(n)}\)-quasi-admissible. Combining this with the results of \cite{GLT25} and Section~3, we obtain the following theorem, which is consistent with the above conjecture.

\begin{theorem} \label{Main result}
    Let \(\overline{G}^{(n)}\) be the \(n\)-fold cover of any of \(\mathrm{GL}_r\), \(\mathrm{SO}_{2r+1}\), \(\mathrm{SO}_{2r}\), \(\mathrm{Sp}_{2r}\) and simply-connected almost-simple exceptional Lie group. Then for every positive integer \(n\), every \(F\)-split nilpotent orbit whose geometry type contained in the image of the covering Barbasch--Vogan duality map \(d_{\mathrm{BV},G}^{(n)}\) is \(\overline{G}^{(n)}\)-quasi‑admissible.
\end{theorem}

\subsection{Acknowledgement}
The author is deeply grateful to Professor Fan Gao, for invaluable guidance and support throughout this work.

\section{Preliminaries}
We assume throughout that \(F\) is a \(p\)-adic local field of characteristic zero and that \(\mathbf{G}\) is a connected split linear reductive group over \(F\) with Lie algebra \(\mathfrak{g}\).

\subsection{Nilpotent orbits}
Let \(u\) be a nilpotent element of \(\mathfrak{g}\). The Jacobson--Morozov theorem yields a Lie algebra homomorphism \(\gamma \colon \mathfrak{sl}_2 \to \mathfrak{g}\) such that \(u = \gamma\bigl(\begin{smallmatrix}0&0\\1&0\end{smallmatrix}\bigr)\). By abuse of notation, we also denote the image \(\gamma(\mathfrak{sl}_2)\) by \(\gamma\). Let \(\mathbf{G}_\gamma\) be the centralizer in \(\mathbf{G}\) of \(\gamma\). The nilpotent orbit \(\mathcal{O}_u\) is called \emph{\(F\)-split} if \(\mathbf{G}_\gamma\) is split over \(F\); following \cite{GLT25}, we shall always assume this to be the case.

Fix a Cartan subalgebra \(\mathfrak{h}\) containing the semisimple element \(h\) of the \(\mathfrak{sl}_2\)-triple associated with \(\gamma\), together with a root system \(\Phi\) and a set of simple roots \(\{\alpha_i\}\). Without loss of generality, we always assume that \(h\) lies in the dominant Weyl chamber. It is known that \(\alpha_i(h) \in \{0,1,2\}\) for every \(i\), see \cite{CM93}. Consequently, one may label each node of the Dynkin diagram of \(\mathfrak{g}\) with the integer \(\alpha_i(h)\), thereby obtaining a \textit{weighted Dynkin diagram}. Each nilpotent orbit corresponds uniquely to a weighted Dynkin diagram. The classification of nilpotent orbits for almost-simple exceptional Lie groups via weighted Dynkin diagrams was given by Bala and Carter, see, for example, \cite{Carter93}.

Relative to \(h\), the Lie algebra \(\mathfrak{g}\) admits a weight‑space decomposition
\[
\mathfrak{g} = \bigoplus_{i \in \mathbb{Z}} \mathfrak{g}[i],
\]
where \(\mathfrak{g}[i]\) is the weight \(i\) subspace of \(\mathfrak{g}\) with respect to the adjoint action of \(h\). Without loss of generality, it suffices to consider \(\mathfrak{g}[i]\) with \(i \ge 0\). From the weighted Dynkin diagram one can directly read off the following useful facts.

\begin{lemma} \label{prop:g0}
    Assume that \(\mathfrak{g}\) is semisimple. The Dynkin diagram of \(\mathfrak{g}[0]\) is obtained by keeping only the nodes corresponding to simple roots \(\alpha_i\) with \(\alpha_i(h)=0\). Moreover, the dimension of the centre of \(\mathfrak{g}[0]\) equals the number of nodes labeled by a value \(\alpha_j(h)>0\).
\end{lemma}

\begin{proof}
    By definition,
    \[
    \mathfrak{g}[0] = \mathfrak{h} \oplus \bigoplus_{\alpha \in \Phi_0} \mathfrak{g}_\alpha,
    \]
    where \(\Phi_0 = \{\alpha = \sum_i m_i\alpha_i \in \Phi : \alpha_i(h) = 0,\; m_i \in \mathbb{Z},\ \forall i\}\). A positive root space \(\mathfrak{g}_\alpha\) belongs to \(\mathfrak{g}[0]\) precisely when \(\alpha\) is a sum of simple roots whose labels are zero; this gives the first assertion. For the second, note that the centre of \(\mathfrak{g}[0]\) is contained in \(\mathfrak{h}\). An element \(H \in \mathfrak{h}\) commutes with every \(\mathfrak{g}_\alpha\) (\(\alpha\in\Phi_0\)) if and only if \(\alpha_i(H)=0\) for all \(\alpha_i\) with \(\alpha_i(h)=0\). The dimension of the space of such \(H\) is exactly the number of simple roots with a non‑zero label.
\end{proof}

\begin{lemma} \label{lem:gi}
    For \(i > 0\), we have
    \[
    \mathfrak{g}[i] = \bigoplus_{\alpha \in \Phi_i} \mathfrak{g}_\alpha,
    \]
    where
    \[
    \Phi_i = \bigl\{\alpha = \sum_j m_j\alpha_j \in \Phi : \sum_j m_j\alpha_j(h) = i,\; m_j \in \mathbb{Z}_{\ge 0},\ \forall j \bigr\}.
    \]
\end{lemma}

\begin{proof}
    The proof is analogous to that of Lemma~\ref{prop:g0}.
\end{proof}

Let \(\mathbf{G}_h\) be the closed subgroup of \(\mathbf{G}\) with Lie algebra \(\mathfrak{g}[0]\) and write \(\mathrm{L} := \mathbf{G}_{h,\mathrm{der}}\). Then each \(\mathfrak{g}[i]\) becomes a representation of \(\mathrm{L}\). The structure of \(\mathfrak{g}[i]\) as an \(\mathrm{L}\)-module, which lies at the heart of our computations, was determined by Jackson and Noël in \cite{JN05}. 

\subsection{Quasi-admissible orbits}
Let \(Y\) denote the cocharacter lattice of a maximal split torus \(\mathbf{T}\) of \(\mathbf{G}\), and let \(Q \colon Y \to \mathbb{Z}\) be a Weyl-invariant quadratic form. We assume that the group \(F^{\times}\) contains the full group \(\mu_n\) of \(n\)th roots of unity. Consider the pair \( (D, \eta = \mathbf{1}) \), where \( D \) is a ``bisector'' of \( Q \) (see \cite[\S~2.6]{GG18}). In any case, associated with \( (D, \mathbf{1}) \) one has an \( n \)-fold central cover \(\overline{G}^{(n)}\) of \(G:=\mathbf{G}(F)\):

\[
\begin{tikzcd}
    \mu_n \arrow[r,hook]& \overline{G}^{(n)} \arrow[r,two heads]& G .
\end{tikzcd}
\]

For simplicity, we shall often write \(\overline{G}\) in place of \(\overline{G}^{(n)}\).

Assume that the derived group \(\mathbf{G}_{\mathrm{der}}\) of \(\mathbf{G}\) is almost simple. The \textit{Brylinski--Deligne invariant} of the covering \(\overline{G}\) is then defined as the integer \(\operatorname{Inv}_{\mathrm{BD}}(\overline{G}) := Q(\alpha^{\vee})\), where \(\alpha^{\vee}\) is any short coroot of \(\mathbf{G}\). This invariant enjoys the following functorial property \cite{GLT25}.

\begin{proposition}
    Let \(\zeta \colon \mathbf{G} \to \mathbf{H}\) be an algebraic group homomorphism, and denote by \(\zeta^{\sharp} \colon Y_G \to Y_H\) the induced map on cocharacter lattices. Let \(\overline{H}\) be an \(n\)-fold cover of \(\mathbf{H}(F)\) associated with a quadratic form \(Q_H\). Then the pull-back cover \(\zeta^{*}(\overline{H})\) of \(\mathbf{G}(F)\) satisfies
    \[
    \operatorname{Inv}_{\mathrm{BD}}\bigl(\zeta^{*}(\overline{H})\bigr)=Q_H \circ \zeta^{\sharp}(\alpha_G^{\vee}),
    \]
    where \(\alpha_G^{\vee}\) is any short coroot of \(\mathbf{G}\).
\end{proposition}

Given a nilpotent element \(u\), one obtains a symplectic form \(\omega_u \colon \mathfrak{g}_F \times \mathfrak{g}_F \to F\) defined by \(\omega_u(x,y):=\kappa(u,[x,y])\), where \(\kappa\) denotes the Killing form. Restricting \(\omega_u\) to the weight space \(\mathfrak{g}[1]\) yields a non‑degenerate symplectic form that is preserved by the group \(G_\gamma:=\mathbf{G}_\gamma(F)\). Consequently, there is a natural homomorphism \(\phi \colon G_\gamma \to \operatorname{Sp}(\mathfrak{g}[1])\).

By pull-backs, we obtain the commutative diagram

\begin{equation}\label{eq:pullback}
\begin{tikzcd}
\overline{G}_{\gamma}^{(n)} \times_{G_\gamma} \overline{G}_{\gamma}^{(2),\phi} \arrow[r] \arrow[d] & 
\overline{G}_{\gamma}^{(2),\phi} \arrow[r] \arrow[d,two heads] & 
\operatorname{Mp}(\mathfrak{g}[1]) \arrow[d,two heads] \\
\overline{G}_{\gamma}^{(n)} \arrow[r,two heads] \arrow[d,hook] & 
G_{\gamma} \arrow[r, "\phi"] \arrow[d,hook] & 
\operatorname{Sp}(\mathfrak{g}[1])  \\
\overline{G}^{(n)} \arrow[r,two heads] & 
G . &
\end{tikzcd}
\end{equation}

We write $\overline{G}_\gamma^{(n,2)}:=\overline{G}_{\gamma}^{(n)} \times_{G_\gamma} \overline{G}_{\gamma}^{(2),\phi}$ for simplicity. 
Let $\mathbf{G}_{\gamma,0}$ be the connected component of the identity in $\mathbf{G}_\gamma$. Replacing \(G_\gamma\) by \(G_{\gamma,0}:= \mathbf{G}_{\gamma,0}(F) \subset G_\gamma\), one obtains analogous covers \(\overline{G}_{\gamma,0}^{(n)}\) and \(\overline{G}_{\gamma,0}^{(2),\phi}\).

A representation of a group \(H\) with \(\mu_n \times \mu_2 \subset Z(H)\) is called \textit{\((n,2)\)-genuine} if both central subgroups \(\mu_n\) and \(\mu_2\) act faithfully. An nilpotent orbit \(\mathcal{O}_u \subset \mathfrak{g}\) is said to be \textit{\(\overline{G}^{(n)}\)-quasi-admissible} if the group \(\overline{G}_\gamma^{(n,2)}\) admits a finite-dimensional \((n,2)\)-genuine representation. When the covering group is clear from the context, we shall simply speak of quasi-admissibility.

Let \(n^*:=\mathrm{lcm}(n,2)\), and \(\mathbf{f}:\mathbf{G}_{\gamma,\mathrm{sc}} \twoheadrightarrow \mathbf{G}_{\gamma,\mathrm{der}}\) be the simply-connected cover. Assume that \(G_{\gamma,0}\) is almost simple first. Set 
\[\mathrm{Q}_1:=\operatorname{Inv}_{\mathrm{BD}}\bigl(\overline{G}_{\gamma,0}^{(n)}\bigr),\quad \mathrm{Q}_2:=\operatorname{Inv}_{\mathrm{BD}}\bigl(\overline{G}_{\gamma,0}^{(2),\phi}\bigr).
\]
The following proposition, taken from \cite[Proposition~2.5]{GLT25}, makes the criterion for quasi-admissibility explicit and computable.

\begin{proposition} \label{prop:quasi}
     Keep the notion above. If \(n^*\) is coprime to the size of \(\mathrm{Ker}\mathbf{f}\), then the corresponding \(F\)-split nilpotent orbit $\mathcal{O}_u$ is quasi-admissiable if and only if
    \begin{align*}
        &\text{$(i)$ }n|\mathrm{Q}_1 \text{ when } 2|\mathrm{Q}_2 \text{, and} \\
        &\text{$(ii)$ } \frac{n}{\text{gcd}(n,\mathrm{Q}_1)}=2 \text{ when } 2 \nmid \mathrm{Q}_2.
    \end{align*}
\end{proposition}

By the proof of this proposition in \cite{GLT25}, we know that the computation of the invariants \(\mathrm{Q}_1\) and \(\mathrm{Q}_2\) only depends on the structure of $\mathbf{G}_{\gamma,\mathrm{der}}$. When \(\mathbf{G}_{\gamma,\mathrm{der}}\) decomposes as product of almost-simple groups, the quasi-admissibility condition factors accordingly \cite[p.18]{GLT25}.

\begin{proposition} \label{prop:product}
    Suppose \(\mathbf{G}_{\gamma,\mathrm{der}} = \prod\mathbf{G}^{a_i} \) with \(\mathbf{G}^{a_i}\)'s almost-simple. Let \((\mathrm{Q}_1^{a_i},\mathrm{Q}_2^{a_i})\) be the pair of invariants attached to the factor \(\mathbf{G}^{a_i}\) for all \(i\) as in Proposition~\ref{prop:quasi}. Then the condition on \(n\) for the quasi-admissibility of \(\mathbf{G}_{\gamma,\mathrm{der}}\) is the intersection of the conditions for all the \(\mathbf{G}^{a_i}\)'s separately.
\end{proposition}

In the calculation of the invariants \(\mathrm{Q}_1\) and \(\mathrm{Q}_2\), the \(\mathfrak{sl}_2\)-triple corresponding to a long root of \(\mathbf{G}_{\gamma,\mathrm{der}}\) plays a crucial role. We shall denote this triple by \(\mathfrak{sl}_{2,\xi}\) for brevity. When a factorisation as in Proposition~\ref{prop:product} occurs, we write \(\mathfrak{sl}_{2,\xi_{a_i}}\) for the corresponding triples attached to the factors \(\mathbf{G}^{a_i}\) for all \(i\).

If the assumption in Proposition~\ref{prop:quasi} fails, namely, \(n^*\) fails to be coprime to the size of \(\mathrm{Ker}\mathbf{f}\), the “only if" part of the statement still hold. Thus, we just need check quasi-admissibility by definition for finite many cases. Fortunately, such phenomena are rare and always can be checked by the following lemma \cite[Lemma~2.8]{GLT25} or its analogue.

\begin{lemma} \label{lem:SO}
    The Brylinski-Deligne cover \(\overline{\mathrm{SO}}_k^{(m)},k \geq 3\) that associated with data \((D, \mathbf{1})\) has a finite-dimensional \(\mu_m\)-genuine representation if and only if \(m \mid \mathrm{Inv}_{\mathrm{BD}}(\mathrm{SO}_k^{(m)})\).
\end{lemma}

\subsection{Raisable orbits}
We now introduce the notion of \(\overline{G}\)-raisability, following \cite{JLS16}. 

Let \(\mathfrak{g}_\gamma \subset \mathfrak{g}\) be the centraliser of \(\gamma\) in \(\mathfrak{g}\). Assume the following condition:

\begin{center}
\begin{minipage}{0.8\textwidth}
\begin{itemize}
    \item[(C0)] There exists a non‑trivial homomorphism \(\tau \colon \mathfrak{sl}_2 \to \mathfrak{g}_\gamma\).
\end{itemize}
\end{minipage}
\end{center}

In this situation we write \(\mathfrak{sl}_{2,\tau} = \tau(\mathfrak{sl}_2)\). For a non‑zero \(x \in F\), set \(u_\tau := \tau(e_{-}(x))\). Then \(\gamma \oplus \tau \colon \mathfrak{sl}_2 \to \mathfrak{g}\) is a Jacobson–Morozov homomorphism for the nilpotent element \( u + u_\tau\). Denote by \(\mathfrak{g}[j,l] \subset \mathfrak{g}\) the subspace of vectors that have \(\gamma\)-weight \(j\) and \(\tau\)-weight \(l\). We impose three further conditions:

\begin{center}
\begin{minipage}{0.8\textwidth}
\begin{itemize}
    \item[(C1)] All \(\tau\)-weights occurring in \(\mathfrak{g}\) are bounded by \(2\).
    
    \item[(C2)] As an \(\mathfrak{sl}_{2,\tau}\)-module,
        \[
        \mathfrak{g}[1] = \mathfrak{g}[1]^{\mathfrak{sl}_{2,\tau}} \oplus m V_2.
        \]
    
    \item[(C3)] \(\dim \mathfrak{g}[0,2] = 1 + \dim \mathfrak{g}[2,2]\).
\end{itemize}
\end{minipage}
\end{center}

By slight abuse of notation we still write \(\tau \colon \mathrm{SL}_2 \to G_\gamma\) for the group homomorphism corresponding to the Lie algebra map \(\tau\). Condition (C2) yields a natural homomorphism \(\phi_m \colon \mathrm{SL}_2 \to \mathrm{Sp}_{2m}\) induced by \(\phi \circ \tau\). Proceeding in complete analogy with the construction leading to diagram~\eqref{eq:pullback}, we obtain covers \(\overline{\mathrm{SL}}_2^{(n),\tau}\), \(\overline{\mathrm{SL}}_2^{(2),\phi_m}\) and their fibre product \(\overline{\mathrm{SL}}_2^{(n,2)}\). The nilpotent orbit \(\mathcal{O}_u\) is called \textit{\(\overline{G}^{(n)}\)-raisable} if \(\overline{\mathrm{SL}}_2^{(n,2)}\) admits a finite-dimensional \((n,2)\)-genuine representation.

A criterion completely analogous to Proposition~\ref{prop:quasi} decides raisability \cite[Proposition~2.10]{GLT25}. To avoid confusion we denote the relevant pair of invariants by \((\mathrm{Q}_1^\tau,\mathrm{Q}_2^\tau)\) instead of \((\mathrm{Q}_1,\mathrm{Q}_2)\).

The following proposition provides a large supply of suitable maps \(\tau\); it shows that in most cases one can simply take \(\tau\) to be the one attached to a long root as shown below.

\begin{proposition} \label{prop:simple}
    Assume that \(\mathfrak{g}\) is a simple Lie algebra. If there exists a long root \(\alpha\) of \(\mathfrak{g}_\gamma\) that is also a long root of \(\mathfrak{g}\), then the homomorphism \(\tau\) associated with \(\alpha\) satisfies conditions (C1)--(C3).
\end{proposition}

\begin{proof}
    For a root space \(\mathfrak{g}_\beta \subset \mathfrak{g}\) the \(\tau\)-weight equals \(\beta(\alpha^\vee)\). Because every root of \(\mathfrak{g}\) can be expressed explicitly as a linear combination of simple roots, a direct verification shows that the \(\tau\)-weights are indeed bounded by \(2\); this gives (C1) and consequently (C2). Moreover, the subspace of \(\tau\)-weight \(2\) in \(\mathfrak{g}\) is precisely \(\mathfrak{g}_\alpha\), which is certainly contained in \(\mathfrak{g}_\gamma\). Hence \(\dim \mathfrak{g}[0,2] = 1\) and \(\dim \mathfrak{g}[2,2] = 0\), so condition (C3) holds as well.
\end{proof}

\subsection{Prehomogeneous spaces}
Let \(V\) be a rational representation of \(\mathbf{G}\). The pair \((\mathbf{G}, V)\) is called a \textit{prehomogeneous space} if \(V\) contains a Zariski‑open \(\mathbf{G}\)-orbit; points of this orbit are called \textit{generic points}. For $x \in V$, define the \textit{isotropy subgroup} $\mathbf{G}_x \subset \mathbf{G}$ as the stablizer of $x$. When \(x\) is generic, the corresponding isotropy subgroup is referred to as a \textit{generic isotropy subgroup} of \((\mathbf{G},V)\) and is denoted by \(\mathbf{G}_V\).

A useful characterisation of prehomogeneous spaces is given in \cite[\S2.Proposition~2]{SK77}.

\begin{lemma} \label{lem:dim}
    The pair \((\mathbf{G},V)\) is a prehomogeneous space if and only if there exists an element \(x \in V\) such that
    \[
    \dim \mathbf{G}_x = \dim \mathbf{G} - \dim V .
    \]
    Moreover, for such an \(x\) one has \(\mathbf{G}_V = \mathbf{G}_x\).
\end{lemma}

A prehomogeneous space \((\mathbf{G}, V)\) is called \textit{irreducible} if \(V\) is irreducible as a \(\mathbf{G}\)-module. Sato and Kimura \cite[\S5]{SK77} determined the structure of \(\mathbf{G}_V\) for every irreducible prehomogeneous space. The following proposition handles the situation where \(V\) is reducible; it will be used repeatedly in later computations.

\begin{proposition} \label{prop:restrict}
    Suppose that \((\mathbf{G},V_a)\), \((\mathbf{G},V_a \oplus V_b)\) and \((\mathbf{G}_{V_a},V_b)\) are all prehomogeneous spaces. Then
    \[
    (\mathbf{G}_{V_a})_{V_b} = \mathbf{G}_{V_a \oplus V_b}.
    \]
\end{proposition}

\begin{proof}
    Choose a generic point \(X_a\) for \((\mathbf{G},V_a)\) and a generic point \(X_b\) for \((\mathbf{G}_{V_a},V_b)\). By definition,
    \[
    (\mathbf{G}_{V_a})_{V_b} = (\mathbf{G}_{X_a})_{X_b} = \mathbf{G}_{(X_a,X_b)} .
    \]
    Applying Lemma~\ref{lem:dim} to each of the three spaces, we obtain
    \[
    \dim (\mathbf{G}_{V_a})_{V_b} = \dim \mathbf{G}_{V_a \oplus V_b}.
    \]
    Hence \((X_a,X_b)\) is a generic point of \((\mathbf{G},V_a \oplus V_b)\), and the two isotropy subgroups coincide.
\end{proof}

\section{Explicit calculations for type \(E\)} \label{sec:E}

\subsection{Method of computation and notation}
In this section we determine, for each \(F\)-split nilpotent orbit of a simply-connected almost-simple group of type \(E\), the set of cover degrees \(n\) for which the orbit is quasi-admissible or raisable, respectively. The method of computation, which follows \cite{JLS16,GLT25,Nevins02}, can be summarised as follows.

Recall that \(L\) denotes the derived subgroup of \(\mathbf{G}_0\). It is well-known that \(\mathbf{G}_0\) is a Levi subgroup of \(\mathbf{G}\) so that \(L\) is always simply-connected; then by Lemma~\ref{prop:g0} and Lemma~\ref{lem:gi}, we obtain the structure of \(L\) and \(\mathbf{G}_0\). The work of Jackson and Noël \cite{JN05} provides the structure of \(\mathfrak{g}[2]\) as an \(L\)-module, thereby exhibiting the pair \((\mathbf{G}_0, \mathfrak{g}[2])\) as a prehomogeneous space. From this we compute the generic isotropy subgroup \((\mathbf{G}_0)_{\mathfrak{g}[2]}\) by \cite{SK77} and Proposition~\ref{prop:restrict}. Because the set of nil‑positive elements for the semisimple part of the \(\mathfrak{sl}_2\)-triple \(\gamma\) is Zariski‑open in \(\mathfrak{g}[2]\), associated with the definition of prehomogeneous space we have an isomorphism \((\mathbf{G}_0)_{\mathfrak{g}[2]} \cong \mathbf{G}_\gamma\), and the embeddings of \((\mathbf{G}_0)_{\mathfrak{g}[2]}\) and of \(\mathbf{G}_\gamma\) into \(\mathbf{G}_0\) coincide up to conjugation by \(\mathbf{G}_0\). Again from \cite{JN05} one obtains the decomposition of \(\mathfrak{g}[1]\) as an \(L\)-module. Restricting this decomposition to the subalgebra \(\mathfrak{sl}_{2,\xi}\) (attached to a long root of \(\mathbf{G}_{\gamma,\mathrm{der}}\)) yields the data needed to compute the invariants \((\mathrm{Q}_1,\mathrm{Q}_2)\); these invariants decide quasi‑admissibility via Proposition~\ref{prop:quasi}.

For raisability, the structure of \(\mathbf{G}_{\gamma}\) gives the Lie algebra \(\mathfrak{g}_{\gamma}\). This information determines all possible choices of an auxiliary homomorphism \(\tau\). Then we check the conditions (C1)--(C3) case by case. In practice, a suitable \(\tau\) can often be chosen directly using Proposition~\ref{prop:simple}.

The results are collected in Tables~\ref{tab:e6}, \ref{tab:e7}, \ref{tab:e8} and~\ref{tab:e8c}. Firstly we discuss two extreme cases.

\begin{itemize}
    \item \textbf{Zero orbit.} For the zero orbit of a group \(\mathbf{G}\) of type \(E\) the orbit is even and \(\mathbf{G}_{\gamma,\mathrm{der}}=\mathbf{G}\). One finds \((\mathrm{Q}_1,\mathrm{Q}_2)=(1,0)\); consequently the orbit is quasi‑admissible precisely for the trivial cover \(n=1\). Concerning raisability, Proposition~\ref{prop:simple} gives \((\mathrm{Q}_1^\tau,\mathrm{Q}_2^\tau)=(1,0)\), so the zero orbit is raisable for every \(n\ge 2\).
    \item \textbf{Distinguished orbits.} For a distinguished orbit one has \(\mathbf{G}_{\gamma,\mathrm{der}}=\mathbf{1}\); hence the raisability method does not apply. The covering splits in this situation, and therefore the orbit is quasi‑admissible for every \(n\).
\end{itemize}

\subsubsection*{Notational conventions}
We write \(\mathrm{V}_j\) for the \(j\)-dimensional irreducible representation of \(\mathfrak{sl}_2\). When a almost-simple Lie group \(\mathbf{G}\) occurs as a standard representation, we denote the representation simply by the same symbol \(\mathbf{G}\) (for example, the standard representation of \(\mathrm{Spin}_N\) is written \(\mathrm{Spin}_N\)) and the corresponding dual representation by \(\mathbf{G}^*\). The fundamental weight corresponding to the standard representation of a almost-simple group is denoted \(\Lambda_1\). The two half‑spin representations of \(\mathrm{Spin}_{2n}\) are denoted \(\mathrm{Spin}_{2n}^\prime\) and \(\mathrm{Spin}_{2n}^{\prime\prime}\); the trivial representation is written \(\mathrm{F}\); \(\mathrm{E}_6\) means either of the \(27\)-dimensional \(\mathrm{E}_6\)-modules; and \(\mathrm{E}_7\) means the \(56\)-dimensional \(\mathrm{E}_7\)-module. All the exceptional Lie groups are assumed to be simply-connected.

\subsection{Type $E_6$}
Consider the Dynkin diagram below:

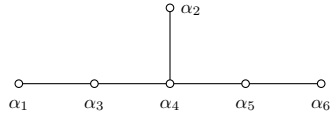
\begin{center}
\newcommand{\al}[1]{\alpha_{#1}}
\dynkin[label,label macro/.code=\al{#1},edge length=1cm]{E}{oooooo}
\end{center}

\subsubsection{\dynkin[edge length=.20cm,labels={0,1,0,0,0,0},label]{E}{oooooo} $A_1$, $\mathrm{L}=\mathrm{SL}_6$.}
\begin{align*}
    \mathfrak{g}[2]&=\mathrm{F} \\
    \mathfrak{g}[1]&=\bigwedge^3\mathrm{SL}_6
\end{align*}
As the trivial summand always reduces the dimension of the centre by one, with Proposition~\ref{prop:restrict}, it suffices to analyse the remaining summands; we shall henceforth omit further mention of such trivial contributions. We have $\mathbf{G}_{\gamma,\mathrm{der}}=\mathrm{SL}_{6}$, associated with $\Delta-\{\alpha_2\}$. As $\mathfrak{sl}_{2,\xi}$-mod, $\mathfrak{g}[1]=8\mathrm{V}_1+6\mathrm{V}_2$. Then we have $(\mathrm{Q}_1,\mathrm{Q}_2)=(\mathrm{Q}_1^\tau,\mathrm{Q}_2^\tau)=(1,6)$. Thus, the orbit is quasi-admissible for $n=1$, raisable if $n \geq 2$.

\subsubsection{\dynkin[edge length=.20cm,labels={1,0,0,0,0,1},label]{E}{oooooo} $2A_1$, $\mathrm{L}=\mathrm{Spin}_8$. } \label{SO1}
\begin{align*}
    \mathfrak{g}[2]&=\mathrm{Spin}_8 \\
    \mathfrak{g}[1]&=\mathrm{Spin}_8^\prime \oplus \mathrm{Spin}_8^{\prime\prime}
\end{align*}
We have $\mathbf{G}_{\gamma,\mathrm{der}}=\mathrm{Spin}_{7} \hookrightarrow \mathrm{Spin}_8$ by $\Lambda_1$. As $\mathfrak{sl}_{2,\xi}$-mod, $\mathfrak{g}[1]=8\mathrm{V}_1+4\mathrm{V}_2$. Then we have $(\mathrm{Q}_1,\mathrm{Q}_2)=(\mathrm{Q}_1^\tau,\mathrm{Q}_2^\tau)=(1,4)$. Thus, the orbit is quasi-admissible for $n=1$, raisable if $n \geq 2$.

\subsubsection{\dynkin[edge length=.20cm,labels={0,0,0,1,0,0},label]{E}{oooooo} $3A_1$, $\mathrm{L}=\mathrm{SL}_3^1 \times \mathrm{SL}_2 \times \mathrm{SL}_3^2$. } \label{dA2}
\begin{align*}
    \mathfrak{g}[2]&=\mathrm{SL}_3^1 \otimes \mathrm{SL}_3^2\\
    \mathfrak{g}[1]&=(\mathrm{SL}_3^1)^* \otimes \mathrm{SL}_2 \otimes (\mathrm{SL}_3^2)^*
\end{align*}
We have $\mathbf{G}_{\gamma,\mathrm{der}}=\mathrm{SL}_{2,\alpha_2}^a \times \Delta(\mathrm{SL}_{3})^b$, where $\Delta(\mathrm{SL}_{3})$ is diagonally embedded into $\mathrm{SL}_3^1 \times \mathrm{SL}_3^2=\mathrm{SL}_{3,\alpha_1,\alpha_3} \times \mathrm{SL}_{3,\alpha_5,\alpha_6}$. As $\mathfrak{sl}_{2,\xi_a}$-mod, $\mathfrak{g}[1]=9\mathrm{V}_2$. As $\mathfrak{sl}_{2,\xi_b}$-mod, $\mathfrak{g}[1]=4\mathrm{V}_1 + 4\mathrm{V}_2 + 2\mathrm{V}_3$. Then we have $(\mathrm{Q}_1^a,\mathrm{Q}_2^a)=(\mathrm{Q}_1^\tau,\mathrm{Q}_2^\tau)=(1,9),(\mathrm{Q}_1^b,\mathrm{Q}_2^b)=(2,12)$. Thus, the orbit is quasi-admissible for $n=2$, raisable if $n \neq 2$.

\subsubsection{\dynkin[edge length=.20cm,labels={0,2,0,0,0,0},label]{E}{oooooo} $A_2$, $\mathrm{L}=\mathrm{SL}_6$. } \label{A53}
\begin{align*}
    \mathfrak{g}[2]&=\bigwedge^3\mathrm{SL}_6
\end{align*}
We have $\mathbf{G}_{\gamma,\mathrm{der}}=\mathrm{SL}_{3,\alpha_1,\alpha_3}^a \times \mathrm{SL}_{3,\alpha_5,\alpha_6}^b$. Then we have $(\mathrm{Q}_1^a,\mathrm{Q}_2^a)=(\mathrm{Q}_1^b,\mathrm{Q}_2^b)=(\mathrm{Q}_1^\tau,\mathrm{Q}_2^\tau)=(1,0)$. Thus, the orbit is quasi-admissible for $n=1$, raisable if $n \geq 2$.

\subsubsection{\dynkin[edge length=.20cm,labels={1,1,0,0,0,1},label]{E}{oooooo} $A_2+A_1$, $\mathrm{L}=\mathrm{SL}_{4}$. } \label{A3*}
\begin{align*}
    \mathfrak{g}[2] &= \mathrm{F} \oplus \mathrm{SL}_4\oplus \mathrm{SL}_4^* \\
    \mathfrak{g}[1] &= \mathrm{SL}_4\oplus \bigwedge^2\mathrm{SL}_4\oplus \mathrm{SL}_4^*
\end{align*}
Where $\mathrm{SL}_4^*=:V$ denotes the dual of the standard representation of $\mathrm{SL}_4$. The generic isotropy subgroup with respect to $\mathrm{F} \oplus \mathrm{SL}_4$ of $\mathbf{G}_0$ with generic point $(1,(1,0,0,0))$ is $(\mathrm{GL}_3 \times \mathrm{GL}_1) \rtimes (G_a)^3$, acts on $V$ by $g' \cdot v :=(^{\top}(\begin{smallmatrix} x & y \\ 0 & g \end{smallmatrix})^{-1})v$, where $x \in \mathrm{GL}_1, y=(y_1,y_2,y_3) \in (G_a)^{1 \times 3}, g \in \mathrm{GL}_3, v \in V$. We can take a suitable basis $\{\beta_1,\beta_2,\beta_3,\beta_4 \}$ of $V$ such that $g'(\beta_1)=y_1\beta_1+y_2\beta_2+y_3\beta_3+x\beta_4$. Hence $\beta_1$ is a generic point in $V$, which gives generic isotropy subgroup $\mathrm{GL}_3$. And clearly $(1,(1,0,0,0),\beta_1)$ is also a generic point in $\mathfrak{g}[2]$. Thus, we have $\mathbf{G}_{\gamma,\mathrm{der}}=\mathrm{SL}_{3,\alpha_3,\alpha_4}$. As $\mathfrak{sl}_{2,\xi}$-mod, $\mathfrak{g}[1]=6V_1+4V_2$. Then we have $(\mathrm{Q}_1,\mathrm{Q}_2)=(\mathrm{Q}_1^\tau,\mathrm{Q}_2^\tau)=(1,4)$. Thus, the orbit is quasi-admissible for $n=1$, raisable if $n \geq 2$.

\subsubsection{\dynkin[edge length=.20cm,labels={2,0,0,0,0,2},label]{E}{oooooo} $2A_2$, $\mathrm{L}=\mathrm{Spin}_8$. } \label{SpinSO8}
\begin{align*}
    \mathfrak{g}[2] &= \mathrm{Spin}_8 \oplus \mathrm{Spin}_8^\prime 
\end{align*}
Let $V_a=\mathrm{Spin}_8$, we obtain prehomogeneous space $(\mathbf{G}_{V_a}, V_b)=(\mathrm{Spin}_7 \times \mathrm{GL}_1,\; \mathrm{Spin}_7^\prime)$, which has been computed in \cite{SK77}. By Proposition~\ref{prop:restrict}, we have $\mathbf{G}_{\gamma,\mathrm{der}}=\mathrm{G}_2$, whose long root is a simple root of $\mathrm{E}_6$. Then we have $(\mathrm{Q}_1,\mathrm{Q}_2)=(\mathrm{Q}_1^\tau,\mathrm{Q}_2^\tau)=(1,0)$. Thus, the orbit is quasi-admissible for $n=1$, raisable if $n \geq 2$.

\subsubsection{\dynkin[edge length=.20cm,labels={0,0,1,0,1,0},label]{E}{oooooo} $A_2+2A_1$, $\mathrm{L}=\mathrm{SL}_{2}^1 \times \mathrm{SL}_3 \times \mathrm{SL}_{2}^2$. } \label{SO31}
\begin{align*}
    \mathfrak{g}[2] &= \mathrm{SL}_{2}^1 \otimes \mathrm{SL}_3 \otimes \mathrm{SL}_{2}^2\\
    \mathfrak{g}[1] &= (\mathrm{SL}_{2}^1 \otimes\mathrm{SL}_3^*) \oplus (\mathrm{SL}_{2}^2 \otimes\mathrm{SL}_3^*)
\end{align*}
We have $\mathbf{G}_{\gamma,\mathrm{der}}=\mathrm{SL}_{2} \hookrightarrow \mathrm{SL}_{2}^1 \times \mathrm{SL}_3 \times \mathrm{SL}_{2}^2$, the embedding is given by $\Lambda_1 \otimes 2\Lambda_1 \otimes \Lambda_1$. Hence $\mathrm{Q}_1= \mathrm{Q}_{E_6}(\alpha_1^\vee + 2\alpha_2^\vee + \alpha_6^\vee )= 6$. As $\mathfrak{sl}_{2,\xi}$-mod, $\mathfrak{g}[1]=2V_2+2V_4$. Then we have $(\mathrm{Q}_1,\mathrm{Q}_2)=(6,22)$. Thus, the orbit is quasi-admissible for $n=1,2,3,6$, the only choice of $\tau$ fails to satisfy the condition (C1).

\subsubsection{\dynkin[edge length=.20cm,labels={1,2,0,0,0,1},label]{E}{oooooo} $A_3$, $\mathrm{L}=\mathrm{SL}_{4}$. } \label{Spin6}
\begin{align*}
    \mathfrak{g}[2] &= \mathrm{F} \oplus \bigwedge^2\mathrm{SL}_4 \\
    \mathfrak{g}[1] &= \mathrm{SL}_4 \oplus \mathrm{SL}_4^*
\end{align*}
We have $\mathbf{G}_{\gamma,\mathrm{der}}=\mathrm{Sp}_{4}$. As $\mathfrak{sl}_{2,\xi}$-mod, $\mathfrak{g}[1]= 4V_1 + 2V_2$. Then we have $(\mathrm{Q}_1,\mathrm{Q}_2)=(\mathrm{Q}_1^\tau,\mathrm{Q}_2^\tau)=(1,2)$. Thus, the orbit is quasi-admissible for $n=1$, raisable if $n \geq 2$.

\subsubsection{\dynkin[edge length=.20cm,labels={1,0,0,1,0,1},label]{E}{oooooo} $2A_2+A_1$, $\mathrm{L}=\mathrm{SL}_{2}^1 \times \mathrm{SL}_{2}^2 \times \mathrm{SL}_{2}^3$. } \label{dA1}
\begin{align*}
    \mathfrak{g}[2] &= (\mathrm{SL}_{2}^1 \otimes \mathrm{SL}_{2}^2) \oplus \mathrm{F} \oplus (\mathrm{SL}_{2}^1 \otimes \mathrm{SL}_{2}^3) \\
    \mathfrak{g}[1] &= \mathrm{SL}_{2}^2 \oplus \mathrm{SL}_{2}^3 \oplus (\mathrm{SL}_{2}^1 \otimes \mathrm{SL}_{2}^2 \otimes \mathrm{SL}_{2}^3)
\end{align*}
Let \(V_a = \mathrm{GL}_{2}^1 \otimes \mathrm{SL}_{2}^2\) and let \(V_b\) be the remaining summand in \(\mathfrak{g}[2]\). Then the generic isotropy subgroup for \(V_a\) is \((\mathbf{G}_0)_{V_a} = \Delta(\mathrm{SL}_{2}) \times \mathrm{GL}_{2}^3\), where \(\Delta(\mathrm{SL}_{2})\) denotes the diagonal copy inside \(\mathrm{SL}_{2}^1 \times \mathrm{SL}_{2}^2\). The pair \(\bigl((\mathbf{G}_0)_{V_a}, V_b\bigr) = \bigl(\Delta(\mathrm{SL}_2) \times \mathrm{GL}_2^3,\; \mathrm{SL}_{2} \otimes \mathrm{GL}_{2}^3\bigr)\) is a prehomogeneous space. Applying Proposition~\ref{prop:restrict} yields \(\mathbf{G}_{\gamma,\mathrm{der}} = \Delta(\mathrm{SL}_{2})\), embedded diagonally into \(\mathrm{SL}_{2,\alpha_2} \times \mathrm{SL}_{2,\alpha_3} \times \mathrm{SL}_{2,\alpha_5}\). As an \(\mathfrak{sl}_{2,\xi}\)-module, \(\mathfrak{g}[1] = 4V_2 \oplus V_4\); hence \((\mathrm{Q}_1,\mathrm{Q}_2) = (3,14)\). Consequently the orbit is quasi‑admissible for \(n = 1, 3\). No admissible homomorphism \(\tau\) satisfying conditions (C0)–(C3) exists, so the raisability test is not applicable.

\subsubsection{\dynkin[edge length=.20cm,labels={0,1,1,0,1,0},label]{E}{oooooo} $A_3+A_1$, $\mathrm{L}=\mathrm{SL}_{2}^1 \times \mathrm{SL}_{2}^2 \times \mathrm{SL}_{2}^3$. } \label{A1}
\begin{align*}
    \mathfrak{g}[2] &= \mathrm{SL}_{2}^1 \oplus \mathrm{SL}_{2}^3 \oplus (\mathrm{SL}_{2}^1 \otimes \mathrm{SL}_{2}^3)\\
    \mathfrak{g}[1] &= \mathrm{SL}_{2}^2 \oplus (\mathrm{SL}_{2}^1 \otimes \mathrm{SL}_{2}^2) \oplus( \mathrm{SL}_{2}^2 \otimes \mathrm{SL}_{2}^3)
\end{align*}
By computing dimension, $\text{dim}\mathbf{G}_{\gamma}-\text{dim}\mathrm{SL}_2^2=\text{dim}\mathbf{G}_0-\text{dim}\mathfrak{g}[2]-3=12-8-3=1$. Thus, we have $\mathbf{G}_{\gamma,\mathrm{der}}=\mathrm{SL}_{2}^2=\mathrm{SL}_{2,\alpha_4}$. As $\mathfrak{sl}_{2,\xi}$-mod, $\mathfrak{g}[1]=5V_2$. Then we have $(\mathrm{Q}_1,\mathrm{Q}_2)=(\mathrm{Q}_1^\tau,\mathrm{Q}_2^\tau)=(1,5)$. Thus, the orbit is quasi-admissible for $n=2$, raisable if $n \neq 2$.

\subsubsection{\dynkin[edge length=.20cm,labels={0,0,0,2,0,0},label]{E}{oooooo} $D_4(a_1)$, $\mathrm{L}=\mathrm{SL}_3 \times \mathrm{SL}_2 \times \mathrm{SL}_3$. } \label{323}
\begin{align*}
    \mathfrak{g}[2] &= \mathrm{SL}_3 \otimes \mathrm{SL}_2 \otimes \mathrm{SL}_3 
\end{align*}
We have $\mathbf{G}_{\gamma,\mathrm{der}}=1$, the same as in the case of distinguished orbits. Thus, the orbit is quasi-admissible for all $n$, the method for raisability does not work.

\subsubsection{\dynkin[edge length=.20cm,labels={2,2,0,0,0,2},label]{E}{oooooo} $A_4$, $\mathrm{L}=\mathrm{SL}_{4}$. }
\begin{align*}
    \mathfrak{g}[2] &= \mathrm{SL}_4 \oplus \mathrm{SL}_4^* \oplus \bigwedge^2\mathrm{SL}_4
\end{align*}
Let \(V_a = \mathrm{SL}_4 \oplus \mathrm{SL}_4^*\) and let \(V_b\) be the remaining summand in \(\mathfrak{g}[2]\). The same as \ref{A3*}, $\mathrm{SL}_4 \oplus \mathrm{SL}_4^*$ gives the isotropy subgroup $(\mathbf{G}_0)_{V_a}=\mathrm{SL}_3$. Then we obtain prehomogeneous space \(\bigl((\mathbf{G}_0)_{V_a}, V_b\bigr) = (\mathrm{SL}_3 \times 2\mathrm{GL}_1,\; \mathrm{SL}_{3} \otimes \mathrm{SL}_{3}^*)\). Similarly to the calculation in \ref{A3*} and by Proposition~\ref{prop:restrict}, we have $\mathbf{G}_{\gamma,\mathrm{der}}=\mathrm{SL}_{2} \hookrightarrow \mathrm{SL}_3 \hookrightarrow \mathrm{SL}_4$, associated with a simple root of $\mathrm{E}_6$. As $\mathfrak{sl}_{2,\xi}$-mod. Then we have $(\mathrm{Q}_1,\mathrm{Q}_2)=(\mathrm{Q}_1^\tau,\mathrm{Q}_2^\tau)=(1,0)$. Thus, the orbit is quasi-admissible for $n=1$, raisable if $n \geq 2$.

\subsubsection{\dynkin[edge length=.20cm,labels={0,2,0,2,0,0},label]{E}{oooooo} $D_4$, $\mathrm{L}=\mathrm{SL}_{3}^1 \times \mathrm{SL}_{3}^2$. }
\begin{align*}
    \mathfrak{g}[2] &= \mathrm{F} \oplus (\mathrm{SL}_{3}^1 \otimes \mathrm{SL}_{3}^2)
\end{align*}
The same as \ref{dA2}, we have $\mathbf{G}_{\gamma,\mathrm{der}}=\Delta(\mathrm{SL}_{3})$. Take $\tau$ to be associated with one of the simple roots of $\mathfrak{g}_\gamma$, the conditions $\text{(C1)-(C3)}$ are satisfied. Then we have $(\mathrm{Q}_1,\mathrm{Q}_2)=(\mathrm{Q}_1^\tau,\mathrm{Q}_2^\tau)=(2,0)$. Thus, the orbit is quasi-admissible for $n=1,2$, raisable if $n \geq 3$.

\subsubsection{\dynkin[edge length=.20cm,labels={1,1,1,0,1,1},label]{E}{oooooo} $A_4+A_1$, $\mathrm{L}=\mathrm{SL}_2$. }
\begin{align*}
    \mathfrak{g}[2] &= 2\mathrm{F} \oplus 3 \mathrm{SL}_2 
\end{align*}
By computing dimension, we have $\text{dim}\mathbf{G}_{\gamma}=\text{dim}\mathbf{G}_0-\text{dim}\mathfrak{g}[2]=8-8=0$. Thus, we have $\mathbf{G}_{\gamma,\mathrm{der}}=1$.

\begin{table}[H]
\centering
\resizebox{\textwidth}{!}{%  % 缩放表格到页面宽度
\begin{tabular}{cccccc}
\hline
$\mathcal{O}$ & special/even? & $\mathbf{G}_{\gamma,\text{der}}$ & $(Q_1, Q_2)$ & quasi-admissible, if and only if & raisable if \\
\hline
$\{0\}$ & yes/yes & $\mathrm{E}_{6}$ & $(1,0)$ & $n=1$ & $n \geq 2$ \\
$A_1$ & yes/no & $\mathrm{SL}_{6}$ & $(1,6)$ & $n=1$ & $n \geq 2$ \\
$2A_1$ & yes/no & $\mathrm{Spin}_{7}$ & $(1,4)$ & $n=1$ & $n \geq 2$ \\
$3A_1$ & no/no & $\mathrm{SL}_{2,\alpha_2} \times \mathrm{SL}_{3}$ & $(1,9),(2,12)$ & $n=2$ & $n \neq 2$ \\
$A_2$ & yes/yes & $\mathrm{SL}_{3} \times \mathrm{SL}_{3}$ & $(1, 0),(1,0)$ & $n=1$ & $n \geq 2$ \\
$A_2+A_1$ & yes/no & $\mathrm{SL}_{3}$ & $(1,4)$ & $n=1$ & $n \geq 2$ \\
$2A_2$ & yes/yes & $\mathrm{G}_2$ & $(1,0)$ &$n=1$ & $n \geq 2$ \\
$A_2+2A_1$ & yes/no & $\mathrm{SL}_{2}$ & $(6,22)$ & $n=1,2,3,6$ & n.a. \\
$A_3$ & yes/no & $\mathrm{Sp}_{4}$ & $(1,2)$ & $n=1$ & $n \geq 2$ \\
$2A_2+A_1$ & no/no & $\mathrm{SL}_{2}$ & $(3,14)$ & $n=1,3$ & n.a.  \\
$A_3+A_1$ & no/no & $\mathrm{SL}_{2,\alpha_4}$ & $(1,5)$ & $n=2$ & $n \neq 2$ \\
$D_4(a_1)$ & yes/yes & $\mathrm{1}$ & n.a. & all n & n.a. \\
$A_4$ & yes/yes & $\mathrm{SL}_{2}$ & $(1,0)$ & $n=1$ & $n \geq 2$ \\
$D_4$ & yes/yes & $\mathrm{SL}_{3}$ & $(2,0)$ & $n=1,2$ & $n \geq 3$ \\
$A_4+A_1$ & yes/no & $\mathrm{1}$ & n.a. & all n & n.a. \\
$A_5$ & no/no & $\mathrm{SL}_{2,\alpha_4}$ & $(1,3)$ & $n=2$ & $n \neq 2$ \\
$D_5(a_1)$ & yes/no & $\mathrm{1}$ & n.a. & all n & n.a. \\
$E_6(a_3)$ & yes/yes & $\mathrm{1}$ & n.a. & all n & n.a. \\
$D_5$ & yes/yes & $\mathrm{1}$ & n.a. & all n & n.a. \\
$E_6(a_1)$ & yes/yes & $\mathrm{1}$ & n.a. & all n & n.a. \\
$E_6$ & yes/yes & $\mathrm{1}$ & n.a. & all n & n.a. \\
\hline
\end{tabular}
}
\caption{Nilpotent orbits for \( \overline{E_6}^{(n)} \)} \label{tab:e6}
\end{table}

\subsubsection{\dynkin[edge length=.20cm,labels={2,1,1,0,1,2},label]{E}{oooooo} $A_5$, $\mathrm{L}=\mathrm{SL}_2$. }
\begin{align*}
    \mathfrak{g}[2] &= 5\mathrm{F} \\
    \mathfrak{g}[1] &= 3 \mathrm{SL}_2
\end{align*}
We have $\mathbf{G}_{\gamma,\mathrm{der}}=\mathrm{SL}_{2,\alpha_4}$. As $\mathfrak{sl}_{2,\xi}$-mod, $\mathfrak{g}[1]=3V_2$. Then we have $(\mathrm{Q}_1,\mathrm{Q}_2)=(\mathrm{Q}_1^\tau,\mathrm{Q}_2^\tau)=(1,3)$. Thus, the orbit is quasi-admissible for $n=2$, raisable if $n \neq 2$.

\subsubsection{\dynkin[edge length=.20cm,labels={1,2,1,0,1,1},label]{E}{oooooo} $D_5(a_1)$, $\mathrm{L}=\mathrm{SL}_2$. }
\begin{align*}
    \mathfrak{g}[2] &=  3 \mathrm{SL}_2 \oplus \mathrm{F}
\end{align*}
By computing dimension, we have $\mathbf{G}_{\gamma,\mathrm{der}}=1$.

\subsubsection{ \dynkin[edge length=.20cm,labels={2,2,0,2,0,2},label]{E}{oooooo} $D_5$, $\mathrm{L}=\mathrm{SL}_2^1 \times \mathrm{SL}_2^2$. }
\begin{align*}
    \mathfrak{g}[2] &= \mathrm{F} \oplus \mathrm{SL}_2^1 \oplus \mathrm{SL}_2^2 \oplus (\mathrm{SL}_2^1 \otimes \mathrm{SL}_2^2)
\end{align*}
By computing dimension, we have $\mathbf{G}_{\gamma,\mathrm{der}}=1$.

\subsection{Type $E_7$}
Consider the Dynkin diagram below:

\begin{center}
\newcommand{\al}[1]{\alpha_{#1}}
\dynkin[label,label macro/.code=\al{#1},edge length=1cm]{E}{ooooooo}
\end{center}

\subsubsection{\dynkin[edge length=.20cm,labels={1,0,0,0,0,0,0},label]{E}{ooooooo} $A_1$, $\mathrm{L}=\mathrm{Spin}_{12}$. }
\begin{align*}
    \mathfrak{g}[2] &= \mathrm{F} \\
    \mathfrak{g}[1] &= \mathrm{Spin}_{12}^\prime
\end{align*}
We have $\mathbf{G}_{\gamma,\mathrm{der}}=\mathrm{Spin}_{12}$, associated with $\Delta-\{\alpha_1\}$. As $\mathfrak{sl}_{2,\xi}$-mod, $\mathfrak{g}[1]=16V_1 +8V_2$. Then we have $(\mathrm{Q}_1,\mathrm{Q}_2)=(\mathrm{Q}_1^\tau,\mathrm{Q}_2^\tau)=(1,8)$. Thus, the orbit is quasi-admissible for $n=1$, raisable if $n \geq 2$.

\subsubsection{\dynkin[edge length=.20cm,labels={0,0,0,0,0,1,0},label]{E}{ooooooo} $2A_1$, $\mathrm{L}=\mathrm{SL}_{2} \times \mathrm{Spin}_{10}$. }
\begin{align*}
    \mathfrak{g}[2] &= \mathrm{Spin}_{10} \\
    \mathfrak{g}[1] &= \mathrm{SL}_2 \otimes \mathrm{Spin}_{10}^\prime
\end{align*}
We have $\mathbf{G}_{\gamma,\mathrm{der}}=\mathrm{SL}_{2,\alpha_7}^a \times \mathrm{Spin}_9^b$, where $\mathrm{Spin}_9^b \hookrightarrow \mathrm{Spin}_{10}$ by $\Lambda_1$. As $\mathfrak{sl}_{2,\xi_a}$-mod, $\mathfrak{g}[1]=16V_2$. As $\mathfrak{sl}_{2,\xi_b}$-mod, $\mathfrak{g}[1]=16V_1 + 8V_2$. Then we have $(\mathrm{Q}_1^a,\mathrm{Q}_2^a)=(\mathrm{Q}_1^\tau,\mathrm{Q}_2^\tau)=(1,16),(\mathrm{Q}_1^b,\mathrm{Q}_2^b)=(1,8)$. Thus, the orbit is quasi-admissible for $n=1$, raisable if $n \geq 2$.

\subsubsection{\dynkin[edge length=.20cm,labels={0,0,0,0,0,0,2},label]{E}{ooooooo} $(3A_1)^{\prime\prime}$, $\mathrm{L}=\mathrm{E}_{6}$. } \label{E6}
\begin{align*}
    \mathfrak{g}[2] &= \mathrm{E}_6 
\end{align*}
We have $\mathbf{G}_{\gamma,\mathrm{der}}=\mathrm{F}_{4}$, one of its long roots is $\alpha_2$. Then we have $(\mathrm{Q}_1,\mathrm{Q}_2)=(\mathrm{Q}_1^\tau,\mathrm{Q}_2^\tau)=(1,0)$. Thus, the orbit is quasi-admissible for $n=1$, raisable if $n \geq 2$.

\subsubsection{\dynkin[edge length=.20cm,labels={0,0,1,0,0,0,0},label]{E}{ooooooo} $(3A_1)^\prime$, $\mathrm{L}=\mathrm{SL}_{2} \times \mathrm{SL}_6$. } \label{A52} 
\begin{align*}
    \mathfrak{g}[2] &= \bigwedge^2\mathrm{SL}_6 \\
    \mathfrak{g}[1] &= \mathrm{SL}_{2} \otimes \bigwedge^4\mathrm{SL}_6
\end{align*}
We have $\mathbf{G}_{\gamma,\mathrm{der}}=\mathrm{SL}_{2,\alpha_1}^a \times \mathrm{Sp}_6^b$, where $\mathrm{Sp}_6 \hookrightarrow \mathrm{SL}_6$ is the nature inclusion. As $\mathfrak{sl}_{2,\xi_a}$-mod, $\mathfrak{g}[1]=15V_2$. As $\mathfrak{sl}_{2,\xi_b}$-mod, $\mathfrak{g}[1]=14V_1 + 8V_2$. Then we have $(\mathrm{Q}_1^a,\mathrm{Q}_2^a)=(\mathrm{Q}_1^{\tau_1},\mathrm{Q}_2^{\tau_1})=(1,15),(\mathrm{Q}_1^b,\mathrm{Q}_2^b)=(\mathrm{Q}_1^{\tau_2},\mathrm{Q}_2^{\tau_2})=(1,8)$. Thus, the orbit is not quasi-admissible for any $n$, raisable for all $n$.

\subsubsection{\dynkin[edge length=.20cm,labels={2,0,0,0,0,0,0},label]{E}{ooooooo} $A_2$, $\mathrm{L}=\mathrm{Spin}_{12}$. } \label{Spin12}
\begin{align*}
    \mathfrak{g}[2] &= \mathrm{Spin}_{12}^\prime
\end{align*}
We have $\mathbf{G}_{\gamma,\mathrm{der}}=\mathrm{SL}_{6} \hookrightarrow \mathrm{Spin}_{12}$. The embedding is given by $\Lambda_1 \oplus \Lambda_1$. Then we have $(\mathrm{Q}_1,\mathrm{Q}_2)=(\mathrm{Q}_1^\tau,\mathrm{Q}_2^\tau)=(1,0)$. Thus, the orbit is quasi-admissible for $n=1$, raisable if $n \geq 2$.

\subsubsection{\dynkin[edge length=.20cm,labels={0,1,0,0,0,0,1},label]{E}{ooooooo} $4A_1$, $\mathrm{L}=\mathrm{SL}_6$. }
\begin{align*}
    \mathfrak{g}[2] &= \mathrm{F} \oplus \bigwedge^2\mathrm{SL}_6 \\
    \mathfrak{g}[1] &= \mathrm{SL}_6^* \oplus \bigwedge^3\mathrm{SL}_6
\end{align*}
The same as \ref{A52}, we have $\mathbf{G}_{\gamma,\mathrm{der}}=\mathrm{Sp}_6$. As $\mathfrak{sl}_{2,\xi}$-mod, $\mathfrak{g}[1]=12V_1 +7V_2$. Then we have $(\mathrm{Q}_1,\mathrm{Q}_2)=(\mathrm{Q}_1^\tau,\mathrm{Q}_2^\tau)=(1,7)$. Thus, the orbit is quasi-admissible for $n=2$, raisable if $n \neq 2$.

\subsubsection{\dynkin[edge length=.20cm,labels={1,0,0,0,0,1,0},label]{E}{ooooooo} $A_2+A_1$, $\mathrm{L}=\mathrm{Spin}_8 \times \mathrm{SL}_{2}$. } \label{SO2}
\begin{align*}
    \mathfrak{g}[2] &= \mathrm{F} \oplus (\mathrm{Spin}_8 \otimes \mathrm{SL}_{2}) \\
    \mathfrak{g}[1] &= \mathrm{Spin}_8^\prime \oplus (\mathrm{Spin}_8^{\prime\prime} \otimes \mathrm{SL}_{2})
\end{align*}
We have $\mathbf{G}_{\gamma,\mathrm{der}}=\mathrm{Spin}_{6} \hookrightarrow \mathrm{Spin}_8$ by standard representation. As $\mathfrak{sl}_{2,\xi}$-mod, $\mathfrak{g}[1]=12V_1 + 6V_2$. Then we have $(\mathrm{Q}_1,\mathrm{Q}_2)=(\mathrm{Q}_1^\tau,\mathrm{Q}_2^\tau)=(1,6)$. Thus, the orbit is quasi-admissible for $n=1$, raisable if $n \geq 2$.

\subsubsection{\dynkin[edge length=.20cm,labels={0,0,0,1,0,0,0},label]{E}{ooooooo} $A_2+2A_1$, $\mathrm{L}=\mathrm{SL}_2 \times \mathrm{SL}_3 \times \mathrm{SL}_4$. } \label{SO3}
\begin{align*}
    \mathfrak{g}[2] &= \mathrm{SL}_3 \otimes \bigwedge^2\mathrm{SL}_4 \\
    \mathfrak{g}[1] &= \mathrm{SL}_2 \otimes \mathrm{SL}_3^* \otimes \mathrm{SL}_4
\end{align*}
We have $\mathbf{G}_{\gamma,\mathrm{der}}=\mathrm{SL}_2^a \times \mathrm{SL}_2^b \times \mathrm{SL}_2^c$, where $\mathrm{SL}_2^a= \mathrm{SL}_{2,\alpha_2}$, $\mathrm{SL}_2^b \times \mathrm{SL}_2^c \hookrightarrow \mathrm{SL}_4$ by tensoring and $\mathrm{SL}_2^b \hookrightarrow \mathrm{SL}_3$ by $2\Lambda_1$. As $\mathfrak{sl}_{2,\xi_a}$-mod, $\mathfrak{g}[1]=12V_2$. As $\mathfrak{sl}_{2,\xi_b}$-mod, $\mathfrak{g}[1]=4V_2+4V_4$. As $\mathfrak{sl}_{2,\xi_c}$-mod, $\mathfrak{g}[1]=12V_2$. Then we have $(\mathrm{Q}_1^a,\mathrm{Q}_2^a)=(\mathrm{Q}_1^\tau,\mathrm{Q}_2^\tau)=(1,12),(\mathrm{Q}_1^b,\mathrm{Q}_2^b)=(6,44),(\mathrm{Q}_1^c,\mathrm{Q}_2^c)=(2,12)$. Thus, the orbit is quasi-admissible for $n=1$, raisable if $n \geq 2$.

\subsubsection{\dynkin[edge length=.20cm,labels={2,0,0,0,0,1,0},label]{E}{ooooooo} $A_3$, $\mathrm{L}=\mathrm{SL}_{2} \times \mathrm{Spin}_{8}$. }
\begin{align*}
    \mathfrak{g}[2] &= \mathrm{Spin}_{8} \oplus \mathrm{F} \\
    \mathfrak{g}[1] &= \mathrm{SL}_{2} \otimes \mathrm{Spin}_{8}^\prime
\end{align*}
The same as \ref{SO1}, we have $\mathbf{G}_{\gamma,\mathrm{der}}=\mathrm{SL}_{2,\alpha_7}^a \times \mathrm{Spin}_{7}^b$. As $\mathfrak{sl}_{2,\xi_a}$-mod, $\mathfrak{g}[1]=8V_2$. As $\mathfrak{sl}_{2,\xi_b}$-mod, $\mathfrak{g}[1]=8V_1+4V_2$. Then we have $(\mathrm{Q}_1^a,\mathrm{Q}_2^a)=(\mathrm{Q}_1^\tau,\mathrm{Q}_2^\tau)=(1,8),(\mathrm{Q}_1^b,\mathrm{Q}_2^b)=(1,4)$. Thus, the orbit is quasi-admissible for $n=1$, raisable if $n \geq 2$.

\subsubsection{\dynkin[edge length=.20cm,labels={0,0,0,0,0,2,0},label]{E}{ooooooo} $2A_2$, $\mathrm{L}=\mathrm{SL}_{2} \times \mathrm{Spin}_{10}$. } \label{Spin102}
\begin{align*}
    \mathfrak{g}[2] &=  \mathrm{SL}_{2} \otimes \mathrm{Spin}_{10}^\prime
\end{align*}
We have $\mathbf{G}_{\gamma,\mathrm{der}}=\mathrm{SL}_{2}^a \times \mathrm{G}_2^b$, where $\mathrm{SL}_2^a \hookrightarrow \mathrm{SL}_{2} \times \mathrm{Spin}_{10}$ by $\Lambda_1 \otimes 2\Lambda_1$, $\mathrm{G}_2^b \hookrightarrow \mathrm{Spin}_{10}$ 
 by standard representation. Then we have $(\mathrm{Q}_1^a,\mathrm{Q}_2^a)=(3,0),(\mathrm{Q}_1^b,\mathrm{Q}_2^b)=(\mathrm{Q}_1^\tau,\mathrm{Q}_2^\tau)=(1,0)$. Thus, the orbit is quasi-admissible for $n=1$, raisable if $n \geq 2$.

\subsubsection{\dynkin[edge length=.20cm,labels={0,2,0,0,0,0,0},label]{E}{ooooooo} $A_2+3A_1$, $\mathrm{L}=\mathrm{SL}_7$. } \label{A63}
\begin{align*}
    \mathfrak{g}[2] &= \bigwedge^3\mathrm{SL}_7
\end{align*}
We have $\mathbf{G}_{\gamma,\mathrm{der}}=\mathrm{G}_2 \hookrightarrow \mathrm{SL}_7$ by standard representation. Take $\tau$ to be associated with the long root of $\mathfrak{g}_\gamma$, the conditions $\text{(C1)-(C3)}$ are satisfied. Then we have $(\mathrm{Q}_1,\mathrm{Q}_2)=(\mathrm{Q}_1^\tau,\mathrm{Q}_2^\tau)=(\mathrm{Q}_{\mathrm{E}_7}(\alpha_3^\vee-\alpha_6^\vee),0)=(2,0)$. Thus, the orbit is quasi-admissible for $n=1,2$, raisable if $n \geq 3$.

\subsubsection{\dynkin[edge length=.20cm,labels={2,0,0,0,0,0,2},label]{E}{ooooooo} $(A_3+A_1)^{\prime \prime}$, $\mathrm{L}=\mathrm{Spin}_{10}$. } \label{SpinSO10}
\begin{align*}
    \mathfrak{g}[2] &= \mathrm{Spin}_{10} \oplus \mathrm{Spin}_{10}^\prime
\end{align*}
Similarly to \ref{SpinSO8}, we obtain the prehomogeneous space $(\mathrm{Spin}_9 \times \mathrm{GL}_1,\; \mathrm{Spin}_9)$. Then we have $\mathbf{G}_{\gamma,\mathrm{der}}=\mathrm{Spin}_7 \hookrightarrow \mathrm{Spin}_9 \hookrightarrow \mathrm{Spin}_{10}$ by standard representation. Then we have $(\mathrm{Q}_1,\mathrm{Q}_2)=(\mathrm{Q}_1^\tau,\mathrm{Q}_2^\tau)=(1,0)$. Thus, the orbit is quasi-admissible for $n=1$, raisable if $n \geq 2$.

\subsubsection{\dynkin[edge length=.20cm,labels={0,0,1,0,0,1,0},label]{E}{ooooooo} $2A_2 + A_1$, $\mathrm{L}=\mathrm{SL}_{2}^1 \times \mathrm{SL}_{4} \times \mathrm{SL}_{2}^2$. } \label{SO0}
\begin{align*}
    \mathfrak{g}[2] &= \mathrm{F} \oplus (\mathrm{SL}_{2}^1 \otimes \mathrm{SL}_{4} \otimes \mathrm{SL}_{2}^2) \\
    \mathfrak{g}[1] &= (\mathrm{SL}_{4}^* \otimes \mathrm{SL}_{2}^2) \oplus (\mathrm{SL}_2^1 \otimes \bigwedge^2\mathrm{SL}_4)
\end{align*}
We have $\mathbf{G}_{\gamma,\mathrm{der}}=\mathrm{SL}_{2}^a \times \mathrm{SL}_{2}^b \hookrightarrow \mathrm{SL}_{2}^1 \times \mathrm{SL}_{4} \times \mathrm{SL}_{2}^2$. Where $\mathrm{SL}_{2}^a \times \mathrm{SL}_{2}^b \hookrightarrow \mathrm{SL}_{2}^1 \times \mathrm{SL}_{2}^2$ by identity, $\mathrm{SL}_{2}^a \times \mathrm{SL}_{2}^b \hookrightarrow \mathrm{SL}_4$ by tensoring. As $\mathfrak{sl}_{2,\xi_a}$-mod, $\mathfrak{g}[1]=8V_2+V_4$. As $\mathfrak{sl}_{2,\xi_b}$-mod, $\mathfrak{g}[1]=6V_1+4V_3$. Then we have $(\mathrm{Q}_1^a,\mathrm{Q}_2^a)=(\mathrm{Q}_{\mathrm{E_7}}(\alpha_1^\vee + \alpha_2^\vee + \alpha_5^\vee),18)=(3,18),(\mathrm{Q}_1^b,\mathrm{Q}_2^b)=(3,16)$. Thus, the orbit is quasi-admissible for $n=1,3$, the method for raisablility does not work.

\subsubsection{\dynkin[edge length=.20cm,labels={1,0,0,1,0,0,0},label]{E}{ooooooo}$(A_3 + A_1)^\prime$, $\mathrm{L}=\mathrm{SL}_{2}^1 \times \mathrm{SL}_{2}^2 \times \mathrm{SL}_{4}$. } \label{3214}
\begin{align*}
    \mathfrak{g}[2] &= (\mathrm{SL}_{2}^1 \otimes \mathrm{SL}_{4}) \oplus \bigwedge^2\mathrm{SL}_4 \\
    \mathfrak{g}[1] &= \mathrm{SL}_{2}^2 \oplus (\mathrm{SL}_{2}^1 \otimes  \mathrm{SL}_{2}^2 \otimes \mathrm{SL}_{4})
\end{align*}
Let $V_a=\bigwedge^2\mathrm{SL}_4$, we obtain prehomogeneous space $(\mathbf{G}_{V_a}, V_b)=(\mathrm{Sp}_4 \times \mathrm{GL}_2^1,\; \mathrm{Sp}_4 \otimes \mathrm{GL}_2^1)$, which has been computed in \cite{SK77}. By Proposition~\ref{prop:restrict}, we have $\mathbf{G}_{\gamma,\mathrm{der}}=\mathrm{SL}_{2,\alpha_3}^a \times \Delta(\mathrm{SL}_{2})^b \times \mathrm{SL}_{2}^c$, where $\mathrm{SL}_{2,\alpha_3}^a = \mathrm{SL}_{2}^2$, $\Delta(\mathrm{SL}_{2})^b \hookrightarrow \mathrm{SL}_{2}^1 \times \mathrm{SL}_{4}$ by $\Lambda_1 \otimes \Lambda_1$, $\mathrm{SL}_{2}^c \hookrightarrow \mathrm{SL}_{4}$ by $\Lambda_1$. As $\mathfrak{sl}_{2,\xi_a}$-mod, $\mathfrak{g}[1]=9V_2$. As $\mathfrak{sl}_{2,\xi_b}$-mod, $\mathfrak{g}[1]=4V_1+4V_2+2V_3$. As $\mathfrak{sl}_{2,\xi_c}$-mod, $\mathfrak{g}[1]=10V_1+4V_2$. Then we have $(\mathrm{Q}_1^a,\mathrm{Q}_2^a)=(\mathrm{Q}_1^{\tau_1},\mathrm{Q}_2^{\tau_1})=(1,9),(\mathrm{Q}_1^b,\mathrm{Q}_2^b)=(2,12),(\mathrm{Q}_1^c,\mathrm{Q}_2^c)=(\mathrm{Q}_1^{\tau_2},\mathrm{Q}_2^{\tau_2})=(1,4)$. Thus, the orbit is not quasi-admissible, raisable for any $n$.

\subsubsection{\dynkin[edge length=.20cm,labels={0,0,2,0,0,0,0},label]{E}{ooooooo} $D_4(a_1)$, $\mathrm{L}=\mathrm{SL}_2 \times \mathrm{SL}_6$. } \label{A52S2}
\begin{align*}
    \mathfrak{g}[2] &= \mathrm{SL}_2 \otimes \bigwedge^2 \mathrm{SL}_6
\end{align*}
We have $\mathbf{G}_{\gamma,\mathrm{der}}=\mathrm{SL}_2^a \times \mathrm{SL}_2^b \times \mathrm{SL}_2^c \hookrightarrow \mathrm{SL}_6$ by $\Lambda_1 \oplus \Lambda_1 \oplus\Lambda_1$. Then we have $(\mathrm{Q}_1,\mathrm{Q}_2)=(\mathrm{Q}_1^\tau,\mathrm{Q}_2^\tau)=(1,0),(1,0),(1,0)$. Thus, the orbit is quasi-admissible for $n=1$, raisable if $n \geq 2$.

\subsubsection{\dynkin[edge length=.20cm,labels={1,0,0,0,1,0,1},label]{E}{ooooooo} $A_3 + 2A_1$, $\mathrm{L}=\mathrm{SL}_{4} \times \mathrm{SL}_{2}$. }
\begin{align*}
    \mathfrak{g}[2] &= \mathrm{F} \oplus \bigwedge^2 \mathrm{SL}_4  \oplus (\mathrm{SL}_{4} \otimes \mathrm{SL}_{2}) \\
    \mathfrak{g}[1] &= \mathrm{SL}_{4}^* \oplus \mathrm{SL}_{2} \oplus (\bigwedge^2 \mathrm{SL}_{4} \otimes \mathrm{SL}_{2})
\end{align*}
The same as \ref{3214}, we have $\mathbf{G}_{\gamma,\mathrm{der}}=\Delta(\mathrm{SL}_{2})^a \times \mathrm{SL}_{2}^b$. As $\mathfrak{sl}_{2,\xi_a}$-mod, $\mathfrak{g}[1]=4V_1+4V_2+2V_3$. As $\mathfrak{sl}_{2,\xi_b}$-mod, $\mathfrak{g}[1]=8V_1+5V_2$. Then we have $(\mathrm{Q}_1^a,\mathrm{Q}_2^a)=(2,12),(\mathrm{Q}_1^b,\mathrm{Q}_2^b)=(\mathrm{Q}_1^{\tau},\mathrm{Q}_2^{\tau})=(1,5)$. Thus, the orbit is quasi-admissible for $n=2$, raisable if $n \neq 2$.

\subsubsection{\dynkin[edge length=.20cm,labels={2,0,2,0,0,0,0},label]{E}{ooooooo} $D_4$, $\mathrm{L}=\mathrm{SL}_6$. }
\begin{align*}
    \mathfrak{g}[2] &= \mathrm{F} \oplus \bigwedge^2 \mathrm{SL}_6 
\end{align*}
The same as \ref{A52}, we have $\mathbf{G}_{\gamma,\mathrm{der}}=\mathrm{Sp}_6$. Then we have $(\mathrm{Q}_1,\mathrm{Q}_2)=(\mathrm{Q}_1^\tau,\mathrm{Q}_2^\tau)=(1,0)$. Thus, the orbit is quasi-admissible for $n=1$, raisable if $n \geq 2$.

\subsubsection{\dynkin[edge length=.20cm,labels={0,1,1,0,0,0,1},label]{E}{ooooooo} $D_4(a_1)+A_1$, $\mathrm{L}=\mathrm{SL}_{2} \times \mathrm{SL}_{4}$. } \label{3218}
\begin{align*}
    \mathfrak{g}[2] &= \mathrm{F} \oplus \mathrm{SL}_{2} \oplus (\mathrm{SL}_{2} \otimes \bigwedge^2\mathrm{SL}_{4})\\
    \mathfrak{g}[1] &= \mathrm{SL}_4 \oplus \mathrm{SL}_4^* \oplus (\mathrm{SL}_{2} \otimes \mathrm{SL}_{4})
\end{align*}
Let \(V_a = \mathrm{GL}_{2} \otimes \bigwedge^2\mathrm{SL}_{4}\) and let \(V_b\) be the remaining summand in \(\mathfrak{g}[2]\). Then the generic isotropy subgroup for \(V_a\) is \((\mathbf{G}_0)_{V_a} = \mathrm{SL}_{2}^1 \times \mathrm{GL}_{2}^2 \hookrightarrow \mathrm{SL}_4\) by tensoring. Note that $V_b$ is trivial as $(\mathbf{G}_0)_{V_a}$-mod. By Proposition~\ref{prop:restrict}, we have $\mathbf{G}_{\gamma,\mathrm{der}}=\mathrm{SL}_{2}^1 \times \mathrm{SL}_{2}^2$. Two $\mathrm{SL}_2$'s are equal. As $\mathfrak{sl}_{2,\xi}$-mod, $\mathfrak{g}[1]=8V_2$. Then we have $(\mathrm{Q}_1,\mathrm{Q}_2)=(\mathrm{Q}_1^\tau,\mathrm{Q}_2^\tau)=(1,8)$. Thus, the orbit is quasi-admissible for $n=1$, raisable if $n \geq 2$.

\subsubsection{\dynkin[edge length=.20cm,labels={0,0,0,1,0,1,0},label]{E}{ooooooo} $A_3+A_2$, $\mathrm{L}=\mathrm{SL}_{3} \times \mathrm{SL}_{2}^1 \times \mathrm{SL}_{2}^2 \times \mathrm{SL}_2^3$. } \label{3219}
\begin{align*}
    \mathfrak{g}[2] &= \mathrm{SL}_3^* \oplus (\mathrm{SL}_3 \otimes \mathrm{SL}_2^1 \otimes \mathrm{SL}_2^3) \\
    \mathfrak{g}[1] &= (\mathrm{SL}_2^1 \otimes \mathrm{SL}_2^2) \oplus (\mathrm{SL}_3 \otimes \mathrm{SL}_2^2 \otimes \mathrm{SL}_2^3)
\end{align*}
By computing dimension, we have $\mathbf{G}_{\gamma,\mathrm{der}}=\mathrm{SL}_2^2=\mathrm{SL}_{2,\alpha_5}$. As $\mathfrak{sl}_{2,\xi}$-mod, $\mathfrak{g}[1]=8V_2$. Then we have $(\mathrm{Q}_1,\mathrm{Q}_2)=(\mathrm{Q}_1^\tau,\mathrm{Q}_2^\tau)=(1,8)$. Thus, the orbit is quasi-admissible for $n=1$, raisable if $n \geq 2$.

\subsubsection{\dynkin[edge length=.20cm,labels={2,0,0,0,0,2,0},label]{E}{ooooooo} $A_4$, $\mathrm{L}=\mathrm{Spin}_{8} \times \mathrm{SL}_{2}$. } \label{Spin8SOA1}
\begin{align*}
    \mathfrak{g}[2] &= \mathrm{Spin}_8^\prime \oplus (\mathrm{SL}_2 \otimes \mathrm{Spin}_8)
\end{align*}
Let \(V_a = \mathrm{SL}_2 \otimes \mathrm{Spin}_8\) and let \(V_b\) be the remaining summand in \(\mathfrak{g}[2]\). Then the generic isotropy subgroup for \(V_a\) is \((\mathbf{G}_0)_{V_a} = \mathrm{SL}_{4} \times \mathrm{GL}_1\), where \(\mathrm{SL}_{4} \hookrightarrow \mathrm{Spin}_8\) by $\Lambda_1 \oplus \Lambda_1$. The pair \(\bigl((\mathbf{G}_0)_{V_a}, V_b\bigr) = \bigl(\mathrm{SL}_4 \times 2\mathrm{GL}_1,\; \mathrm{SL}_{4} \oplus \mathrm{SL}_{4}^*\bigr)\) is a prehomogeneous space. Applying Proposition~\ref{prop:restrict} and the result of \ref{A3*}, we have $\mathbf{G}_{\gamma,\mathrm{der}}=\mathrm{SL}_{3,\alpha_2,\alpha_4}$. Then we have $(\mathrm{Q}_1,\mathrm{Q}_2)=(1,0)$. Thus, the orbit is quasi-admissible for $n=1$, raisable if $n \geq 2$.

\subsubsection{\dynkin[edge length=.20cm,labels={0,0,0,0,2,0,0},label]{E}{ooooooo} $A_3+A_2+A_1$, $\mathrm{L}=\mathrm{SL}_5 \times \mathrm{SL}_3$. } \label{A42S3}
\begin{align*}
    \mathfrak{g}[2] &= \bigwedge^2\mathrm{SL}_5 \otimes \mathrm{SL}_3
\end{align*}
We have $\mathbf{G}_{\gamma,\mathrm{der}}=\mathrm{SO}_3 \hookrightarrow \mathrm{SL}_5 \times \mathrm{SL}_3$ by $2\Lambda_1 \otimes \Lambda_1$. Then we have $(\mathrm{Q}_1,\mathrm{Q}_2)=(24,0)$. By lemma~\ref{lem:SO}, \(\overline{\mathrm{SO}}_3^{(n,2)}\) has a finite-dimensional \((n,2)\)-genuine representation for $n \mid 24$, the method for raisability does not work.

\subsubsection{\dynkin[edge length=.20cm,labels={2,0,0,0,0,2,2},label]{E}{ooooooo} $(A_5)^{\prime \prime}$, $\mathrm{L}=\mathrm{Spin}_8$. }
\begin{align*}
    \mathfrak{g}[2] &= \mathrm{F} \oplus \mathrm{Spin}_8 \oplus \mathrm{Spin}_8^\prime
\end{align*}
The same as \ref{SpinSO8}, we have $\mathbf{G}_{\gamma,\mathrm{der}}=\mathrm{G}_2$. Then we have $(\mathrm{Q}_1,\mathrm{Q}_2)=(\mathrm{Q}_1^\tau,\mathrm{Q}_2^\tau)=(1,0)$. Thus, the orbit is quasi-admissible for $n=1$, raisable if $n \geq 2$.

\subsubsection{\dynkin[edge length=.20cm,labels={2,1,1,0,0,0,1},label]{E}{ooooooo} $D_4+A_1$, $\mathrm{L}=\mathrm{SL}_{4}$. }
\begin{align*}
    \mathfrak{g}[2] &= \bigwedge^2\mathrm{SL}_4 \oplus 3\mathrm{F} \\
    \mathfrak{g}[1] &= 2\mathrm{SL}_4 \oplus \mathrm{SL}_4^*
\end{align*}
The same as \ref{Spin6}, we have $\mathbf{G}_{\gamma,\mathrm{der}}=\mathrm{Sp}_{4}$. As $\mathfrak{sl}_{2,\xi}$-mod, $\mathfrak{g}[1]=6V_1+3V_2$. Then we have $(\mathrm{Q}_1,\mathrm{Q}_2)=(\mathrm{Q}_1^\tau,\mathrm{Q}_2^\tau)=(1,3)$. Thus, the orbit is quasi-admissible for $n=2$, raisable if $n \neq 2$.

\subsubsection{\dynkin[edge length=.20cm,labels={1,0,0,1,0,1,0},label]{E}{ooooooo} $A_4+A_1$, $\mathrm{L}=\mathrm{SL}_2^1 \times \mathrm{SL}_2^2 \times \mathrm{SL}_2^3 \times \mathrm{SL}_2^4$. }
\begin{align*}
    \mathfrak{g}[2] &= (\mathrm{SL}_2^1 \otimes \mathrm{SL}_2^3) \oplus \mathrm{F} \oplus (\mathrm{SL}_2^1 \otimes \mathrm{SL}_2^2 \otimes \mathrm{SL}_2^4)
\end{align*}
%Similarly as \ref{SO2}, $\mathrm{SL}_2^1 \otimes \mathrm{SO}_4$ gives generic isotropy subgroup $\mathrm{1}$. The same as \ref{dA1}, $\mathrm{SL}_2^1 \otimes \mathrm{SL}_2^2$ gives generic isotropy subgroup $\Delta(\mathrm{SL}_2)$. Thus, we have $\mathbf{G}_{\gamma,\mathrm{der}}=1$.
By computing dimension, we have $\text{dim}\mathbf{G}_{\gamma} = \text{dim}\mathbf{G}_0 - \text{dim}\mathfrak{g}[2] = 15 -13=2$. Thus, we have $\mathbf{G}_{\gamma,\mathrm{der}}=1$.

\subsubsection{\dynkin[edge length=.20cm,labels={2,0,0,1,0,1,0},label]{E}{ooooooo} $D_5(a_1)$, $\mathrm{L}=\mathrm{SL}_2^1 \times \mathrm{SL}_2^2 \times \mathrm{SL}_2^3 \times \mathrm{SL}_2^4$. }
\begin{align*}
    \mathfrak{g}[2] &= \mathrm{SL}_2^2 \oplus \mathrm{F} \oplus (\mathrm{SL}_2^1 \otimes \mathrm{SL}_2^2 \otimes \mathrm{SL}_2^4) \\
    \mathfrak{g}[1] &= (\mathrm{SL}_2^1 \otimes \mathrm{SL}_2^2 \otimes \mathrm{SL}_2^3) \oplus (\mathrm{SL}_2^3 \otimes \mathrm{SL}_2^4)
\end{align*}
By computing dimension, $\text{dim}\mathbf{G}_\gamma-\text{dim}\mathrm{SL}_2^3=12-11=1$. Thus, we have $\mathbf{G}_{\gamma,\mathrm{der}}=\mathrm{SL}_2^3=\mathrm{SL}_{2,\alpha_5}$. As $\mathfrak{sl}_{2,\xi}$-mod, $\mathfrak{g}[1]=6V_2$. Then we have $(\mathrm{Q}_1,\mathrm{Q}_2)=(\mathrm{Q}_1^\tau,\mathrm{Q}_2^\tau)=(1,6)$. Thus, the orbit is quasi-admissible for $n=1$, raisable if $n \geq 2$.

\subsubsection{\dynkin[edge length=.20cm,labels={0,0,0,2,0,0,0},label]{E}{ooooooo} $A_4+A_2$, $\mathrm{L}=\mathrm{SL}_2 \times \mathrm{SL}_3 \times \mathrm{SL}_4$. } \label{234}
\begin{align*}
    \mathfrak{g}[2] &= \mathrm{SL}_2 \otimes \mathrm{SL}_3 \otimes \mathrm{SL}_4
\end{align*}
By \cite[\S 2~Proposition~7]{SK77} and \ref{SO31}, we can deduce that $\mathbf{G}_{\gamma,\mathrm{der}} \cong \mathrm{SL}_2$. Check the embedding methods case by case, we obtain $\mathbf{G}_{\gamma,\mathrm{der}}=\mathrm{SL}_2 \hookrightarrow \mathrm{SL}_2 \times \mathrm{SL}_3 \times \mathrm{SL}_4$ by $\Lambda_1 \otimes 2\Lambda_1 \otimes 3\Lambda_1$. Then we have $(\mathrm{Q}_1,\mathrm{Q}_2)=(15,0)$. Thus, the orbit is quasi-admissible for $n \mid 15$, the method for raisability does not work.

\subsubsection{\dynkin[edge length=.20cm,labels={1,0,0,1,0,2,0},label]{E}{ooooooo} $(A_5)^\prime$, $\mathrm{L}=\mathrm{SL}_2^1 \times \mathrm{SL}_2^2 \times \mathrm{SL}_{2}^3 \times \mathrm{SL}_2^4$. }
\begin{align*}
    \mathfrak{g}[2] &= (\mathrm{SL}_2^1 \otimes \mathrm{SL}_2^3) \oplus(\mathrm{SL}_2^3 \otimes \mathrm{SL}_2^4) \oplus \mathrm{F} \\
    \mathfrak{g}[1] &= \mathrm{SL}_2^2 \oplus (\mathrm{SL}_2^1 \otimes \mathrm{SL}_2^2 \otimes \mathrm{SL}_2^3)
\end{align*}
The same as \ref{dA1}, we have $\mathbf{G}_{\gamma,\mathrm{der}}=\mathrm{SL}_{2}^2\times \Delta(\mathrm{SL}_2)=\mathrm{SL}_{2,\alpha_3}^a \times \Delta(\mathrm{SL}_2)^b$. As $\mathfrak{sl}_{2,\xi_a}$-mod, $\mathfrak{g}[1]=5V_2$. As $\mathfrak{sl}_{2,\xi_b}$-mod, $\mathfrak{g}[1]=4V_1+2V_3$. Then we have $(\mathrm{Q}_1^a,\mathrm{Q}_2^a)=(\mathrm{Q}_1^\tau,\mathrm{Q}_2^\tau)=(1, 5),(\mathrm{Q}_1^b,\mathrm{Q}_2^b)=(3,8)$. Thus, the orbit is not quasi-admissible for any $n$, raisable if $n \neq 2$.

\subsubsection{\dynkin[edge length=.20cm,labels={1,0,0,1,0,1,2},label]{E}{ooooooo} $A_5+A_1$, $\mathrm{L}=\mathrm{SL}_{2}^1 \times \mathrm{SL}_2^2 \times \mathrm{SL}_2^3$. }
\begin{align*}
    \mathfrak{g}[2] &= 2\mathrm{F} \oplus (\mathrm{SL}_{2}^1 \otimes \mathrm{SL}_2^2) \oplus (\mathrm{SL}_{2}^1 \otimes \mathrm{SL}_2^3) \\
    \mathfrak{g}[1] &= \mathrm{SL}_{2}^2 \oplus \mathrm{SL}_2^3 \oplus (\mathrm{SL}_{2}^1 \otimes \mathrm{SL}_2^2 \otimes \mathrm{SL}_2^3)
\end{align*}
The same as \ref{dA1},we have $\mathbf{G}_{\gamma,\mathrm{der}}=\Delta(\mathrm{SL}_{2})$. As $\mathfrak{sl}_{2,\xi}$-mod, $\mathfrak{g}[1]=4V_2 +V_4$. Then we have $(\mathrm{Q}_1,\mathrm{Q}_2)=(3,14)$. Thus, the orbit is quasi-admissible for $n=1,3$, the method for raisability does not work.

\subsubsection{\dynkin[edge length=.20cm,labels={2,0,0,0,2,0,0},label]{E}{ooooooo} $D_5(a_1)+A_1$, $\mathrm{L}=\mathrm{SL}_4 \times\mathrm{SL}_3$. } \label{3329}
\begin{align*}
    \mathfrak{g}[2] &= \mathrm{SL}_4 \oplus (\mathrm{SL}_3 \otimes \bigwedge^2\mathrm{SL}_4)
\end{align*}
The generic isotropy subgroup with respect to $\mathrm{SL}_3 \otimes \bigwedge^2\mathrm{SL}_4$ is $\mathrm{SL}_2^1 \times \mathrm{SL}_2^2 \hookrightarrow \mathrm{SL}_4 \times \mathrm{SL}_3$, where $\mathrm{SL}_2^1 \times \mathrm{SL}_2^2 \hookrightarrow \mathrm{SL}_4$ by tensoring and $\mathrm{SL}_2^1 \hookrightarrow \mathrm{SL}_3$ by $2\Lambda_1$. It gives prehomogeneous space $(\mathrm{SL}_2^1 \times \mathrm{GL}_2^2,\; \mathrm{SL}_2 \otimes \mathrm{GL}_2)$ on the other summand of $\mathfrak{g}[2]$. By Proposition~\ref{prop:restrict}, we have $\mathbf{G}_{\gamma,\mathrm{der}}=\mathrm{SL}_2 \hookrightarrow \mathrm{SL}_4 \times \mathrm{SL}_3$ by $2\Lambda_1 \otimes 2\Lambda_1$. Then we have $(\mathrm{Q}_1,\mathrm{Q}_2)=(8,0)$. Thus, the orbit is quasi-admissible for $n|8$, the method for raisability does not work since the unique choice of $\tau$ does not satisfy the condition $(\text{C}1)$.

\subsubsection{\dynkin[edge length=.20cm,labels={0,1,1,0,1,0,2},label]{E}{ooooooo} $D_6(a_2)$, $\mathrm{L}=\mathrm{SL}_{2}^1 \times \mathrm{SL}_2^2 \times \mathrm{SL}_2^3$. }
\begin{align*}
    \mathfrak{g}[2] &= \mathrm{SL}_{2}^1\oplus 2\mathrm{SL}_2^3 \oplus (\mathrm{SL}_{2}^1 \otimes \mathrm{SL}_2^3) \\
    \mathfrak{g}[1] &= \mathrm{SL}_{2}^2 \oplus (\mathrm{SL}_2^1 \otimes \mathrm{SL}_2^2) \oplus (\mathrm{SL}_2^2 \otimes \mathrm{SL}_2^3)
\end{align*}
By computing dimension, we have $\mathbf{G}_{\gamma,\mathrm{der}}=\mathrm{SL}_{2}^2=\mathrm{SL}_{2,\alpha_4}$. As $\mathfrak{sl}_{2,\xi}$-mod, $\mathfrak{g}[1]=5V_2$. Then we have $(\mathrm{Q}_1,\mathrm{Q}_2)=(\mathrm{Q}_1^\tau,\mathrm{Q}_2^\tau)=(1,5)$. Thus, the orbit is quasi-admissible for $n=2$, raisable if $n \neq 2$.

\subsubsection{\dynkin[edge length=.20cm,labels={0,0,2,0,0,2,0},label]{E}{ooooooo}$E_6(a_3)$, $\mathrm{L}=\mathrm{SL}_{2}^1 \times \mathrm{SL}_2^2 \times \mathrm{SL}_4$.} \label{3231}
\begin{align*}
    \mathfrak{g}[2] &= (\mathrm{SL}_4 \otimes \mathrm{SL}_2^2) \oplus (\bigwedge^2 \mathrm{SL}_4 \otimes \mathrm{SL}_2^1)
\end{align*}
The generic isotropy subgroup with respect to $\bigwedge^2 \mathrm{SL}_4 \otimes \mathrm{GL}_2^1$ in $\mathrm{SL}_4 \times \mathrm{GL}_2^1$ is $\mathrm{GL}_2^3 \times \mathrm{SL}_2^4 \hookrightarrow \mathrm{SL}_4$ by tensoring. It gives prehomogeneous space $(\mathrm{GL}_2^3 \times \mathrm{GL}_2^4 \times \mathrm{SL}_2^2,\; \mathrm{GL}_2^3 \otimes \mathrm{GL}_2^4 \otimes \mathrm{SL}_2^2)$. By Proposition~\ref{prop:restrict}, we have $\mathbf{G}_{\gamma,\mathrm{der}}=\mathrm{1}$.

\subsubsection{\dynkin[edge length=.20cm,labels={2,0,2,0,0,2,0},label]{E}{ooooooo} $D_5$, $\mathrm{L}=\mathrm{SL}_{2} \times \mathrm{SL}_4$. }
\begin{align*}
    \mathfrak{g}[2] &= \mathrm{F} \oplus \bigwedge^2\mathrm{SL}_4 \oplus (\mathrm{SL}_2 \otimes \mathrm{SL}_4)
\end{align*}
The same as \ref{3214}, we have $\mathbf{G}_{\gamma,\mathrm{der}}=\Delta(\mathrm{SL}_{2})^a \times \mathrm{SL}_2^b$. Then we have $(\mathrm{Q}_1^a,\mathrm{Q}_2^a)=(2,0),(\mathrm{Q}_1^b,\mathrm{Q}_2^b)=(\mathrm{Q}_1^\tau,\mathrm{Q}_2^\tau)=(1,0)$. Thus, the orbit is quasi-admissible for $n=1$, raisable if $n \geq 2$.

\begin{table}[H]
\centering
\resizebox{\textwidth}{!}{%  % 缩放表格到页面宽度
\begin{tabular}{cccccc}
\hline
$\mathcal{O}$ & special/even? & $\mathbf{G}_{\gamma,\text{der}}$ & $(Q_1, Q_2)$ & quasi-admissible, if and only if & raisable if \\
\hline
$\{0\}$ & yes/yes & $\mathrm{E}_{7}$ & $(1,0)$ & $n=1$ & $n \geq 2$ \\
$A_1$ & yes/no & $\mathrm{Spin}_{12}$ & $(1,8)$ & $n=1$ & $n \geq 2$ \\
$2A_1$ & yes/no & $\mathrm{SL}_{2,\alpha_7} \times \mathrm{Spin}_9$ & $(1,16),(1,8)$ & $n=1$ & $n \geq 2$ \\
$(3A_1)^{\prime\prime}$ & yes/yes & $\mathrm{F}_{4}$ & $(1,0)$ & $n=1$ & $n \geq 2$ \\
$(3A_1)^\prime$ & no/no & $\mathrm{SL}_{2,\alpha_1} \times \mathrm{Sp}_6$ & $(1,15),(1,8)$ & no such $n$ &  all $n$ \\
$A_2$ & yes/yes & $\mathrm{SL}_{6}$ & $(1,0)$ & $n=1$ & $n \geq 2$ \\
$4A_1$ & no/no & $\mathrm{Sp}_6$ & $(1,7)$ & $n=2$ & $n \neq 2$ \\
$A_2+A_1$ & yes/no & $\mathrm{SL}_{4}$ & $(1,6)$ & $n=1$ & $n \geq 2$ \\
$A_2+2A_1$ & yes/no & $\mathrm{SL}_{2,\alpha_2} \times \mathrm{SL}_2 \times \mathrm{SL}_2$ & $(1,12),(6,44),(2,12)$ & $n=1$ & $n \geq 2$ \\
$A_3$ & yes/no & $\mathrm{SL}_{2,\alpha_7} \times \mathrm{Spin}_{7}$ & $(1,8),(1,4)$ & $n=1$ & $n \geq 2$ \\
$2A_2$ & yes/yes & $\mathrm{SL}_{2} \times \mathrm{G}_2$ & $(3,0),(1,0)$ & $n=1$ & $n \geq 2$ \\
$A_2+3A_1$ & yes/yes & $\mathrm{G}_2$ & $(2,0)$ & $n=1,2$ & $n \geq 3$ \\
$(A_3+A_1)^{\prime \prime}$ & yes/yes & $\mathrm{Spin}_7$ & $(1,0)$ & $n=1$ & $n \geq 2$ \\
$2A_2 + A_1$ & no/no & $\mathrm{SL}_{2} \times \mathrm{SL}_2$ & $(3,18),(3,16)$& $n=1,3$ & n.a.\\
$(A_3 + A_1)^\prime$ & no/no & $\mathrm{SL}_{2,\alpha_3} \times \mathrm{SL}_{2} \times \mathrm{SL}_{2}$ & $(1,9), (2,12), (1,4)$ & no such $n$ & all $n$ \\
$D_4(a_1)$ & yes/yes & $\mathrm{SL}_2 \times \mathrm{SL}_2 \times \mathrm{SL}_2$ & $(1,0),(1,0),(1,0)$ & $n=1$ & $n \geq 2$ \\
$A_3 + 2A_1$ & no/no & $\mathrm{SL}_{2} \times \mathrm{SL}_{2}$ & $(2,12),(1,5)$ & $n=2$ & $n \neq 2$ \\
$D_4$ & yes/yes & $\mathrm{Sp}_6$& $(1,0)$ & $n=1$ &  $n \geq 2$\\
$D_4(a_1)+A_1$ & yes/no & $\mathrm{SL}_{2} \times \mathrm{SL}_2$ & $(1,8),(1,8)$ & $n=1$ & $n \geq 2$ \\
$A_3+A_2$ & yes/no & $\mathrm{SL}_{2,\alpha_5}$ & $(1,8)$ & $n=1$ & $n \geq 2$ \\
$A_4$ & yes/yes & $\mathrm{SL}_{3} $& $(1,0)$ & $n=1$ & $n \geq 2$ \\
$A_3+A_2+A_1$ & yes/yes & $\mathrm{SO}_3$& $(24,0)$ & $n|24$ & n.a. \\
$(A_5)^{\prime \prime}$ & yes/yes & $\mathrm{G}_2$& $(1,0)$ & $n=1$ & $n \geq 2$ \\
$D_4+A_1$ & no/no & $\mathrm{Sp}_{4}$ & $(1,3)$ & $n=2$ & $n \neq 2$ \\
$A_4+A_1$ & yes/no & $\mathrm{1}$ & n.a. & all n & n.a. \\
$D_5(a_1)$ & yes/no & $\mathrm{SL}_{2,\alpha_5}$ & $(1,6)$ & $n=1$ & $n \geq 2$ \\
$A_4+A_2$ & yes/yes & $\mathrm{SL}_2$& $(15,0)$ & $n|15$ & n.a. \\
$(A_5)^\prime$ & no/no & $\mathrm{SL}_{2,\alpha_3} \times \mathrm{SL}_2$ & $(1, 5),(3,8)$ & no such $n$ &  $n \neq 2$\\
$A_5+A_1$ & no/no & $\mathrm{SL}_{2}$ & $(3,14)$ & $n=1,3$ & n.a. \\
$D_5(a_1)+A_1$ & yes/yes & $\mathrm{SL}_2$& $(8,0)$ & $n | 8$ & n.a. \\
$D_6(a_2)$ & no/no & $\mathrm{SL}_{2,\alpha_4}$ & $(1,5)$ & $n=2$ & $n \neq 2$ \\
$E_6(a_3)$ & yes/yes & $\mathrm{1}$& n.a. & all n & n.a. \\
$D_5$ & yes/yes & $\mathrm{SL}_{2} \times \mathrm{SL}_2$ & $(2,0),(1,0)$ & $n=1$ & $n \geq 2$ \\
$E_7(a_5)$ & yes/yes & $\mathrm{1}$& n.a. & all n & n.a. \\
$A_6$ & yes/yes & $\mathrm{SL}_2$& $(7,0)$ & $n|7$ & n.a. \\
$D_5+A_1$ & yes/no & $\mathrm{SL}_{2}$ & $(2,4)$ & $n=1,2$ & $n \geq 3$ \\
$D_6(a_1)$ & yes/no & $\mathrm{SL}_{2,\alpha_4}$ & $(1,4)$ & $n=1$ & $n \geq 2$ \\
$E_7(a_4)$ & yes/yes & $\mathrm{1}$& n.a. & all n & n.a. \\
$D_6$ & no/no & $\mathrm{SL}_{2,\alpha_4}$ & $(1,3)$ & $n=2$ & $n \neq 2$ \\
$E_6(a_1)$ & yes/yes & $\mathrm{1}$& n.a. & all n & n.a. \\
$E_6$ & yes/yes & $\mathrm{SL}_2$& $(3,0)$ & $n=1,3$ & $n \neq 1,3$ \\
$E_7(a_3)$ & yes/yes & $\mathrm{1}$& n.a. & all n & n.a. \\
$E_7(a_2)$ & yes/yes & $\mathrm{1}$& n.a. & all n & n.a. \\
$E_7(a_1)$ & yes/yes & $\mathrm{1}$& n.a. & all n & n.a. \\
$E_7$ & yes/yes & $\mathrm{1}$& n.a. & all n & n.a. \\
\hline
\end{tabular}
}
\caption{Nilpotent orbits for \( \overline{E_7}^{(n)}\)} \label{tab:e7}
\end{table}

\subsubsection{\dynkin[edge length=.20cm,labels={0,0,0,2,0,2,0},label]{E}{ooooooo} $A_6$, $\mathrm{L}=\mathrm{SL}_3 \times \mathrm{SL}_{2}^1 \times \mathrm{SL}_2^2 \times \mathrm{SL}_2^3$. } \label{3.33}
\begin{align*}
    \mathfrak{g}[2] &= (\mathrm{SL}_2^2 \otimes \mathrm{SL}_2^3) \oplus (\mathrm{SL}_3 \otimes \mathrm{SL}_{2}^1 \otimes \mathrm{SL}_2^2)
\end{align*}
Let \(V_a = \mathrm{GL}_{2}^2 \otimes \mathrm{SL}_{2}^3\) and let \(V_b\) be the remaining summand in \(\mathfrak{g}[2]\). Then the generic isotropy subgroup for \(V_a\) is \((\mathbf{G}_0)_{V_a} = \Delta(\mathrm{SL}_{2}) \times \mathrm{GL}_{3} \times \mathrm{SL}_2^1\), where \(\Delta(\mathrm{SL}_{2})\) denotes the diagonal copy inside \(\mathrm{SL}_{2}^2 \times \mathrm{SL}_{2}^3\). The pair \(\bigl((\mathbf{G}_0)_{V_a}, V_b\bigr) = \bigl(\Delta(\mathrm{SL}_{2}) \times \mathrm{GL}_{3} \times \mathrm{SL}_2^1,\; \mathrm{SL}_{2} \otimes \mathrm{GL}_{3} \otimes \mathrm{SL}_2^1\bigr)\) is a prehomogeneous space. Applying Proposition~\ref{prop:restrict}, we have $\mathbf{G}_{\gamma,\mathrm{der}}=\mathrm{SL}_2 \hookrightarrow \mathrm{SL}_3 \times \mathrm{SL}_{2}^1 \times \mathrm{SL}_2^2 \times \mathrm{SL}_2^3$, given by $2\Lambda_1 \otimes \Lambda_1 \otimes \Lambda_1 \otimes \Lambda_1$. Then we have $(\mathrm{Q}_1,\mathrm{Q}_2)=(7,0)$. Thus, the orbit is quasi-admissible for $n \mid 7$, the method for raisability does not work.

\subsubsection{\dynkin[edge length=.20cm,labels={2,1,1,0,1,1,0},label]{E}{ooooooo} $D_5+A_1$, $\mathrm{L}=\mathrm{SL}_{2}^1 \times \mathrm{SL}_{2}^2$. } \label{3.34}
\begin{align*}
    \mathfrak{g}[2] &= 4\mathrm{F} \oplus (\mathrm{SL}_{2}^1 \otimes \mathrm{SL}_{2}^2) \\
    \mathfrak{g}[1] &= 3 \mathrm{SL}_{2}^1 \oplus \mathrm{SL}_{2}^2
\end{align*}
We have $\mathbf{G}_{\gamma,\mathrm{der}}=\Delta(\mathrm{SL}_{2})$, diagonally embedding into $\mathrm{SL}_{2}^1 \times \mathrm{SL}_{2}^2$. As $\mathfrak{sl}_{2,\xi}$-mod, $\mathfrak{g}[1]=4V_2$. Take $\tau$ to be the only choice, the conditions $\text{(C1)-(C3)}$ are satisfied. Then we have $(\mathrm{Q}_1,\mathrm{Q}_2)=(\mathrm{Q}_1^\tau,\mathrm{Q}_2^\tau)=(2,4)$. Thus, the orbit is quasi-admissible for $n=1,2$, raisable if $n \geq 3$.

\subsubsection{\dynkin[edge length=.20cm,labels={2,1,1,0,1,0,2},label]{E}{ooooooo} $D_6(a_1)$, $\mathrm{L}=\mathrm{SL}_{2}^1 \times \mathrm{SL}_{2}^2$. }
\begin{align*}
    \mathfrak{g}[2] &= 3\mathrm{SL}_{2}^2 \oplus 2\mathrm{F} \\
    \mathfrak{g}[1] &= 2\mathrm{SL}_{2}^1 \oplus (\mathrm{SL}_{2}^1 \otimes \mathrm{SL}_{2}^2)
\end{align*}
By computing dimension, we have $\mathbf{G}_{\gamma,\mathrm{der}}=\mathrm{SL}_{2}^1=\mathrm{SL}_{2,\alpha_4}$. As $\mathfrak{sl}_{2,\xi}$-mod, $\mathfrak{g}[1]=4V_2$. Then we have $(\mathrm{Q}_1,\mathrm{Q}_2)=(\mathrm{Q}_1^\tau,\mathrm{Q}_2^\tau)=(1,4)$. Thus, the orbit is quasi-admissible for $n=1$, raisable if $n \geq 2$.

\subsubsection{\dynkin[edge length=.20cm,labels={2,1,1,0,1,2,2},label]{E}{ooooooo} $D_6$, $\mathrm{L}=\mathrm{SL}_{2,\alpha_4}$. }
\begin{align*}
    \mathfrak{g}[2] &= 6\mathrm{F} \\
    \mathfrak{g}[1] &= 3\mathrm{SL}_{2,\alpha_4}
\end{align*}
We have $\mathbf{G}_{\gamma,\mathrm{der}}=\mathrm{SL}_{2,\alpha_4}$. As $\mathfrak{sl}_{2,\xi}$-mod, $\mathfrak{g}[1]=3V_2$. Then we have $(\mathrm{Q}_1,\mathrm{Q}_2)=(\mathrm{Q}_1^\tau,\mathrm{Q}_2^\tau)=(1,3)$. Thus, the orbit is quasi-admissible for $n=2$, raisable if $n \neq 2$.

\subsubsection{\dynkin[edge length=.20cm,labels={2,0,0,2,0,2,0},label]{E}{ooooooo} $E_6(a_1)$, $\mathrm{L}=\mathrm{SL}_{2}^1 \times \mathrm{SL}_2^2 \times \mathrm{SL}_2^3 \times \mathrm{SL}_2^4$. }
\begin{align*}
    \mathfrak{g}[2] &= \mathrm{SL}_2^2 \oplus (\mathrm{SL}_2^3 \otimes \mathrm{SL}_2^4) \oplus (\mathrm{SL}_{2}^1 \otimes \mathrm{SL}_2^2 \otimes \mathrm{SL}_2^3)
\end{align*}
By computing dimension, we have $\mathbf{G}_{\gamma,\mathrm{der}}=1$.

\subsubsection{\dynkin[edge length=.20cm,labels={2,0,2,2,0,2,0},label]{E}{ooooooo} $E_6$, $\mathrm{L}=\mathrm{SL}_{2}^1 \times \mathrm{SL}_2^2 \times \mathrm{SL}_2^3$. }
\begin{align*}
    \mathfrak{g}[2] &= 2\mathrm{F} \oplus (\mathrm{SL}_{2}^1 \otimes \mathrm{SL}_2^2) \oplus (\mathrm{SL}_2^2 \otimes \mathrm{SL}_2^3)
\end{align*}
The same as \ref{dA1}, we have $\mathbf{G}_{\gamma,\mathrm{der}}=\Delta(\mathrm{SL}_2)$. Take $\tau$ to be the only choice, the conditions $\text{(C1)-(C3)}$ are satisfied. Then we have $(\mathrm{Q}_1,\mathrm{Q}_2)=(\mathrm{Q}_1^\tau,\mathrm{Q}_2^\tau)=(3,0)$. Thus, the orbit is quasi-admissible for $n=1,3$, raisable if $n \neq 1,3$.

\subsection{Type $E_8$}
Consider the Dynkin diagram below:

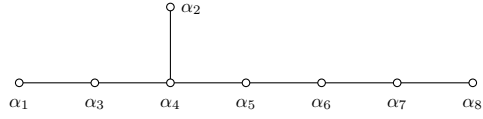
\begin{center}
\newcommand{\al}[1]{\alpha_{#1}}
\dynkin[label,label macro/.code=\al{#1},edge length=1cm]{E}{oooooooo}
\end{center}

\subsubsection{\dynkin[edge length=.20cm,labels={0,0,0,0,0,0,0,1},label]{E}{oooooooo} $A_1$, $\mathrm{L}=\mathrm{E}_{7}$.}
\begin{align*}
    \mathfrak{g}[2] &= \mathrm{F} \\
    \mathfrak{g}[1] &= \mathrm{E}_7
\end{align*}
We have $\mathbf{G}_{\gamma,\mathrm{der}}=\mathrm{E}_{7}$. As $\mathfrak{sl}_{2,\xi}$-mod, $\mathfrak{g}[1]=32V_1+12V_2$. Then we have $(\mathrm{Q}_1,\mathrm{Q}_2)=(\mathrm{Q}_1^\tau,\mathrm{Q}_2^\tau)=(1,12)$. Thus, the orbit is quasi-admissible for $n=1$, raisable if $n \geq 2$.

\subsubsection{\dynkin[edge length=.20cm,labels={1,0,0,0,0,0,0,0},label]{E}{oooooooo} $2A_1$, $\mathrm{L}=\mathrm{Spin}_{14}$.}
\begin{align*}
    \mathfrak{g}[2] &= \mathrm{Spin}_{14} \\
    \mathfrak{g}[1] &= \mathrm{Spin}_{14}^\prime
\end{align*}
We have $\mathbf{G}_{\gamma,\mathrm{der}}=\mathrm{Spin}_{13} \hookrightarrow \mathrm{Spin}_{14}$ by $\Lambda_1$. As $\mathfrak{sl}_{2,\xi}$-mod, $\mathfrak{g}[1]=32V_1+16V_2$. Then we have $(\mathrm{Q}_1,\mathrm{Q}_2)=(\mathrm{Q}_1^\tau,\mathrm{Q}_2^\tau)=(1,16)$. Thus, the orbit is quasi-admissible for $n=1$, raisable if $n \geq 2$.

\subsubsection{\dynkin[edge length=.20cm,labels={0,0,0,0,0,0,1,0},label]{E}{oooooooo} $3A_1$, $\mathrm{L}=\mathrm{SL}_{2} \times \mathrm{E}_6$.}
\begin{align*}
    \mathfrak{g}[2] &= \mathrm{E}_6 \\
    \mathfrak{g}[1] &= \mathrm{SL}_{2} \otimes \mathrm{E}_6^*
\end{align*}
The same as \ref{E6}, we have $\mathbf{G}_{\gamma,\mathrm{der}}=\mathrm{SL}_{2,\alpha_8}^a \times \mathrm{F}_4^b$. As $\mathfrak{sl}_{2,\xi_a}$-mod, $\mathfrak{g}[1]=27V_2$. As $\mathfrak{sl}_{2,\xi_b}$-mod, $\mathfrak{g}[1]=30V_1+12V_2$. Then we have $(\mathrm{Q}_1^a,\mathrm{Q}_2^a)=(\mathrm{Q}_1^{\tau_1},\mathrm{Q}_2^{\tau_1})=(1,27),(\mathrm{Q}_1^b,\mathrm{Q}_2^b)=(\mathrm{Q}_1^{\tau_2},\mathrm{Q}_2^{\tau_2})=(1,12)$. Thus, the orbit is not quasi-admissible for any $n$, raisable for all $n$.

\subsubsection{\dynkin[edge length=.20cm,labels={0,0,0,0,0,0,0,2},label]{E}{oooooooo} $A_2$, $\mathrm{L}=\mathrm{E}_{7}$.} \label{E7}
\begin{align*}
    \mathfrak{g}[2] &= \mathrm{E}_7
\end{align*}
We have $\mathbf{G}_{\gamma,\mathrm{der}}=\mathrm{E}_{6} \hookrightarrow \mathrm{E}_7$ by standard representation. Then we have $(\mathrm{Q}_1,\mathrm{Q}_2)=(\mathrm{Q}_1^\tau,\mathrm{Q}_2^\tau)=(1,0)$. Thus, the orbit is quasi-admissible for $n=1$, raisable if $n \geq 2$.

\subsubsection{\dynkin[edge length=.20cm,labels={0,1,0,0,0,0,0,0},label]{E}{oooooooo} $4A_1$, $\mathrm{L}=\mathrm{SL}_{8}$.}
\begin{align*}
    \mathfrak{g}[2] &= \bigwedge^2 \mathrm{SL}_8 \\
    \mathfrak{g}[1] &= \bigwedge^5 \mathrm{SL}_8
\end{align*}
We have $\mathbf{G}_{\gamma,\mathrm{der}}=\mathrm{Sp}_{8} \hookrightarrow \mathrm{SL}_8$ by $\Lambda_1$. As $\mathfrak{sl}_{2,\xi}$-mod, $\mathfrak{g}[1]=26V_1+15V_2$. Then we have $(\mathrm{Q}_1,\mathrm{Q}_2)=(\mathrm{Q}_1^\tau,\mathrm{Q}_2^\tau)=(1,15)$. Thus, the orbit is quasi-admissible for $n=2$, raisable if $n \neq 2$.

\subsubsection{\dynkin[edge length=.20cm,labels={1,0,0,0,0,0,0,1},label]{E}{oooooooo} $A_2+A_1$, $\mathrm{L}=\mathrm{Spin}_{12}$.}
\begin{align*}
    \mathfrak{g}[2] &= \mathrm{F} \oplus \mathrm{Spin}_{12}^\prime \\
    \mathfrak{g}[1] &= \mathrm{Spin}_{12} \oplus \mathrm{Spin}_{12}^{\prime\prime}
\end{align*}
The same as \ref{Spin12}, we have $\mathbf{G}_{\gamma,\mathrm{der}}=\mathrm{SL}_{6}$. As $\mathfrak{sl}_{2,\xi}$-mod, $\mathfrak{g}[1]=24V_1+10V_2$. Then we have $(\mathrm{Q}_1,\mathrm{Q}_2)=(\mathrm{Q}_1^\tau,\mathrm{Q}_2^\tau)=(1,10)$. Thus, the orbit is quasi-admissible for $n=1$, raisable if $n \geq 2$.

\subsubsection{\dynkin[edge length=.20cm,labels={0,0,0,0,0,1,0,0},label]{E}{oooooooo} $A_2+2A_1$, $\mathrm{L}=\mathrm{Spin}_{10} \times \mathrm{SL}_{3}$.}
\begin{align*}
    \mathfrak{g}[2] &= \mathrm{Spin}_{10} \otimes \mathrm{SL}_{3} \\
    \mathfrak{g}[1] &= \mathrm{Spin}_{10}^\prime \otimes \mathrm{SL}_{3}^*
\end{align*}
We have $\mathbf{G}_{\gamma,\mathrm{der}}=\mathrm{Spin}_{7}^a \times \mathrm{SL}_{2}^b$, where $\mathrm{Spin}_{7}^a \hookrightarrow \mathrm{Spin}_{10}$ by $\Lambda_1$, $\mathrm{SL}_{2}^b \hookrightarrow \mathrm{Spin}_{10} \times \mathrm{SL}_{3}$ by $(\Lambda_1 \oplus \Lambda_1) \otimes 2\Lambda_1$. As $\mathfrak{sl}_{2,\xi_a}$-mod, $\mathfrak{g}[1]=24V_1+12V_2$. As $\mathfrak{sl}_{2,\xi_b}$-mod, $\mathfrak{g}[1]=8V_2+8V_4$. Then we have $(\mathrm{Q}_1^a,\mathrm{Q}_2^a)=(\mathrm{Q}_1^\tau,\mathrm{Q}_2^\tau)=(1,12),(\mathrm{Q}_1^b,\mathrm{Q}_2^b)=(6,88)$. Thus, the orbit is quasi-admissible for $n=1$, raisable if $n \geq 2$.

\subsubsection{\dynkin[edge length=.20cm,labels={1,0,0,0,0,0,0,2},label]{E}{oooooooo} $A_3$, $\mathrm{L}=\mathrm{Spin}_{12}$.}
\begin{align*}
    \mathfrak{g}[2] &= \mathrm{Spin}_{12} \oplus \mathrm{F} \\
    \mathfrak{g}[1] &= \mathrm{Spin}_{12}^\prime
\end{align*}
We have $\mathbf{G}_{\gamma,\mathrm{der}}=\mathrm{Spin}_{11} \hookrightarrow \mathrm{Spin}_{12}$ by $\Lambda_1$. As $\mathfrak{sl}_{2,\xi}$-mod, $\mathfrak{g}[1]=16V_1+8V_1$. Then we have $(\mathrm{Q}_1,\mathrm{Q}_2)=(\mathrm{Q}_1^\tau,\mathrm{Q}_2^\tau)=(1,8)$. Thus, the orbit is quasi-admissible for $n=1$, raisable if $n \geq 2$.

\subsubsection{\dynkin[edge length=.20cm,labels={0,0,1,0,0,0,0,0},label]{E}{oooooooo} $A_2+3A_1$, $\mathrm{L}=\mathrm{SL}_{2} \times \mathrm{SL}_{7}$.}
\begin{align*}
    \mathfrak{g}[2] &= \bigwedge^3\mathrm{SL}_7 \\
    \mathfrak{g}[1] &= \mathrm{SL}_2 \otimes \bigwedge^5 \mathrm{SL}_7
\end{align*}
The same as \ref{A63}, we have $\mathbf{G}_{\gamma,\mathrm{der}}=\mathrm{SL}_{2,\alpha_1}^a \times \mathrm{G}_{2}^b$. As $\mathfrak{sl}_{2,\xi_a}$-mod, $\mathfrak{g}[1]=21V_2$. As $\mathfrak{sl}_{2,\xi_b}$-mod, $\mathfrak{g}[1]=12V_1+12V_2+2V_3$. Then we have $(\mathrm{Q}_1^a,\mathrm{Q}_2^a)=(\mathrm{Q}_1^\tau,\mathrm{Q}_2^\tau)=(1,21),(\mathrm{Q}_1^b,\mathrm{Q}_2^b)=(2,20)$. Thus, the orbit is quasi-admissible for $n=2$, raisable if $n \neq 2$.

\subsubsection{\dynkin[edge length=.20cm,labels={2,0,0,0,0,0,0,0},label]{E}{oooooooo} $2A_2$, $\mathrm{L}=\mathrm{Spin}_{14}$.} \label{Spin14}
\begin{align*}
    \mathfrak{g}[2] &= \mathrm{Spin}_{14}^\prime
\end{align*}
We have $\mathbf{G}_{\gamma,\mathrm{der}}=\mathrm{G}_{2} \times \mathrm{G}_{2} \hookrightarrow \mathrm{Spin}_{14}$ by tensoring. For each $\mathrm{G}_2$ we have $(\mathrm{Q}_1,\mathrm{Q}_2)=(\mathrm{Q}_1^\tau,\mathrm{Q}_2^\tau)=(1,0)$. Thus, the orbit is quasi-admissible for $n=1$, raisable if $n \geq 2$.

\subsubsection{\dynkin[edge length=.20cm,labels={1,0,0,0,0,0,1,0},label]{E}{oooooooo} $2A_2+A_1$, $\mathrm{L}=\mathrm{Spin}_{10} \times \mathrm{SL}_2$.}
\begin{align*}
    \mathfrak{g}[2] &= \mathrm{F} \oplus (\mathrm{Spin}_{10}^\prime \otimes \mathrm{SL}_2) \\
    \mathfrak{g}[1] &= (\mathrm{Spin}_{10} \otimes \mathrm{SL}_2) \oplus \mathrm{Spin}_{10}^{\prime\prime}
\end{align*}
The same as \ref{Spin102}, we have $\mathbf{G}_{\gamma,\mathrm{der}}=\mathrm{SL}_{2}^a \times \mathrm{G}_2^b$. As $\mathfrak{sl}_{2,\xi_a}$-mod, $\mathfrak{g}[1]=16V_2+V_4$. As $\mathfrak{sl}_{2,\xi_b}$-mod, $\mathfrak{g}[1]=20V_1+8V_2$. Then we have $(\mathrm{Q}_1^a,\mathrm{Q}_2^a)=(3,26),(\mathrm{Q}_1^b,\mathrm{Q}_2^b)=(\mathrm{Q}_1^\tau,\mathrm{Q}_2^\tau)=(1,8)$. Thus, the orbit is quasi-admissible for $n=1$, raisable if $n \geq 2$.

\subsubsection{\dynkin[edge length=.20cm,labels={0,0,0,0,0,1,0,1},label]{E}{oooooooo} $A_3+A_1$, $\mathrm{L}=\mathrm{Spin}_{10} \times \mathrm{SL}_2$.}
\begin{align*}
    \mathfrak{g}[2] &= \mathrm{Spin}_{10} \oplus \mathrm{Spin}_{10}^\prime \\
    \mathfrak{g}[1] &= \mathrm{SL}_2 \oplus (\mathrm{Spin}_{10}^\prime \otimes \mathrm{SL}_2)
\end{align*}
The same as \ref{SpinSO10}, we have $\mathbf{G}_{\gamma,\mathrm{der}}=\mathrm{SL}_{2,\alpha_7}^a \times \mathrm{Spin}_7^b$. As $\mathfrak{sl}_{2,\xi_a}$-mod, $\mathfrak{g}[1]=17V_2$. As $\mathfrak{sl}_{2,\xi_b}$-mod, $\mathfrak{g}[1]=18V_1+8V_2$. Then we have $(\mathrm{Q}_1^a,\mathrm{Q}_2^a)=(\mathrm{Q}_1^{\tau_1},\mathrm{Q}_2^{\tau_1})=(1,17),(\mathrm{Q}_1^b,\mathrm{Q}_2^b)=(\mathrm{Q}_1^{\tau_2},\mathrm{Q}_2^{\tau_2})=(1,8)$. Thus, the orbit is not quasi-admissible for any $n$, raisable for all $n$.

\subsubsection{\dynkin[edge length=.20cm,labels={0,0,0,0,0,0,2,0},label]{E}{oooooooo} $D_4(a_1)$, $\mathrm{L}=\mathrm{E}_{6} \times \mathrm{SL}_2$.} \label{E62}
\begin{align*}
    \mathfrak{g}[2] &= \mathrm{E}_{6} \otimes \mathrm{SL}_2 
\end{align*}
We have $\mathbf{G}_{\gamma,\mathrm{der}}=\mathrm{Spin}_{8} \hookrightarrow \mathrm{E}_6$, one of its long root is $\alpha_2$. Then we have $(\mathrm{Q}_1,\mathrm{Q}_2)=(\mathrm{Q}_1^\tau,\mathrm{Q}_2^\tau)=(1,0)$. Thus, the orbit is quasi-admissible for $n=1$, raisable if $n \geq 2$.

\subsubsection{\dynkin[edge length=.20cm,labels={0,0,0,0,0,0,2,2},label]{E}{oooooooo} $D_4$, $\mathrm{L}=\mathrm{E}_{6}$.}
\begin{align*}
    \mathfrak{g}[2] &= \mathrm{F} \oplus \mathrm{E}_6 
\end{align*}
The same as \ref{E6}, we have $\mathbf{G}_{\gamma,\mathrm{der}}=\mathrm{F}_{4}$. Then we have $(\mathrm{Q}_1,\mathrm{Q}_2)=(\mathrm{Q}_1^\tau,\mathrm{Q}_2^\tau)=(1,0)$. Thus, the orbit is quasi-admissible for $n=1$, raisable if $n \geq 2$.

\subsubsection{\dynkin[edge length=.20cm,labels={0,0,0,0,1,0,0,0},label]{E}{oooooooo} $2A_2+2A_1$, $\mathrm{L}=\mathrm{SL}_5 \times \mathrm{SL}_4$.}
\begin{align*}
    \mathfrak{g}[2] &= \mathrm{SL}_5 \otimes \bigwedge^2\mathrm{SL}_4 \\
    \mathfrak{g}[1] &= \bigwedge^3\mathrm{SL}_5 \otimes \mathrm{SL}_4
\end{align*}
We have $\mathbf{G}_{\gamma,\mathrm{der}}=\mathrm{Sp}_4 \hookrightarrow \mathrm{SL}_5 \times \mathrm{SL}_4$ by $\Lambda(5) \otimes \Lambda_1$, where $\Lambda(5)$ denotes the $5$-dimensional irreducible representation of $\mathrm{Sp}_4$ factoring through \(\mathrm{SO}_5\). As $\mathfrak{sl}_{2,\xi}$-mod, $\mathfrak{g}[1]=8V_1+8V_2+4V_3+V_4$. Then we have $(\mathrm{Q}_1,\mathrm{Q}_2)=(3,34)$. Thus, the orbit is quasi-admissible for $n=1,3$, the method for raisability does not work.

\subsubsection{\dynkin[edge length=.20cm,labels={0,0,1,0,0,0,0,1},label]{E}{oooooooo} $A_3+2A_1$, $\mathrm{L}=\mathrm{SL}_{2} \times \mathrm{SL}_{6}$.}
\begin{align*}
    \mathfrak{g}[2] &= (\mathrm{SL}_{2} \otimes \mathrm{SL}_{6}^*) \oplus \bigwedge^2\mathrm{SL}_6\\
    \mathfrak{g}[1] &= \mathrm{SL}_6 \oplus (\mathrm{SL}_2 \otimes \bigwedge^4\mathrm{SL}_6)
\end{align*}
Similarly to \ref{3214}, we obtain prehomogeneous space $(\mathrm{GL}_{2} \times \mathrm{Sp}_{6},\; \mathrm{GL}_{2} \otimes \mathrm{Sp}_{6})$. Thus, we have $\mathbf{G}_{\gamma,\mathrm{der}}=\Delta(\mathrm{SL}_{2})^a \times \mathrm{Sp}_{4}^b$, where $\Delta(\mathrm{SL}_{2})^a \hookrightarrow \mathrm{SL}_{2} \times \mathrm{SL}_{6}$ by $\Lambda_1 \otimes \Lambda_1$, $\mathrm{Sp}_4^b \hookrightarrow \mathrm{SL}_6$ by standard representation. As $\mathfrak{sl}_{2,\xi_a}$-mod, $\mathfrak{g}[1]=8V_1+8V_2+4V_3$. As $\mathfrak{sl}_{2,\xi_b}$-mod, $\mathfrak{g}[1]=18V_1+9V_2$. Then we have $(\mathrm{Q}_1^a,\mathrm{Q}_2^a)=(2,24),(\mathrm{Q}_1^b,\mathrm{Q}_2^b)=(\mathrm{Q}_1^{\tau},\mathrm{Q}_2^{\tau})=(1,9)$. Thus, the orbit is quasi-admissible for $n=2$, raisable if $n \neq 2$.

\subsubsection{\dynkin[edge length=.20cm,labels={0,1,0,0,0,0,1,0},label]{E}{oooooooo} $D_4(a_1)+A_1$, $\mathrm{L}=\mathrm{SL}_{6} \times \mathrm{SL}_2$.}
\begin{align*}
    \mathfrak{g}[2] &= \mathrm{F} \oplus (\bigwedge^2\mathrm{SL}_6 \otimes \mathrm{SL}_2) \\
    \mathfrak{g}[1] &= (\mathrm{SL}_{6}^* \otimes \mathrm{SL}_2) \oplus \bigwedge^3\mathrm{SL}_6
\end{align*}
The same as \ref{A52S2}, we have $\mathbf{G}_{\gamma,\mathrm{der}}=\mathrm{SL}_{2} \times \mathrm{SL}_2 \times \mathrm{SL}_2$. As $\mathfrak{sl}_{2,\xi}$-mod, $\mathfrak{g}[1]=16V_1+8V_2$. Then we have $(\mathrm{Q}_1,\mathrm{Q}_2)=(1,8)$. Thus, the orbit is quasi-admissible for $n=1$, raisable if $n \geq 2$.

\subsubsection{\dynkin[edge length=.20cm,labels={1,0,0,0,0,1,0,0},label]{E}{oooooooo} $A_3+A_2$, $\mathrm{L}=\mathrm{Spin}_{8} \times \mathrm{SL}_3$.}
\begin{align*}
    \mathfrak{g}[2] &= \mathrm{SL}_3^* \oplus (\mathrm{Spin}_8 \otimes \mathrm{SL}_3) \\
    \mathfrak{g}[1] &= \mathrm{Spin}_8^{\prime\prime} \oplus (\mathrm{Spin}_8^\prime \otimes \mathrm{SL}_3)
\end{align*}
The generic isotropy subgroup with respect to $\mathrm{Spin}_8 \otimes \mathrm{SL}_3$ is $\mathrm{Spin}_5 \times \mathrm{SL}_2$, where $\mathrm{Spin}_5 \hookrightarrow \mathrm{Spin}_8$ by $\Lambda_1$ and $\mathrm{SL}_2 \hookrightarrow \mathrm{Spin}_{8} \times \mathrm{SL}_3$ by $2\Lambda_1 \otimes 2\Lambda_1$. Then we get prehomogeneous space $(\mathrm{GL}_2,\; S^2(\mathrm{SL}_2))$. By Proposition~\ref{prop:restrict}, we have $\mathbf{G}_{\gamma,\mathrm{der}}=\mathrm{Spin}_{5} \hookrightarrow \mathrm{Spin}_8$ by standard representation. As $\mathfrak{sl}_{2,\xi}$-mod, $\mathfrak{g}[1]=16V_1 + 8V_2$. Then we have $(\mathrm{Q}_1,\mathrm{Q}_2)=(\mathrm{Q}_1^\tau,\mathrm{Q}_2^\tau)=(1,8)$. Thus, the orbit is quasi-admissible for $n=1$, raisable if $n \geq 2$.

\subsubsection{\dynkin[edge length=.20cm,labels={2,0,0,0,0,0,0,2},label]{E}{oooooooo} $A_4$, $\mathrm{L}=\mathrm{Spin}_{12}$.}\label{SpinSO12}
\begin{align*}
    \mathfrak{g}[2] &= \mathrm{Spin}_{12} \oplus \mathrm{Spin}_{12}^\prime
\end{align*}
Similarly to \ref{SpinSO8}, we obtain prehomogeneous space $(\mathrm{Spin}_{11} \times \mathrm{GL}_1,\; \mathrm{Spin}_{11})$. Thus, we have $\mathbf{G}_{\gamma,\mathrm{der}}=\mathrm{SL}_{5} \hookrightarrow \mathrm{Spin}_{11} \mathrm{Spin}_{12}$ by $\Lambda_1 \oplus \Lambda_1$. Then we have $(\mathrm{Q}_1,\mathrm{Q}_2)=(\mathrm{Q}_1^\tau,\mathrm{Q}_2^\tau)=(1,0)$. Thus, the orbit is quasi-admissible for $n=1$, raisable if $n \geq 2$.

\subsubsection{\dynkin[edge length=.20cm,labels={0,0,0,1,0,0,0,0},label]{E}{oooooooo} $A_3+A_2+A_1$, $\mathrm{L}=\mathrm{SL}_{3} \times \mathrm{SL}_{2} \times \mathrm{SL}_5$.}
\begin{align*}
    \mathfrak{g}[2] &= \mathrm{SL}_3 \otimes \bigwedge^2 \mathrm{SL}_5\\
    \mathfrak{g}[1] &= \mathrm{SL}_3^* \otimes \mathrm{SL}_{2} \otimes \mathrm{SL}_5 
\end{align*}
The same as \ref{A42S3}, we have $\mathbf{G}_{\gamma,\mathrm{der}}=\mathrm{SO}_{3}^a \times \mathrm{SL}_{2,\alpha_2}^b$, where $\mathrm{SO}_3^a \hookrightarrow \mathrm{SL}_{3} \times \mathrm{SL}_{5}$ by $\Lambda_1 \otimes 2\Lambda$. As $\mathfrak{sl}_{2,\xi_a}$-mod, $\mathfrak{g}[1]=2V_3+2V_5+2V_7$. As $\mathfrak{sl}_{2,\xi_b}$-mod, $\mathfrak{g}[1]=15V_2$. Then we have $(\mathrm{Q}_1^a,\mathrm{Q}_2^a)=(24,160),(\mathrm{Q}_1^b,\mathrm{Q}_2^b)=(\mathrm{Q}_1^{\tau},\mathrm{Q}_2^{\tau})=(1,15)$. Thus, the orbit is quasi-admissible for $n=2$, raisable if $n \neq 2$.

\subsubsection{\dynkin[edge length=.20cm,labels={0,1,0,0,0,0,1,2},label]{E}{oooooooo} $D_4+A_1$, $\mathrm{L}=\mathrm{SL}_{6}$.}
\begin{align*}
    \mathfrak{g}[2] &= 2\mathrm{F} \oplus \bigwedge^2 \mathrm{SL}_6 \\
    \mathfrak{g}[1] &= \mathrm{SL}_6^* \oplus \bigwedge^3 \mathrm{SL}_6
\end{align*}
The same as \ref{A52}, we have $\mathbf{G}_{\gamma,\mathrm{der}}=\mathrm{Sp}_{6}$. As $\mathfrak{sl}_{2,\xi}$-mod, $\mathfrak{g}[1]=12V_1 +7V_2$. Then we have $(\mathrm{Q}_1,\mathrm{Q}_2)=(\mathrm{Q}_1^\tau,\mathrm{Q}_2^\tau)=(1,7)$. Thus, the orbit is quasi-admissible for $n=2$, raisable if $n \neq 2$.

\subsubsection{\dynkin[edge length=.20cm,labels={0,2,0,0,0,0,0,0},label]{E}{oooooooo} $D_4(a_1)+A_2$, $\mathrm{L}=\mathrm{SL}_{8}$.} \label{A73}
\begin{align*}
    \mathfrak{g}[2] &= \bigwedge^3\mathrm{SL}_8
\end{align*}
We have $\mathbf{G}_{\gamma,\mathrm{der}}=\mathrm{PGL}_{3} \hookrightarrow \mathrm{SL}_8$ by adjoint representation. Take $\tau$ to be associated with one of the simple roots of $\mathfrak{g}_\gamma$, the conditions $\text{(C1)-(C3)}$ are satisfied. Then we have $(\mathrm{Q}_1,\mathrm{Q}_2)=(\mathrm{Q}_1^\tau,\mathrm{Q}_2^\tau)=(6,0)$. A necessary condition for this orbit to be quasi-admissible is that \(n\) divides \(6\); we claim that this condition is also sufficient. For \(n = 1, 2\) the claim follows directly from Propposition~\ref{prop:quasi}. For \(n = 3, 6\) we can show that \(\overline{\mathrm{PGL}}_3^{(n)}\) admits a finite dimensional \(\mu_n\)-genuine representation by an argument completely analogous to that of Lemma~\ref{lem:SO}, replacing \(F^\times/2\) by \(F^\times/3\) throughout the “if" part of the proof in \cite[Lemma~2.8]{GLT25}. Hence the orbit is quasi‑admissible in these cases as well. Thus, the orbit is quasi-admissible for $n=1,2,3,6$, raisable if  $n \neq 1,2,3,6$.

\subsubsection{\dynkin[edge length=.20cm,labels={1,0,0,0,0,1,0,1},label]{E}{oooooooo} $A_4+A_1$, $\mathrm{L}=\mathrm{Spin}_8 \times \mathrm{SL}_2$.}
\begin{align*}
    \mathfrak{g}[2] &= \mathrm{F} \oplus \mathrm{Spin}_8^\prime \oplus (\mathrm{Spin}_8 \otimes \mathrm{SL}_2) \\
    \mathfrak{g}[1] &= \mathrm{SL}_2 \oplus \mathrm{Spin}_8^{\prime\prime} \oplus (\mathrm{Spin}_8^\prime \otimes \mathrm{SL}_2)
\end{align*}
The same as \ref{Spin8SOA1}, we have $\mathbf{G}_{\gamma,\mathrm{der}}=\mathrm{SL}_{3}$. As $\mathfrak{sl}_{2,\xi}$-mod, $\mathfrak{g}[1]=14V_1+6V_2$. Then we have $(\mathrm{Q}_1,\mathrm{Q}_2)=(\mathrm{Q}_1^\tau,\mathrm{Q}_2^\tau)=(1,6)$. Thus, the orbit is quasi-admissible for $n=1$, raisable if $n \geq 2$.

\subsubsection{\dynkin[edge length=.20cm,labels={1,0,0,0,1,0,0,0},label]{E}{oooooooo} $2A_3$, $\mathrm{L}=\mathrm{SL}_4^1 \times \mathrm{SL}_4^2$.}
\begin{align*}
    \mathfrak{g}[2] &= (\mathrm{SL}_4^1 \otimes \mathrm{SL}_4^2) \oplus \bigwedge^2 \mathrm{SL}_4^2\\
    \mathfrak{g}[1] &= (\mathrm{SL}_4^1)^* \oplus (\bigwedge^2 \mathrm{SL}_4^1 \otimes \mathrm{SL}_4^2)
\end{align*}
Similarly to \ref{3214}, we obtain prehomogeneous space $(\mathrm{GL}_{4}^1 \times \mathrm{Sp}_{4},\; \mathrm{GL}_{4}^1 \otimes \mathrm{Sp}_{4})$. Thus, we have $\mathbf{G}_{\gamma,\mathrm{der}}=\Delta(\mathrm{Sp}_{4}) \hookrightarrow \mathrm{SL}_4^1 \times \mathrm{SL}_4^2$. As $\mathfrak{sl}_{2,\xi}$-mod, $\mathfrak{g}[1]=8V_1 +7V_2+2V_3$. Then we have $(\mathrm{Q}_1,\mathrm{Q}_2)=(2,15)$. Thus, the orbit is quasi-admissible for $n=4$, the method for raisability does not work.

\subsubsection{\dynkin[edge length=.20cm,labels={1,0,0,0,0,1,0,2},label]{E}{oooooooo} $D_5(a_1)$, $\mathrm{L}=\mathrm{Spin}_{8} \times \mathrm{SL}_2$.}
\begin{align*}
    \mathfrak{g}[2] &= \mathrm{SL}_2 \oplus \mathrm{F} \oplus (\mathrm{Spin}_{8} \otimes \mathrm{SL}_2) \\
    \mathfrak{g}[1] &= \mathrm{Spin}_8^{\prime\prime} \oplus (\mathrm{Spin}_8^\prime \otimes \mathrm{SL}_2)
\end{align*}
Let \(V_a = \mathrm{Spin}_{8} \otimes \mathrm{SL}_2\) and let \(V_b\) be the remaining summand in \(\mathfrak{g}[2]\). Note that $V_b$ is trivial as $(\mathbf{G}_0)_{V_a}$-mod. By Proposition~\ref{prop:restrict}, we have $\mathbf{G}_{\gamma,\mathrm{der}}=\mathrm{SL}_{4} \hookrightarrow \mathrm{Spin}_8$ by $\Lambda_1 \oplus \Lambda_1$. As $\mathfrak{sl}_{2,\xi}$-mod, $\mathfrak{g}[1]=12V_1+6V_2$. Then we have $(\mathrm{Q}_1,\mathrm{Q}_2)=(\mathrm{Q}_1^\tau,\mathrm{Q}_2^\tau)=(1,6)$. Thus, the orbit is quasi-admissible for $n=1$, raisable if $n \geq 2$.

\subsubsection{\dynkin[edge length=.20cm,labels={0,0,0,1,0,0,0,1},label]{E}{oooooooo} $A_4+2A_1$, $\mathrm{L}=\mathrm{SL}_3 \times \mathrm{SL}_2 \times \mathrm{SL}_4$.}
\begin{align*}
    \mathfrak{g}[2] &= (\mathrm{SL}_3 \otimes \mathrm{SL}_2) \oplus (\mathrm{SL}_3^* \otimes \bigwedge^2\mathrm{SL}_4)\\
    \mathfrak{g}[1] &= \mathrm{SL}_4^* \oplus (\mathrm{SL}_3 \otimes \mathrm{SL}_2 \otimes \mathrm{SL}_4)
\end{align*}
The generic isotropy subgroup with respect to $\mathrm{SL}_3^* \otimes \bigwedge^2\mathrm{SL}_4$ is $\mathrm{SL}_2^1 \times \mathrm{SL}_2^2$, where $\mathrm{SL}_2^1 \times \mathrm{SL}_2^2 \hookrightarrow \mathrm{SL}_4$ by tensoring and $\mathrm{SL}_2^1 \hookrightarrow \mathrm{SL}_3$ by $2\Lambda_1$. Then we get prehomogeneous space $(\mathrm{SL}_2^1 \times \mathrm{SL}_2^2 \times \mathrm{GL}_2,\; S^2(\mathrm{SL}_2) \otimes \mathrm{F} \otimes \mathrm{GL}_2)$, whose generic isotropy subgroup is $\mathrm{SL}_{2}^2$. By Proposition~\ref{prop:restrict}, We have $\mathbf{G}_{\gamma,\mathrm{der}}=\mathrm{SL}_{2}^2$. As $\mathfrak{sl}_{2,\xi}$-mod, $\mathfrak{g}[1]=14V_2$. Take $\tau$ to be the only choice, the conditions $\text{(C1)-(C3)}$ are satisfied. Then we have $(\mathrm{Q}_1,\mathrm{Q}_2)=(\mathrm{Q}_1^\tau,\mathrm{Q}_2^\tau)=(2,14)$. Thus, the orbit is quasi-admissible for $n=1,2$, raisable if $n \geq 3$.

\subsubsection{\dynkin[edge length=.20cm,labels={0,0,0,0,0,2,0,0},label]{E}{oooooooo} $A_4+A_2$, $\mathrm{L}=\mathrm{Spin}_{10} \times \mathrm{SL}_{3}$.} \label{Spin103}
\begin{align*} 
    \mathfrak{g}[2] &= \mathrm{Spin}_{10}^\prime \otimes \mathrm{SL}_3
\end{align*}
We have $\mathbf{G}_{\gamma,\mathrm{der}}=\mathrm{SL}_{2}^a \times \mathrm{SL}_{2}^b$, where $\mathrm{SL}_{2}^a \hookrightarrow \mathrm{Spin}_{10}$ by $\Lambda_1 \oplus \Lambda_1$, $\mathrm{SL}_{2}^b \hookrightarrow \mathrm{Spin}_{10} \times \mathrm{SL}_{3}$ by $(\Lambda_1 \oplus \Lambda_1 \oplus 4\Lambda_1)\otimes 2\Lambda_1 $. Then we have $(\mathrm{Q}_1^a,\mathrm{Q}_2^a)=(\mathrm{Q}_1^\tau,\mathrm{Q}_2^\tau)=(1,0),(\mathrm{Q}_1^b,\mathrm{Q}_2^b)=(15,0)$. Thus, the orbit is quasi-admissible for $n=1$, raisable if $n \geq 2$.

\subsubsection{\dynkin[edge length=.20cm,labels={2,0,0,0,0,1,0,1},label]{E}{oooooooo} $A_5$, $\mathrm{L}=\mathrm{Spin}_{8} \times \mathrm{SL}_2$.}
\begin{align*}
    \mathfrak{g}[2] &= \mathrm{Spin}_8 \oplus \mathrm{F} \oplus \mathrm{Spin}_8^\prime \\
    \mathfrak{g}[1] &= \mathrm{SL}_2 \oplus (\mathrm{Spin}_{8} \otimes \mathrm{SL}_2)
\end{align*}
The same as \ref{SpinSO8}, we have $\mathbf{G}_{\gamma,\mathrm{der}}=\mathrm{SL}_{2,\alpha_7}^a \times \mathrm{G}_2^b$. As $\mathfrak{sl}_{2,\xi_a}$-mod, $\mathfrak{g}[1]=9V_2$. As $\mathfrak{sl}_{2,\xi_b}$-mod, $\mathfrak{g}[1]=10V_1+4V_2$. Then we have $(\mathrm{Q}_1^a,\mathrm{Q}_2^a)=(\mathrm{Q}_1^{\tau_1},\mathrm{Q}_2^{\tau_1})=(1,9),(\mathrm{Q}_1^b,\mathrm{Q}_2^b)=(\mathrm{Q}_1^{\tau_2},\mathrm{Q}_2^{\tau_2})=(1,4)$. Thus, the orbit is not quasi-admissible for any $n$, raisable if for all $n$.

\subsubsection{\dynkin[edge length=.20cm,labels={0,0,0,1,0,0,0,2},label]{E}{oooooooo} $D_5(a_1)+A_1$, $\mathrm{L}=\mathrm{SL}_3 \times \mathrm{SL}_2 \times \mathrm{SL}_4$.}
\begin{align*}
    \mathfrak{g}[2] &= \mathrm{SL}_4 \oplus (\mathrm{SL}_3 \otimes \bigwedge^2\mathrm{SL}_4) \\
    \mathfrak{g}[1] &= \mathrm{SL}_3^* \otimes \mathrm{SL}_2 \otimes \mathrm{SL}_4^*
\end{align*}
The same as \ref{3329}, we have $\mathbf{G}_{\gamma,\mathrm{der}}=\mathrm{SL}_{2,\alpha_2}^a \times \mathrm{SL}_2^b$. As $\mathfrak{sl}_{2,\xi_a}$-mod, $\mathfrak{g}[1]=12V_2$. As $\mathfrak{sl}_{2,\xi_b}$-mod, $\mathfrak{g}[1]=2V_1+4V_3+2V_5$. Then we have $(\mathrm{Q}_1^a,\mathrm{Q}_2^a)=(\mathrm{Q}_1^{\tau},\mathrm{Q}_2^{\tau})=(1,12),(\mathrm{Q}_1^b,\mathrm{Q}_2^b)=(8,56)$. Thus, the orbit is quasi-admissible for $n=1$, raisable if $n \geq 2$.

\subsubsection{\dynkin[edge length=.20cm,labels={0,0,1,0,0,1,0,0},label]{E}{oooooooo} $A_4+A_2+A_1$, $\mathrm{L}=\mathrm{SL}_2 \times \mathrm{SL}_4 \times \mathrm{SL}_3$.}
\begin{align*}
    \mathfrak{g}[2] &= \mathrm{F} \oplus (\mathrm{SL}_2 \otimes \mathrm{SL}_4 \otimes \mathrm{SL}_3) \\
    \mathfrak{g}[1] &= (\mathrm{SL}_4^* \otimes \mathrm{SL}_3) \oplus (\mathrm{SL}_2 \otimes \bigwedge^2\mathrm{SL}_4)
\end{align*}
The same as \ref{234}, we have $\mathbf{G}_{\gamma,\mathrm{der}}=\mathrm{SL}_{2}$. As $\mathfrak{sl}_{2,\xi}$-mod, $\mathfrak{g}[1]=2V_2+2V_4+2V_6$. Then we have $(\mathrm{Q}_1,\mathrm{Q}_2)=(15,92)$. Thus, the orbit is quasi-admissible for $n|15$, the method for raisability does not work.

\subsubsection{\dynkin[edge length=.20cm,labels={0,2,0,0,0,0,0,2},label]{E}{oooooooo} $D_4+A_2$, $\mathrm{L}=\mathrm{SL}_7$.}
\begin{align*}
    \mathfrak{g}[2] &= \mathrm{SL}_7^* \oplus \bigwedge^3 \mathrm{SL}_7
\end{align*}
The same as \ref{A63}, the summand $\bigwedge^3 \mathrm{SL}_7$
gives prehomogeneous space $(\mathrm{G}_2 \times \mathrm{GL}_1, \mathrm{G}_2)$, which has been computed in \cite{SK77}. By Prop~\ref{prop:restrict}, we have $\mathbf{G}_{\gamma,\mathrm{der}}=\mathrm{SL}_3 \hookrightarrow \mathrm{G}_2 \hookrightarrow \mathrm{SL}_7$ by $\Lambda_1 \oplus \Lambda_1$. Take $\tau$ to be associated with one of the simple roots of $\mathfrak{g}_\gamma$, the conditions $\text{(C1)-(C3)}$ are satisfied. Then we have $(\mathrm{Q}_1,\mathrm{Q}_2)=(\mathrm{Q}_1^\tau,\mathrm{Q}_2^\tau)=(2,0)$. Thus, the orbit is quasi-admissible for $n=1,2$, raisable if $n \geq 3$.

\subsubsection{\dynkin[edge length=.20cm,labels={2,0,0,0,0,0,2,0},label]{E}{oooooooo} $E_6(a_3)$, $\mathrm{L}=\mathrm{Spin}_{10} \times \mathrm{SL}_2$.}
\begin{align*}
    \mathfrak{g}[2] &= (\mathrm{Spin}_{10} \otimes \mathrm{SL}_2) \oplus \mathrm{Spin}_{10}^\prime
\end{align*}
The first summand gives prehomogeneous space $(\mathrm{Spin}_8 \times 2\mathrm{GL}_1, \mathrm{Spin}_8 \oplus \mathrm{Spin}_8^\prime)$, which is the same as \ref{SpinSO8}. By Proposition~\ref{prop:restrict}, we have $\mathbf{G}_{\gamma,\mathrm{der}}=\mathrm{G}_{2} \hookrightarrow \mathrm{SO}_8 \hookrightarrow \mathrm{SO}_{10}$ by standard representation. Then we have $(\mathrm{Q}_1,\mathrm{Q}_2)=(\mathrm{Q}_1^\tau,\mathrm{Q}_2^\tau)=(1,0)$. Thus, the orbit is quasi-admissible for $n=1$, raisable if $n \geq 2$.

\subsubsection{\dynkin[edge length=.20cm,labels={2,0,0,0,0,0,2,2},label]{E}{oooooooo} $D_5$, $\mathrm{L}=\mathrm{Spin}_{10}$.}
\begin{align*}
    \mathfrak{g}[2] &= \mathrm{F} \oplus \mathrm{Spin}_{10} \oplus \mathrm{Spin}_{10}^\prime
\end{align*}
The same as \ref{SpinSO10}, we have $\mathbf{G}_{\gamma,\mathrm{der}}=\mathrm{Spin}_{7}$. Then we have $(\mathrm{Q}_1,\mathrm{Q}_2)=(\mathrm{Q}_1^\tau,\mathrm{Q}_2^\tau)=(1,0)$. Thus, the orbit is quasi-admissible for $n=1$, raisable if $n \geq 2$.

\subsubsection{\dynkin[edge length=.20cm,labels={0,0,0,1,0,0,1,0},label]{E}{oooooooo} $A_4+A_3$, $\mathrm{L}=\mathrm{SL}_3^1 \times \mathrm{SL}_2^1 \times \mathrm{SL}_3^2 \times \mathrm{SL}_2^2$.}
\begin{align*}
    \mathfrak{g}[2] &= (\mathrm{SL}_3^1 \otimes \mathrm{SL}_2^1 \otimes \mathrm{SL}_2^2) \oplus ((\mathrm{SL}_3^1)^* \otimes \mathrm{SL}_3^2) \\
    \mathfrak{g}[1] &= (\mathrm{SL}_3^2 \otimes \mathrm{SL}_2^2) \oplus (\mathrm{SL}_3^1 \otimes \mathrm{SL}_2^1 \otimes (\mathrm{SL}_3^2)^*)
\end{align*}
Let \(V_a = (\mathrm{GL}_{3}^1)^* \otimes \mathrm{SL}_{3}^2\) and let \(V_b\) be the remaining summand in \(\mathfrak{g}[2]\). Then the generic isotropy subgroup for \(V_a\) is \((\mathbf{G}_0)_{V_a} = \Delta(\mathrm{SL}_{3}) \times \mathrm{GL}_{1} \times \mathrm{SL}_2^1 \times \mathrm{SL}_2^2\), where \(\Delta(\mathrm{SL}_{3})\) denotes the diagonal copy inside \(\mathrm{SL}_{3}^1 \times \mathrm{SL}_{3}^2\). The pair \(\bigl((\mathbf{G}_0)_{V_a}, V_b\bigr) = \bigl(\Delta(\mathrm{SL}_{3}) \times \mathrm{GL}_{1} \times \mathrm{SL}_2^1 \times \mathrm{SL}_2^2,\; \mathrm{GL}_{3} \otimes \mathrm{SL}_2^1 \otimes \mathrm{SL}_2^2)\) is a prehomogeneous space. Applying Proposition~\ref{prop:restrict} yields \(\mathbf{G}_{\gamma,\mathrm{der}} = \mathrm{SL}_{2} \hookrightarrow \mathrm{SL}_3^1 \times \mathrm{SL}_2^1 \times \mathrm{SL}_3^2 \times \mathrm{SL}_2^2\) by $2\Lambda_1 \otimes \Lambda_1 \otimes 2\Lambda_1 \otimes \Lambda_1$. As $\mathfrak{sl}_{2,\xi}$-mod, $\mathfrak{g}[1]=3V_2+3V_4+V_6$. Then we have $(\mathrm{Q}_1,\mathrm{Q}_2)=(10,68)$. Thus, the orbit is quasi-admissible for $n|10$, the method for raisability does not work.

\subsubsection{\dynkin[edge length=.20cm,labels={1,0,0,1,0,0,0,1},label]{E}{oooooooo} $A_5+A_1$, $\mathrm{L}=\mathrm{SL}_{2}^1 \times \mathrm{SL}_2^2 \times \mathrm{SL}_4$.}
\begin{align*}
    \mathfrak{g}[2] &= (\mathrm{SL}_{2}^1 \otimes \mathrm{SL}_4) \oplus (\mathrm{SL}_{2}^1 \otimes \mathrm{SL}_2^2) \oplus \bigwedge^2\mathrm{SL}_4\\
    \mathfrak{g}[1] &= \mathrm{SL}_{2}^2 \oplus \mathrm{SL}_4^* \oplus (\mathrm{SL}_{2}^1 \otimes \mathrm{SL}_2^2 \otimes \mathrm{SL}_4)
\end{align*}
Let \(V_a = (\mathrm{SL}_{2}^1 \otimes \mathrm{SL}_4) \oplus \bigwedge^2\mathrm{SL}_4\) and let \(V_b\) be the remaining summand in \(\mathfrak{g}[2]\). By the result in \ref{3214}, we can deduce that $\bigl((\mathbf{G}_0)_{V_a}, V_b\bigr) = (\mathrm{SL}_2 \times \Delta(\mathrm{SL}_{2}) \times \mathrm{GL}_2^2, \; \mathrm{F} \otimes \mathrm{SL}_{2} \otimes \mathrm{GL}_2^2)$, which is a prehomogeneous space. By Proposition~\ref{prop:restrict}, we have $\mathbf{G}_{\gamma,\mathrm{der}}=\mathrm{SL}_{2}^a \times \mathrm{SL}_2^b$, where $\mathrm{SL}_{2}^a \hookrightarrow \mathrm{SL}_4$ by $\Lambda_1$, $\mathrm{SL}_2^b \hookrightarrow \mathrm{SL}_{2}^1 \times \mathrm{SL}_2^2 \times \mathrm{SL}_4$ by $\Lambda_1 \otimes \Lambda_1 \otimes \Lambda_1$. As $\mathfrak{sl}_{2,\xi_a}$-mod, $\mathfrak{g}[1]=12V_1+5V_2$. As $\mathfrak{sl}_{2,\xi_b}$-mod, $\mathfrak{g}[1]=4V_1+4V_2+2V_3+V_4$. Then we have $(\mathrm{Q}_1^a,\mathrm{Q}_2^a)=(\mathrm{Q}_1^{\tau},\mathrm{Q}_2^{\tau})=(1,5),(\mathrm{Q}_1^b,\mathrm{Q}_2^b)=(3,22)$. Thus, the orbit is not quasi-admissible for any $n$, raisable if $n \neq 2$.

\subsubsection{\dynkin[edge length=.20cm,labels={0,0,1,0,0,1,0,1},label]{E}{oooooooo} $D_5(a_1)+A_2$, $\mathrm{L}=\mathrm{SL}_2^1 \times \mathrm{SL}_4 \times \mathrm{SL}_2^2$.}
\begin{align*}
    \mathfrak{g}[2] &= \mathrm{F} \oplus \mathrm{SL}_4 \oplus (\mathrm{SL}_2^1 \otimes \mathrm{SL}_4^* \otimes \mathrm{SL}_2^2) \\
    \mathfrak{g}[1] &= \mathrm{SL}_2^2 \oplus (\mathrm{SL}_4 \otimes \mathrm{SL}_2^2) \oplus (\mathrm{SL}_2^1 \otimes \bigwedge^2 \mathrm{SL}_4)
\end{align*}
Let \(V_a = \mathrm{SL}_2^1 \otimes \mathrm{GL}_4^* \otimes \mathrm{SL}_2^2\) and let \(V_b\) be the remaining summand in \(\mathfrak{g}[2]\). Then the generic isotropy subgroup for \(V_a\) is \((\mathbf{G}_0)_{V_a} = \mathrm{SL}_{2}^3 \times \mathrm{SL}_{2}^4\), where \(\mathrm{SL}_{2}^3 \times \mathrm{SL}_{2}^4 \hookrightarrow \mathrm{SL}_2^1 \times \mathrm{SL}_4\) is identity and \(\mathrm{SL}_{2}^3 \times \mathrm{SL}_{2}^4 \hookrightarrow \mathrm{SL}_4\) by tensoring. The pair \(\bigl((\mathbf{G}_0)_{V_a}, V_b\bigr) = \bigl(\mathrm{SL}_2^3 \times \mathrm{GL}_2^4,\; \mathrm{SL}_{2}^3 \otimes \mathrm{GL}_{2}^4\bigr)\) is a prehomogeneous space. Applying Proposition~\ref{prop:restrict} yields $\mathbf{G}_{\gamma,\mathrm{der}}=\mathrm{SL}_{2} \hookrightarrow \mathrm{SL}_2^1 \times \mathrm{SL}_4 \times \mathrm{SL}_2^2$ by $\Lambda_1 \otimes 2\Lambda_1 \otimes \Lambda_1$. As $\mathfrak{sl}_{2,\xi}$-mod, $\mathfrak{g}[1]=5V_2+3V_4$. Then we have $(\mathrm{Q}_1,\mathrm{Q}_2)=(6,35)$. Thus, the orbit is quasi-admissible for $n=4,12$, the method for raisability does not work.

\subsubsection{\dynkin[edge length=.20cm,labels={0,1,1,0,0,0,1,0},label]{E}{oooooooo} $D_6(a_2)$, $\mathrm{L}=\mathrm{SL}_2^1 \times \mathrm{SL}_4 \times \mathrm{SL}_2^2$.}
\begin{align*}
    \mathfrak{g}[2] &= \mathrm{SL}_2^2 \oplus (\mathrm{SL}_2^1 \otimes \mathrm{SL}_2^2) \oplus (\mathrm{SL}_2^1 \otimes \bigwedge^2\mathrm{SL}_4) \\
    \mathfrak{g}[1] &= \mathrm{SL}_4 \oplus (\mathrm{SL}_2^1 \otimes \mathrm{SL}_4) \oplus (\mathrm{SL}_4^* \otimes \mathrm{SL}_2^2)
\end{align*}
The generic isotropy subgroup with respect to $\mathrm{GL}_2^1 \otimes \bigwedge^2\mathrm{SL}_4$ is $\mathrm{GL}_2^3 \times \mathrm{SL}_2^4 \hookrightarrow \mathrm{SL}_4$ by tensoring. Then we obtain prehomogeneous space $(\mathrm{GL}_2^2 \times \mathrm{GL}_2^3 \times \mathrm{GL}_2^4,\; 3\mathrm{GL}_2^2)$. By Proposition~\ref{prop:restrict} and dimension computing, we have $\mathbf{G}_{\gamma,\mathrm{der}}=\mathrm{SL}_2^3 \times \mathrm{SL}_2^4 \hookrightarrow \mathrm{SL}_4$ by tensoring. For each $\mathrm{SL}_2$, as $\mathfrak{sl}_{2,\xi}$-mod, $\mathfrak{g}[1]=10V_2$. Then we have $(\mathrm{Q}_1,\mathrm{Q}_2)=(\mathrm{Q}_1^\tau,\mathrm{Q}_2^\tau)=(2,10)$. Thus, the orbit is quasi-admissible for $n=1,2$, raisable if $n \geq 3$.

\subsubsection{\dynkin[edge length=.20cm,labels={1,0,0,0,1,0,1,0},label]{E}{oooooooo} $E_6(a_3)+A_1$, $\mathrm{L}=\mathrm{SL}_4 \times \mathrm{SL}_2^1 \times \mathrm{SL}_2^2$.}
\begin{align*}
    \mathfrak{g}[2] &= \mathrm{F} \oplus (\mathrm{SL}_4 \otimes \mathrm{SL}_2^1) \oplus (\bigwedge^2\mathrm{SL}_4 \otimes \mathrm{SL}_2^2) \\
    \mathfrak{g}[1] &= (\mathrm{SL}_2^1 \otimes \mathrm{SL}_2^2) \oplus \mathrm{SL}_4^* \oplus (\bigwedge^2\mathrm{SL}_4 \otimes \mathrm{SL}_2^1)
\end{align*}
The same as \ref{3231}, we have $\mathbf{G}_{\gamma,\mathrm{der}}=\mathrm{1}$. 

\subsubsection{\dynkin[edge length=.20cm,labels={0,0,0,1,0,1,0,0},label]{E}{oooooooo} $E_7(a_5)$, $\mathrm{L}=\mathrm{SL}_{3}^1 \times \mathrm{SL}_2^1 \times \mathrm{SL}_2^2 \times \mathrm{SL}_3^2$.}
\begin{align*}
    \mathfrak{g}[2] &= \mathrm{SL}_{3}^1 \oplus ((\mathrm{SL}_3^1)^* \otimes \mathrm{SL}_2^1 \otimes \mathrm{SL}_3^2) \\
    \mathfrak{g}[1] &= (\mathrm{SL}_2^2 \otimes \mathrm{SL}_3^2) \oplus ((\mathrm{SL}_{3}^1)^* \otimes \mathrm{SL}_2^1 \otimes \mathrm{SL}_2^2)
\end{align*}
By computing dimension, we have $\mathbf{G}_{\gamma,\mathrm{der}}=\mathrm{SL}_{2,\alpha_5}$. As $\mathfrak{sl}_{2,\xi}$-mod, $\mathfrak{g}[1]=9V_2$. Then we have $(\mathrm{Q}_1,\mathrm{Q}_2)=(\mathrm{Q}_1^\tau,\mathrm{Q}_2^\tau)=(1,9)$. Thus, the orbit is quasi-admissible for $n=2$, raisable if $n \neq 2$.

\subsubsection{\dynkin[edge length=.20cm,labels={1,0,0,0,1,0,1,2},label]{E}{oooooooo} $D_5+A_1$, $\mathrm{L}=\mathrm{SL}_{4} \times \mathrm{SL}_2$.}
\begin{align*}
    \mathfrak{g}[2] &= 2\mathrm{F} \oplus \bigwedge^2 \mathrm{SL}_4 \oplus (\mathrm{SL}_{4} \otimes \mathrm{SL}_2) \\
    \mathfrak{g}[1] &= \mathrm{SL}_{4}^* \oplus \mathrm{SL}_2 \oplus (\bigwedge^2\mathrm{SL}_{4} \otimes \mathrm{SL}_2)
\end{align*}
The same as \ref{3214}, we have $\mathbf{G}_{\gamma,\mathrm{der}}=\Delta(\mathrm{SL}_{2})^a \times \mathrm{SL}_2^b$. As $\mathfrak{sl}_{2,\xi_a}$-mod, $\mathfrak{g}[1]=4V_1+4V_2+2V_3$. As $\mathfrak{sl}_{2,\xi_b}$-mod, $\mathfrak{g}[1]=8V_1+5V_2$. Then we have $(\mathrm{Q}_1^a,\mathrm{Q}_2^a)=(2,12),(\mathrm{Q}_1^a,\mathrm{Q}_2^a)=(\mathrm{Q}_1^\tau,\mathrm{Q}_2^\tau)=(1,5)$. Thus, the orbit is quasi-admissible for $n=2$, raisable if $n \neq 2$.

\subsubsection{\dynkin[edge length=.20cm,labels={2,0,0,0,0,2,0,0},label]{E}{oooooooo} $A_6$, $\mathrm{L}=\mathrm{Spin}_8 \times \mathrm{SL}_3$.}
\begin{align*}
    \mathfrak{g}[2] &= \mathrm{Spin}_8 \oplus (\mathrm{Spin}_8^\prime \otimes \mathrm{SL}_3) 
\end{align*}
The first summand gives generic isotropy subgroup $\mathrm{Spin}_7 \hookrightarrow \mathrm{Spin}_8$. Then we obtain prehomogeneous space $(\mathrm{Spin}_7 \times \mathrm{GL}_3,\; \mathrm{Spin}_7^\prime \otimes \mathrm{GL}_3)$. By Proposition~\ref{prop:restrict}, we have $\mathbf{G}_{\gamma,\mathrm{der}}=\mathrm{SL}_{2}^a \times \mathrm{SO}_3^b$, where $\mathrm{SL}_{2}^a \hookrightarrow \mathrm{Spin}_7$ by $\Lambda_1 \oplus \Lambda_1$, $\mathrm{SO}_3^b \hookrightarrow \mathrm{Spin}_7 \times \mathrm{SL}_3$ by $(\Lambda_1 \oplus \Lambda_1) \otimes \Lambda_1$. Then we have $(\mathrm{Q}_1^a,\mathrm{Q}_2^a)=(\mathrm{Q}_1^\tau,\mathrm{Q}_2^\tau)=(1,0),(\mathrm{Q}_1^b,\mathrm{Q}_2^b)=(8,0)$. Thus, the orbit is quasi-admissible for $n=1$, raisable if $n \geq 2$.

\subsubsection{\dynkin[edge length=.20cm,labels={0,1,1,0,0,0,1,2},label]{E}{oooooooo} $D_6(a_1)$, $\mathrm{L}=\mathrm{SL}_{2} \times \mathrm{SL}_4$.}
\begin{align*}
    \mathfrak{g}[2] &= 2\mathrm{F} \oplus \mathrm{SL}_2 \oplus (\mathrm{SL}_{2} \otimes \bigwedge^2\mathrm{SL}_4) \\
    \mathfrak{g}[1] &= \mathrm{SL}_4 \oplus \mathrm{SL}_4^* \oplus (\mathrm{SL}_{2} \otimes \mathrm{SL}_4)
\end{align*}
The same as \ref{3218}, we have $\mathbf{G}_{\gamma,\mathrm{der}}=\mathrm{SL}_{2} \times \mathrm{SL}_2$. As $\mathfrak{sl}_{2,\xi}$-mod, $\mathfrak{g}[1]=8V_2$. Then we have $(\mathrm{Q}_1,\mathrm{Q}_2)=(\mathrm{Q}_1^\tau,\mathrm{Q}_2^\tau)=(1,8)$. Thus, the orbit is quasi-admissible for $n=1$, raisable if $n \geq 2$.

\subsubsection{\dynkin[edge length=.20cm,labels={1,0,0,1,0,1,0,0},label]{E}{oooooooo} $A_6+A_1$, $\mathrm{L}=\mathrm{SL}_{2}^1 \times \mathrm{SL}_2^2 \times \mathrm{SL}_{2}^3 \times \mathrm{SL}_3$.}
\begin{align*}
    \mathfrak{g}[2] &= \mathrm{F} \oplus (\mathrm{SL}_{2}^1 \otimes \mathrm{SL}_{2}^3) \oplus (\mathrm{SL}_2^1 \otimes \mathrm{SL}_{2}^2 \otimes \mathrm{SL}_3) \\
    \mathfrak{g}[1] &= \mathrm{SL}_2^2 \oplus (\mathrm{SL}_{2}^1 \otimes \mathrm{SL}_2^2 \otimes \mathrm{SL}_{2}^3) \oplus (\mathrm{SL}_{2}^3 \otimes \mathrm{SL}_3)
\end{align*}
The same as \ref{3.33}, we have $\mathbf{G}_{\gamma,\mathrm{der}}=\mathrm{SL}_{2}$. As $\mathfrak{sl}_{2,\xi}$-mod, $\mathfrak{g}[1]=4V_2+2V_4$. Then we have $(\mathrm{Q}_1,\mathrm{Q}_2)=(7,24)$. Thus, the orbit is quasi-admissible for $n\mid 7$, the method for raisability does not work.

\subsubsection{\dynkin[edge length=.20cm,labels={0,0,0,1,0,1,0,2},label]{E}{oooooooo} $E_7(a_4)$, $\mathrm{L}=\mathrm{SL}_3 \times \mathrm{SL}_{2}^1 \times \mathrm{SL}_2^2 \times \mathrm{SL}_{2}^3$.}
\begin{align*}
    \mathfrak{g}[2] &= \mathrm{SL}_3^* \oplus \mathrm{SL}_{2}^3 \oplus (\mathrm{SL}_3 \otimes \mathrm{SL}_{2}^1 \times \mathrm{SL}_2^3) \\
    \mathfrak{g}[1] &= (\mathrm{SL}_2^2 \otimes \mathrm{SL}_{2}^3) \oplus (\mathrm{SL}_3 \otimes \mathrm{SL}_{2}^1 \otimes \mathrm{SL}_2^2)
\end{align*}
By computing dimension, we have $\mathbf{G}_{\gamma,\mathrm{der}}=\mathrm{SL}_{2,\alpha_5}$. As $\mathfrak{sl}_{2,\xi}$-mod, $\mathfrak{g}[1]=8V_2$. Then we have $(\mathrm{Q}_1,\mathrm{Q}_2)=(\mathrm{Q}_1^\tau,\mathrm{Q}_2^\tau)=(1,8)$. Thus, the orbit is quasi-admissible for $n=1$, raisable if $n \geq 2$.

\subsubsection{\dynkin[edge length=.20cm,labels={2,0,0,0,0,2,0,2},label]{E}{oooooooo} $E_6(a_1)$, $\mathrm{L}=\mathrm{Spin}_8 \times \mathrm{SL}_2$.}
\begin{align*}
    \mathfrak{g}[2] &= \mathrm{SL}_2 \oplus \mathrm{Spin}_8^\prime \oplus (\mathrm{Spin}_8 \otimes \mathrm{SL}_2)
\end{align*}
Let \(V_a = \mathrm{Spin}_8^\prime \oplus (\mathrm{Spin}_8 \otimes \mathrm{SL}_2)\) and let \(V_b=\mathrm{SL}_2\). By \ref{Spin8SOA1}, note that $V_b$ is trivial as $(\mathbf{G}_0)_{V_a}$-mod. By Proposition~\ref{prop:restrict}, we have $\mathbf{G}_{\gamma,\mathrm{der}}=\mathrm{SL}_{3}$. Then we have $(\mathrm{Q}_1,\mathrm{Q}_2)=(\mathrm{Q}_1^\tau,\mathrm{Q}_2^\tau)=(1,0)$. Thus, the orbit is quasi-admissible for $n=1$, raisable if $n \geq 2$.

\subsubsection{\dynkin[edge length=.20cm,labels={0,0,0,0,2,0,0,2},label]{E}{oooooooo} $D_5+A_2$, $\mathrm{L}=\mathrm{SL}_5 \oplus \mathrm{SL}_3$.}
\begin{align*}
    \mathfrak{g}[2] &= \mathrm{SL}_3^* \oplus (\mathrm{SL}_3 \otimes \bigwedge^2 \mathrm{SL}_5)
\end{align*}
By computing dimension, we have $\mathbf{G}_{\gamma,\mathrm{der}}=1$.

\subsubsection{\dynkin[edge length=.20cm,labels={2,1,1,0,0,0,1,2},label]{E}{oooooooo} $D_6$, $\mathrm{L}=\mathrm{SL}_{4}$.}
\begin{align*}
    \mathfrak{g}[2] &= 4\mathrm{F} \oplus \bigwedge^2\mathrm{SL}_4 \\
    \mathfrak{g}[1] &= 2\mathrm{SL}_4 \oplus \mathrm{SL}_4^*
\end{align*}
We have $\mathbf{G}_{\gamma,\mathrm{der}}=\mathrm{Sp}_{4} \hookrightarrow \mathrm{SL}_4$ by $\Lambda_1$. As $\mathfrak{sl}_{2,\xi}$-mod, $\mathfrak{g}[1]=6V_1+3V_2$. Then we have $(\mathrm{Q}_1,\mathrm{Q}_2)=(\mathrm{Q}_1^\tau,\mathrm{Q}_2^\tau)=(1,3)$. Thus, the orbit is quasi-admissible for $n=2$, raisable if $n \neq 2$.

\subsubsection{\dynkin[edge length=.20cm,labels={2,0,0,0,0,2,2,2},label]{E}{oooooooo} $E_6$, $\mathrm{L}=\mathrm{Spin}_8$.}
\begin{align*}
    \mathfrak{g}[2] &= 2\mathrm{F} \oplus \mathrm{Spin}_8 \oplus \mathrm{Spin}_8^\prime
\end{align*}
The same as \ref{SpinSO8}, we have $\mathbf{G}_{\gamma,\mathrm{der}}=\mathrm{G}_{2}$. Then we have $(\mathrm{Q}_1,\mathrm{Q}_2)=(\mathrm{Q}_1^\tau,\mathrm{Q}_2^\tau)=(1,0)$. Thus, the orbit is quasi-admissible for $n=1$, raisable if $n \geq 2$.

\subsubsection{\dynkin[edge length=.20cm,labels={1,0,0,1,0,1,0,1},label]{E}{oooooooo} $D_7(a_2)$, $\mathrm{L}=\mathrm{SL}_{2}^1 \times \mathrm{SL}_2^2 \times \mathrm{SL}_{2}^3 \times \mathrm{SL}_2^4$.}
\begin{align*}
    \mathfrak{g}[2] &= \mathrm{F} \oplus \mathrm{SL}_2^3 \oplus (\mathrm{SL}_{2}^1 \otimes \mathrm{SL}_2^3) \oplus (\mathrm{SL}_{2}^1 \otimes \mathrm{SL}_2^2 \otimes \mathrm{SL}_{2}^4)
\end{align*}
By computing dimension, we have $\mathbf{G}_{\gamma,\mathrm{der}}=1$.

\subsubsection{\dynkin[edge length=.20cm,labels={1,0,0,1,0,1,1,0},label]{E}{oooooooo} $A_7$, $\mathrm{L}=\mathrm{SL}_{2}^1 \times \mathrm{SL}_2^2 \times \mathrm{SL}_{2}^3 \times \mathrm{SL}_2^4$.}
\begin{align*}
    \mathfrak{g}[2] &= \mathrm{F} \oplus (\mathrm{SL}_{2}^1 \otimes \mathrm{SL}_2^2) \oplus (\mathrm{SL}_{2}^1 \otimes \mathrm{SL}_2^3) \oplus (\mathrm{SL}_{2}^3 \otimes \mathrm{SL}_2^4) \\
    \mathfrak{g}[1] &= \mathrm{SL}_2^2 \oplus \mathrm{SL}_{2}^3 \oplus \mathrm{SL}_2^4 \oplus (\mathrm{SL}_{2}^1 \otimes \mathrm{SL}_2^2 \otimes \mathrm{SL}_{2}^3)
\end{align*}
Similarly to \ref{dA1}, we have $\mathbf{G}_{\gamma,\mathrm{der}}=\Delta(\mathrm{SL}_{2})$. As $\mathfrak{sl}_{2,\xi}$-mod, $\mathfrak{g}[1]=5V_2+V_4$. Then we have $(\mathrm{Q}_1,\mathrm{Q}_2)=(4,15)$. Thus, the orbit is quasi-admissible for $n=8$, the method for raisability does not work.

\subsubsection{\dynkin[edge length=.20cm,labels={1,0,0,1,0,1,0,2},label]{E}{oooooooo} $E_6(a_1)+A_1$, $\mathrm{L}=\mathrm{SL}_{2}^1 \times \mathrm{SL}_2^2 \times \mathrm{SL}_{2}^3 \times \mathrm{SL}_2^4$.}
\begin{align*}
    \mathfrak{g}[2] &= \mathrm{F} \oplus \mathrm{SL}_2^4 \oplus (\mathrm{SL}_{2}^1 \otimes \mathrm{SL}_2^3) \oplus (\mathrm{SL}_{2}^1 \otimes \mathrm{SL}_2^2 \otimes \mathrm{SL}_2^4)
\end{align*}
By computing dimension, we have $\mathbf{G}_{\gamma,\mathrm{der}}=1$.

\subsubsection{\dynkin[edge length=.20cm,labels={2,0,0,1,0,1,0,2},label]{E}{oooooooo} $E_7(a_3)$, $\mathrm{L}=\mathrm{SL}_{2}^1 \times \mathrm{SL}_2^2 \times \mathrm{SL}_{2}^3 \times \mathrm{SL}_2^4$.}
\begin{align*}
    \mathfrak{g}[2] &= \mathrm{F} \oplus \mathrm{SL}_2^2 \oplus \mathrm{SL}_{2}^4 \oplus (\mathrm{SL}_{2}^1 \otimes \mathrm{SL}_2^2 \otimes \mathrm{SL}_{2}^4)\\
    \mathfrak{g}[1] &= (\mathrm{SL}_2^3 \otimes \mathrm{SL}_{2}^4) \oplus (\mathrm{SL}_{2}^1 \otimes \mathrm{SL}_2^2 \otimes \mathrm{SL}_{2}^3)
\end{align*}
By computing dimension, we have $\mathbf{G}_{\gamma,\mathrm{der}}=\mathrm{SL}_{2,\alpha_5}$. As $\mathfrak{sl}_{2,\xi}$-mod, $\mathfrak{g}[1]=6V_2$. Then we have $(\mathrm{Q}_1,\mathrm{Q}_2)=(\mathrm{Q}_1^\tau,\mathrm{Q}_2^\tau)=(1,6)$. Thus, the orbit is quasi-admissible for $n=1$, raisable if $n \geq 2$.

\subsubsection{\dynkin[edge length=.20cm,labels={2,0,0,0,2,0,0,2},label]{E}{oooooooo} $D_7(a_1)$, $\mathrm{L}=\mathrm{SL}_{3} \times \mathrm{SL}_4$.}
\begin{align*}
    \mathfrak{g}[2] &= \mathrm{SL}_{3} \oplus \mathrm{SL}_4 \oplus (\mathrm{SL}_{3}^* \otimes \bigwedge^2\mathrm{SL}_4)
\end{align*}
We have $\text{dim}\mathbf{G}_{\gamma}=15+8+3-25=1$, hence $\mathbf{G}_{\gamma,\mathrm{der}}=1$.

\subsubsection{\dynkin[edge length=.20cm,labels={1,0,0,1,0,1,2,2},label]{E}{oooooooo} $E_6+A_1$, $\mathrm{L}=\mathrm{SL}_{2}^1 \times \mathrm{SL}_2^2 \times \mathrm{SL}_{2}^3$.}
\begin{align*}
    \mathfrak{g}[2] &= 3\mathrm{F} \oplus (\mathrm{SL}_{2}^1 \otimes \mathrm{SL}_2^2) \oplus (\mathrm{SL}_{2}^1 \otimes \mathrm{SL}_2^3) \\
    \mathfrak{g}[1] &= \mathrm{SL}_{2}^2 \oplus \mathrm{SL}_2^3 \oplus (\mathrm{SL}_{2}^1 \otimes \mathrm{SL}_2^2 \otimes \mathrm{SL}_2^3)
\end{align*}
The same as \ref{dA1}, we have $\mathbf{G}_{\gamma,\mathrm{der}}=\Delta(\mathrm{SL}_{2})$. As $\mathfrak{sl}_{2,\xi}$-mod, $\mathfrak{g}[1]=4V_2+V_4$. Then we have $(\mathrm{Q}_1,\mathrm{Q}_2)=(3,14)$. Thus, the orbit is quasi-admissible for $n=1,3$, the method for raisability does not work.

\subsubsection{\dynkin[edge length=.20cm,labels={0,1,1,0,1,0,2,2},label]{E}{oooooooo} $E_7(a_2)$, $\mathrm{L}=\mathrm{SL}_{2}^1 \times \mathrm{SL}_2^2 \times \mathrm{SL}_{2}^3$.}
\begin{align*}
    \mathfrak{g}[2] &= \mathrm{F} \oplus \mathrm{SL}_{2}^1 \oplus 2\mathrm{SL}_2^3 \oplus (\mathrm{SL}_{2}^1 \otimes \mathrm{SL}_2^3)\\
    \mathfrak{g}[1] &= \mathrm{SL}_{2}^2 \oplus (\mathrm{SL}_2^1 \otimes \mathrm{SL}_{2}^2) \oplus (\mathrm{SL}_2^2 \otimes \mathrm{SL}_2^3)
\end{align*}
By computing dimension, we have $\mathbf{G}_{\gamma,\mathrm{der}}=\mathrm{SL}_{2,\alpha_4}$. As $\mathfrak{sl}_{2,\xi}$-mod, $\mathfrak{g}[1]=5V_2$. Then we have $(\mathrm{Q}_1,\mathrm{Q}_2)=(\mathrm{Q}_1^\tau,\mathrm{Q}_2^\tau)=(1,5)$. Thus, the orbit is quasi-admissible for $n=2$, raisable if $n \neq 2$.

\subsubsection{\dynkin[edge length=.20cm,labels={2,1,1,0,1,1,0,1},label]{E}{oooooooo} $D_7$, $\mathrm{L}=\mathrm{SL}_{2}^1 \times \mathrm{SL}_2^2$.}
\begin{align*}
    \mathfrak{g}[2] &= 5\mathrm{F} \oplus (\mathrm{SL}_{2}^1 \otimes \mathrm{SL}_2^2) \\
    \mathfrak{g}[1] &= 3\mathrm{SL}_{2}^1 \oplus 2\mathrm{SL}_2^2
\end{align*}
The same as \ref{3.34}, we have $\mathbf{G}_{\gamma,\mathrm{der}}=\Delta(\mathrm{SL}_{2})$. As $\mathfrak{sl}_{2,\xi}$-mod, $\mathfrak{g}[1]=5V_2$. The only choice of $\tau$ satisfies $(\text{C}1)-(\text{C}3)$ according to \cite{JLS16}. Then we have $(\mathrm{Q}_1,\mathrm{Q}_2)=(\mathrm{Q}_1^\tau,\mathrm{Q}_2^\tau)=(2,5)$. Thus, the orbit is quasi-admissible for $n=4$, raisable if $n \neq 4$.

\begin{table}[H]
\centering
\resizebox{\textwidth}{!}{%  % 缩放表格到页面宽度
\begin{tabular}{cccccc}
\hline
$\mathcal{O}$ & special/even? & $\mathbf{G}_{\gamma,\text{der}}$ & $(Q_1, Q_2)$ & quasi-admissible, if and only if & raisable if \\
\hline
$\{0\}$ & yes/yes & $\mathrm{E}_{8}$ & $(1,0)$ & $n=1$ & $n \geq 2$ \\
$A_1$ & yes/no & $\mathrm{E}_{7}$ & $(1,12)$ & $n=1$ & $n \geq 2$ \\
$2A_1$ & yes/no & $\mathrm{Spin}_{13}$ & $(1,16)$ & $n=1$ & $n \geq 2$ \\
$3A_1$ & no/no & $\mathrm{SL}_{2,\alpha_8} \times \mathrm{F}_4$ & $(1,27),(1,12)$ & no such n & all $n$ \\
$A_2$ & yes/yes & $\mathrm{E}_{6}$ & $(1,0)$ & $n=1$ & $n \geq 2$ \\
$4A_1$ & no/no & $\mathrm{Sp}_{8}$ & $(1,15)$ & $n=2$ & $n \neq 2$ \\
$A_2+A_1$ & yes/no & $\mathrm{SL}_{6}$ & $(1,10)$ & $n=1$ & $n \geq 2$ \\
$A_2+2A_1$ & yes/no & $\mathrm{Spin}_{7} \times \mathrm{SL}_{2}$ & $(1,12),(6,88)$ & $n=1$ & $n \geq 2$ \\
$A_3$ & yes/no & $\mathrm{Spin}_{11}$ & $(1,8)$ & $n=1$ & $n \geq 2$ \\
$A_2+3A_1$ & no/no & $\mathrm{SL}_{2,\alpha_1} \times \mathrm{G}_{2}$ & $(1,21),(2,20)$ & $n=2$ & $n \neq 2$ \\
$2A_2$ & yes/yes & $\mathrm{G}_{2} \times \mathrm{G}_{2}$ & $(1,0),(1,0)$ & $n=1$ & $n \geq 2$ \\
$2A_2+A_1$ & no/no & $\mathrm{SL}_{2} \times \mathrm{G}_2$ & $(3,26),(1,8)$ & $n=1$ & n.a. \\
$A_3+A_1$ & no/no & $\mathrm{SL}_{2,\alpha_7} \times \mathrm{Spin}_7$ & $(1,17),(1,8)$ & no such n & all $n$ \\
$D_4(a_1)$ & yes/yes & $\mathrm{Spin}_{8}$ & $(1,0)$ & $n=1$ & $n \geq 2$ \\
$D_4$ & yes/yes & $\mathrm{F}_{4}$ & $(1,0)$ & $n=1$ & $n \geq 2$ \\
$2A_2+2A_1$ & no/no & $\mathrm{Sp}_{4}$ & $(3,34)$ & $n=1,3$ & n.a.\\
$A_3+2A_1$ & no/no & $\mathrm{SL}_{2} \times \mathrm{Sp}_{4}$ & $(2,24),(1,9)$ & $n=2$ & $n \neq 2$ \\
$D_4(a_1)+A_1$ & yes/no & $\mathrm{SL}_{2} \times \mathrm{SL}_2 \times \mathrm{SL}_2$ & $(1,8),(1,8),(1,8)$ & $n=1$ & $n \geq 2$ \\
$A_3+A_2$ & yes/no & $\mathrm{Sp}_{4}$ & $(1,8)$ & $n=1$ & $n \geq 2$ \\
$A_4$ & yes/yes & $\mathrm{SL}_{5}$ & $(1,0)$ & $n=1$ & $n \geq 2$\\
$A_3+A_2+A_1$ & no/no & $\mathrm{SO}_{3} \times \mathrm{SL}_{2,\alpha_2}$ & $(24,160),(1,15)$ & $n=2$ & $n \neq 2$ \\
$D_4+A_1$ & no/no & $\mathrm{Sp}_{6}$ & $(1,7)$ & $n=2$ & $n \neq 2$ \\
$D_4(a_1)+A_2$ & yes/yes & $\mathrm{PGL}_{3}$ & $(6,0)$ & $n=1,2,3,6$ & $n \neq 1,2,3,6$ \\
$A_4+A_1$ & yes/no & $\mathrm{SL}_3$ & $(1,6)$ & $n=1$ & $n \geq 2$\\
$2A_3$ & no/no & $\mathrm{Sp}_{4}$ & $(2,15)$ & $n=4$ & n.a.\\
$D_5(a_1)$ & yes/no & $\mathrm{SL}_{4}$ & $(1,6)$ & $n=1$ & $n \geq 2$ \\
$A_4+2A_1$ & yes/no & $\mathrm{SL}_2$ & $(2,14)$ & $n=1,2$ & $n \geq 3$ \\
$A_4+A_2$ & yes/yes & $\mathrm{SL}_{2} \times \mathrm{SL}_{2}$ & $(1,0),(15,0)$ & $n=1$ & $n \geq 2$ \\
$A_5$ & no/no & $\mathrm{SL}_{2,\alpha_7} \times \mathrm{G}_2$ & $(1,9),(1,4)$ & no such n & all $n$ \\
$D_5(a_1)+A_1$ & yes/no & $\mathrm{SL}_{2,\alpha_2} \times \mathrm{SL}_2$ & $(1,12),(8,56)$ & $n=1$ & $n \geq 2$\\
$A_4+A_2+A_1$ & yes/no & $\mathrm{SL}_2$ & $(15,92)$ & $n|15$ & n.a.\\
$D_4+A_2$ & yes/yes & $\mathrm{SL}_3$& $(2,0)$ & $n=1,2$ & $n \geq 3$ \\
$E_6(a_3)$ & yes/yes & $\mathrm{G}_2$ & $(1,0)$ & $n=1$ & $n \geq 2$ \\
$D_5$ & yes/yes & $\mathrm{Spin}_7$ & $(1,0)$ & $n=1$ & $n \geq 2$ \\
$A_4+A_3$ & no/no & $\mathrm{SL}_2$ & $(10,68)$ & $n \mid 10$ & n.a. \\
$A_5+A_1$ & no/no & $\mathrm{SL}_{2} \times \mathrm{SL}_2$ & $(1,5),(3,22)$ & no such n & $n \neq 2$ \\
$D_5(a_1)+A_2$ & no/no & $\mathrm{SL}_2$ & $(6,35)$ & $n=4,12$ & n.a.\\
$D_6(a_2)$ & no/no & $\mathrm{SL}_2 \times \mathrm{SL}_2$ & $(2,10),(2,10)$ & $n=1,2$ & $n \geq 3$\\
$E_6(a_3)+A_1$ & no/no & $\mathrm{1}$& n.a. & all n & n.a. \\
$E_7(a_5)$ & no/no & $\mathrm{SL}_{2,\alpha_5}$ & $(1,9)$ & $n=2$ & $n \neq 2$ \\
$D_5+A_1$ & no/no & $\mathrm{SL}_{2} \times \mathrm{SL}_2 $ & $(2,12),(1,5)$ & $n=2$ & $n \neq 2$ \\
$E_8(a_7)$ & yes/yes & $\mathrm{1}$& n.a. & all n & n.a. \\
$A_6$ & yes/yes & $\mathrm{SL}_{2} \times \mathrm{SO}_3 $ & $(1,0),(8,0)$ & $n=1$ & $n \geq 2$ \\
$D_6(a_1)$ & yes/no & $\mathrm{SL}_{2} \times \mathrm{SL}_2$ & $(1,8),(1,8)$ & $n=1$ & $n \geq 2$ \\
$A_6+A_1$ & no/no & $\mathrm{SL}_{2}$ & $(7,24)$ & $n\mid7$ & n.a. \\
$E_7(a_4)$ & yes/no & $\mathrm{SL}_{2,\alpha_5}$ & $(1,8)$ & $n=1$ & $n \geq 2$ \\
$E_6(a_1)$ & yes/yes & $\mathrm{SL}_{3}$ & $(1,0)$ & $n=1$ & $n \geq 2$ \\
$D_5+A_2$ & yes/yes & $\mathrm{1}$& n.a. & all n & n.a. \\
$D_6$ & no/no & $\mathrm{Sp}_{4}$ & $(1,3)$ & $n=2$ & $n \neq 2$ \\
$E_6$ & yes/yes & $\mathrm{G}_{2}$ & $(1,0)$ & $n=1$ & $n \geq 2$ \\
$D_7(a_2)$ & yes/no & $\mathrm{1}$& n.a. & all n & n.a. \\
$A_7$ & no/no & $\mathrm{SL}_{2}$ & $(4,15)$ & $n=8$ & n.a. \\
$E_6(a_1)+A_1$ & yes/no & $\mathrm{1}$& n.a. & all n & n.a. \\
$E_7(a_3)$ & yes/no & $\mathrm{SL}_{2,\alpha_5}$ & $(1,6)$ & $n=1$ & $n \geq 2$ \\
$E_8(b_6)$ & yes/yes & $\mathrm{1}$& n.a. & all n & n.a. \\
$D_7(a_1)$ & yes/yes & $\mathrm{1}$& n.a. & all n & n.a. \\
$E_6+A_1$ & no/no & $\mathrm{SL}_{2}$ & $(3,14)$ & $n=1,3$ & n.a. \\
\hline
\end{tabular}
}
\caption{Nilpotent orbits for \( E_8 \)} \label{tab:e8}
\end{table}

\begin{table}[H]
\centering
\resizebox{\textwidth}{!}{%  % 缩放表格到页面宽度
\begin{tabular}{cccccc}
\hline
$\mathcal{O}$ & special/even? & $\mathbf{G}_{\gamma,\text{der}}$ & $(Q_1, Q_2)$ & quasi-admissible, if and only if & raisable if \\
\hline
$E_7(a_2)$ & no/no & $\mathrm{SL}_{2,\alpha_4}$ & $(1,5)$ & $n=2$ & $n \neq 2$ \\
$E_8(a_6)$ & yes/yes & $\mathrm{1}$& n.a. & all n & n.a. \\
$D_7$ & no/no & $\mathrm{SL}_{2}$ & $(2,5)$ & $n=4$ & $n \neq 4$ \\
$E_8(b_5)$ & yes/yes & $\mathrm{1}$& n.a. & all n & n.a. \\
$E_7(a_1)$ & yes/no & $\mathrm{SL}_{2,\alpha_4}$ & $(1,4)$ & $n=1$ & $n \geq 2$ \\
$E_8(a_5)$ & yes/yes & $\mathrm{1}$& n.a. & all n & n.a. \\
$E_8(b_4)$ & yes/yes & $\mathrm{1}$& n.a. & all n & n.a. \\
$E_7$ & no/no & $\mathrm{SL}_{2,\alpha_4}$ & $(1,3)$ & $n=2$ & $n \neq 2$ \\
$E_8(a_4)$ & yes/yes & $\mathrm{1}$& n.a. & all n & n.a. \\
$E_8(a_3)$ & yes/yes & $\mathrm{1}$& n.a. & all n & n.a. \\
$E_8(a_2)$ & yes/yes & $\mathrm{1}$& n.a. & all n & n.a. \\
$E_8(a_1)$ & yes/yes & $\mathrm{1}$& n.a. & all n & n.a. \\
$E_8$ & yes/yes & $\mathrm{1}$& n.a. & all n & n.a. \\
\hline
\end{tabular}
}
\caption{Nilpotent orbits for \( \overline{E_8}^{(n)} \)(continued)} \label{tab:e8c}
\end{table}

\subsubsection{\dynkin[edge length=.20cm,labels={2,1,1,0,1,0,2,2},label]{E}{oooooooo} $E_7(a_1)$, $\mathrm{L}=\mathrm{SL}_{2}^1 \times \mathrm{SL}_2^2$.}
\begin{align*}
    \mathfrak{g}[2] &= 3\mathrm{F} \oplus 3\mathrm{SL}_2^2 \\
    \mathfrak{g}[1] &= 2\mathrm{SL}_{2}^1 \oplus (\mathrm{SL}_{2}^1 \otimes \mathrm{SL}_2^2)
\end{align*}
By computing dimension, we have $\mathbf{G}_{\gamma,\mathrm{der}}=\mathrm{SL}_{2}^1=\mathrm{SL}_{2,\alpha_4}$. As $\mathfrak{sl}_{2,\xi}$-mod, $\mathfrak{g}[1]=4V_2$. Then we have $(\mathrm{Q}_1,\mathrm{Q}_2)=(\mathrm{Q}_1^\tau,\mathrm{Q}_2^\tau)=(1,4)$. Thus, the orbit is quasi-admissible for $n=1$, raisable if $n \geq 2$.

\subsubsection{\dynkin[edge length=.20cm,labels={2,1,1,0,1,2,2,2},label]{E}{oooooooo} $E_7$, $\mathrm{L}=\mathrm{SL}_{2,\alpha_4}$.}
\begin{align*}
    \mathfrak{g}[2] &= 7\mathrm{F} \\
    \mathfrak{g}[1] &= 3\mathrm{SL}_{2,\alpha_4}
\end{align*}
We have $\mathbf{G}_{\gamma,\mathrm{der}}=\mathrm{SL}_{2,\alpha_4}$. As $\mathfrak{sl}_{2,\xi}$-mod, $\mathfrak{g}[1]=3V_2$. Then we have $(\mathrm{Q}_1,\mathrm{Q}_2)=(\mathrm{Q}_1^\tau,\mathrm{Q}_2^\tau)=(1,3)$. Thus, the orbit is quasi-admissible for $n=2$, raisable if $n \neq 2$.

\section{Proof of Theorem~\ref{Main result}}

\subsection{Preliminaries for classical case}

Let \(\overline{G}^\vee\) be the complex dual group of the covering group \(\overline{G}^{(n)}\) (see \cite{Weissman18}). Denote by \(\mathcal{N}(\overline{G}^\vee)\) the set of nilpotent orbits in \(\operatorname{Lie}(\overline{G}^\vee)\) under conjugation by \(\overline{G}^\vee\), and by \(\mathcal{N}(\mathbf{G})\) the set of geometric (stable) nilpotent orbits. The \textit{covering Barbasch--Vogan duality} is a map
\[
d_{\mathrm{BV},G}^{(n)} \colon \mathcal{N}(\overline{G}^\vee) \longrightarrow \mathcal{N}(\mathbf{G}),
\]
whose explicit construction is given in \cite{GLLS26}. In this subsection we recall the concrete description for classical groups.

Write a partition as \(\mathfrak{p}=[p_1,p_2,\dots ,p_k]\) with \(p_1 \ge p_2 \ge \cdots \ge p_k > 0\). Define \(\mathfrak{p}^+ := [p_1+1,p_2,\dots ,p_k]\) and \(\mathfrak{p}^- := [p_1,\dots ,p_{k-1},p_k-1]\). For a type \(X \in \{B,C,D\}\) we let \(\mathfrak{p}_X\) denote the \(X\)-collapse of \(\mathfrak{p}\). Following the notation of \cite{GLLS26}, we often omit parentheses between superscripts and subscripts; for example, \(\mathfrak{p}_{D\;\;\; B}^{\;\;\; +\;\;\; -}\) stands for \((((\mathfrak{p}_D)^+)_B)^-\).

\subsubsection{Type \(A\)} \label{def:CBV-A}
Consider \(G = \mathrm{GL}_{r}\). Let \(\overline{G}\) be an \(n\)-fold cover associated with a quadratic form \(Q\) satisfying \(Q(\alpha^\vee)=1\) for every simple coroot of \(G\). Its dual group \(\overline{G}^\vee\) is also of type \(A\). For an orbit \(\mathcal{O}_\mathfrak{p} \subset \overline{G}^\vee\) corresponding to a partition \(\mathfrak{p}\) and for any integer \(m \ge 1\) we introduce the partition
\[
\mathfrak{s}(m;n) := (n^{a}b) = (n, n, \ldots, n, b),
\]
where \(m = na + b\) with \(0 \le b < n\). Then
\[
d_{\mathrm{BV},G}^{(n)}(\mathfrak{p}) = d^{(n)}_{\mathrm{com},A}(\mathfrak{p}) := \sum_{i=1}^{k} \mathfrak{s}(p_i;n),
\]
the sum being taken componentwise.

\subsubsection{Type \(B\)} \label{def:CBV-B}
Consider \(G = \mathrm{SO}_{2r+1}\). Let \(\overline{G}\) be the \(n\)-fold cover attached to the unique quadratic form \(Q\) with \(Q(\alpha^\vee)=2\) for all short coroots of \(G\). Set \(n^* := n/\gcd(n,2)\). The dual group is
\[
\overline{G}^{\vee} \simeq 
\begin{cases}
\operatorname{Sp}_{2r} & \text{if } n^* \text{ is odd}, \\[2pt]
\operatorname{SO}_{2r+1} & \text{if } n^* \text{ is even}.
\end{cases}
\]
For an orbit of \(\overline{G}^{\vee}\) represented by a partition \(\mathfrak{p}\),
\begin{equation*}
d_{\text{BV},G}^{(n)}(\mathfrak{p}) =
\begin{cases}
d_{\mathrm{com},A}^{(n^*)}(\mathfrak{p})^+_{\;\;\; B} & \text{if $n^*$ is odd;} \\
d_{\mathrm{com},A}^{(n^*)}(\mathfrak{p})_{B} & \text{if $n^*$ is even.}
\end{cases}
\end{equation*}

\subsubsection{Type \(C\)} \label{def:CBV-C}
Consider \(G = \mathrm{Sp}_{2r}\). Let \(\overline{G}\) be the \(n\)-fold cover corresponding to the unique \(Q\) such that \(Q(\alpha^\vee)=1\) for the short coroot of \(G\). Then
\[
\overline{G}^{\vee} =
\begin{cases}
\operatorname{SO}_{2r+1} & \text{if } n \text{ is odd}, \\[2pt]
\operatorname{Sp}_{2r} & \text{if } n \text{ is even}.
\end{cases}
\]
Given any orbit of $\overline{G}^{\vee}$ represented by a partition $\mathfrak{p}$, we have
\begin{equation*}
d_{\text{BV},G}^{(n)}(\mathfrak{p}) =
\begin{cases}
d_{\mathrm{com},A}^{(n)}(\mathfrak{p})^-_{\;\;\; C} & \text{if $n$ is odd;} \\
d_{\mathrm{com},A}^{(n/2)}(\mathfrak{p})^{+-}_{\;\;\;\;\;\; C} & \text{if $n$ is even with $n/2$ odd;} \\
d_{\mathrm{com},A}^{(n/2)}(\mathfrak{p})_{C} & \text{if $n$ is even with $n/2$ even.}
\end{cases}
\end{equation*}

\subsubsection{Type \(D\)} \label{def:CBV-D}
Consider \(G = \mathrm{SO}_{2r}\). Let \(\overline{G}\) be the \(n\)-fold cover associated with the unique quadratic form \(Q\) satisfying \(Q(\alpha^\vee)=2\) for every coroot of \(G\). Its dual group is
\[
\overline{G}^{\vee} \simeq \operatorname{SO}_{2r}.
\]
For each partition $\mathfrak{p}$ of $\overline{G}^{\vee}$, we have
\begin{equation*}
d_{\text{BV},G}^{(n)}(\mathfrak{p}) =
d_{\mathrm{com},A}^{(n^*)}(\mathfrak{p})_{D}
\end{equation*}
For the very even orbits the definition requires a slightly more refined prescription; since this refinement does not affect our proofs, we omit the details.

We recall the explicit criteria for quasi‑admissibility of \(F\)-split nilpotent orbits in classical groups, as obtained in \cite{GLT25}.

\begin{proposition} \label{prop:QAA}
    Retain the setting in \ref{def:CBV-A}. Let \(\mathcal{O}\) be a \(F\)-split nilpotent orbit of type \(A\) whose Jordan type is given by a partition
    \[
    \mathfrak{p} = [p_1^{d_1}, \dots , p_k^{d_k}, q_1, \dots , q_m],
    \]
    where the \(p_i\) are distinct positive integers with multiplicities \(d_i\ge 2\) and the \(q_j\) are distinct positive integers appearing only once. Then \(\mathcal{O}\) is \(\overline{G}^{(n)}\)-quasi‑admissible if and only if \(n\) divides \(p_i\) for every \(i\).
\end{proposition}

For the other classical types we write a partition in the form
\[
\mathfrak{p}_{\mathcal{O}} = [p_1^{d_1},\dots ,p_i^{d_i},\dots ,p_k^{d_k},\; q_1^{e_1},\dots ,q_j^{e_j},\dots ,q_m^{e_m}],
\]
where the \(p_i\) are distinct even numbers, the \(q_j\) are distinct odd numbers, and all exponents \(d_i, e_j\) are positive. For such a partition and a positive integer \(p\) we define
\[
\mathfrak{A}(\mathfrak{p}, p) = \sum_{\substack{i:\; p_i > p,\\[1pt] p_i-p+1 \in 2\mathbb{Z}}} d_i, \qquad 
\mathfrak{B}(\mathfrak{p}, p) = \sum_{\substack{i:\; p_i < p,\\[1pt] p_i-p+1 \in 2\mathbb{Z}}} d_i .
\]

\begin{proposition} \label{prop:QABCD}
    Retain the setting in \ref{def:CBV-B}, \ref{def:CBV-C} and \ref{def:CBV-D}. Let \(\mathcal{O}\) be a \(F\)-split nilpotent orbit with partition \(\mathfrak{p}_{\mathcal{O}}\) as above.
    \begin{enumerate}[label=(\roman*)]
        \item If \(\mathcal{O}\) is of type \(B\) or \(D\), it is \(\overline{G}^{(n)}\)-quasi‑admissible if and only if for every \(i\) with \(d_i\ge 2\) and every \(j\) with \(e_j\ge 3\) the following conditions hold:
        \begin{itemize}
            \item if \(n\) is odd: \(n \mid p_i\) and \(\mathfrak{B}(\mathfrak{p}_{\mathcal{O}}, p_i) \in 2\mathbb{Z}\);
            \item if \(n\) is even and \(\mathfrak{B}(\mathfrak{p}_{\mathcal{O}}, p_i) \in 2\mathbb{Z}\): \(n \mid p_i\);
            \item if \(n\) is even and \(\mathfrak{B}(\mathfrak{p}_{\mathcal{O}}, p_i) \notin 2\mathbb{Z}\): \(\gcd(n, p_i)=n/2\);
            \item if \(n\) is odd: \(n \mid q_j\);
            \item if \(n\) is even and \(e_j\ge 4\): \(n \mid (2q_j)\);
            \item if \(n\) is even and \(e_j = 3\): \(n \mid (4q_j)\).
        \end{itemize}
        
        \item If \(\mathcal{O}\) is of type \(C\), it is \(\overline{G}^{(n)}\)-quasi‑admissible if and only if for every \(i\) with \(d_i\ge 3\) and every \(j\) with \(e_j\ge 2\) the following conditions hold:
        \begin{itemize}
            \item if \(n\) is odd: \(n \mid q_j\) and \(\mathfrak{A}(\mathfrak{p}_{\mathcal{O}}, q_j) \in 2\mathbb{Z}\);
            \item if \(n\) is even and \(\mathfrak{A}(\mathfrak{p}_{\mathcal{O}}, q_j) \in 2\mathbb{Z}\): \(n \mid q_j\);
            \item if \(n\) is even and \(\mathfrak{A}(\mathfrak{p}_{\mathcal{O}}, q_j) \notin 2\mathbb{Z}\): \(\gcd(n, q_j)=n/2\);
            \item if \(n\) is odd: \(n \mid p_i\);
            \item if \(n\) is even and \(d_i\ge 4\): \(n \mid (2p_i)\);
            \item if \(n\) is even and \(d_i = 3\): \(n \mid (4p_i)\).
        \end{itemize}
    \end{enumerate}
\end{proposition}

A first reduction is provided by the trivial cover.

\begin{lemma} \label{lem:n=1}
    Let \(G\) be a almost-simple classical group as in \ref{def:CBV-A}-\ref{def:CBV-D}. Every \(F\)-split nilpotent orbit whose geometry type that lies in the image of the covering Barbasch--Vogan duality map \(d_{\mathrm{BV},G}^{(1)}\) is \(G\)-quasi‑admissible.
\end{lemma}
\begin{proof}
    As explained in \cite[\S 3]{GLT25}, for a classical group a nilpotent orbit is \(G\)-quasi‑admissible precisely when it is special. The statement for type \(A\) follows immediately. For types \(B,C,D\) it is a consequence of \cite[Proposition~6.3.11]{CM93} together with \cite[Lemma~3.11(i),(iv)]{GLLS26} and the explicit description of \(d_{\mathrm{BV},G}^{(1)}\) given in \ref{def:CBV-A}-\ref{def:CBV-D}.
\end{proof}

We shall also need a concrete algorithmic description of the operations that appear in the definition of the covering Barbasch--Vogan map. The following procedure is analogous to \cite[Proposition~6.3.8]{CM93}. Given an ordered sequence \((p_1^*,\dots ,p_{2m}^*)\) with \(p_1^*\ge \cdots \ge p_{2m}^*\), we define an operation \((\star)\) on it:
\begin{enumerate}[label=($\star$)]
    \item Let \(i_1<\dots <i_m\) be the indices for which \(p_{2i_j-1}^* > p_{2i_j}^*\). Replace each pair \((p_{2i_j-1}^*,p_{2i_j}^*)\) by \((p_{2i_j-1}^*-1,\;p_{2i_j}^*+1)\); leave all other entries unchanged.
\end{enumerate}

We can deduce the following proposition.

\begin{proposition} \label{prop:alg}
    Let \(\mathfrak{p}\) be a partition.
    \begin{enumerate}[label=(\roman*)]
        \item \textbf{\(X\)-collapse \((X\in\{B,D\})\):} Take all even parts of \(\mathfrak{p}\) (adding a zero if necessary) and arrange them in decreasing order as \((p_1^*,\dots ,p_{2m}^*)\). Apply \((\star)\) to this sequence. The \(X\)-collapse of \(\mathfrak{p}\) is obtained by combining the resulting parts with the original odd parts.
        
        \item \textbf{\(X\)-collapse \((X=C)\):} Take all odd parts of \(\mathfrak{p}\) and arrange them in decreasing order as \((p_1^*,\dots ,p_{2m}^*)\). Apply \((\star)\); the result together with the original even parts gives the \(C\)-collapse.
        
        \item \textbf{\((\cdot)^+_{\;\;B}\):} Take all even parts of \(\mathfrak{p}\) (adding a zero if necessary) and arrange them in decreasing order as \((p_0^*,p_1^*,\dots ,p_{2m}^*)\). Apply \((\star)\) to \((p_1^*,\dots ,p_{2m}^*)\) and replace \(p_0^*\) by \(p_0^*+1\).
        
        \item \textbf{\((\cdot)^-_{\;\;C}\):} Take all odd parts of \(\mathfrak{p}\) and arrange them in decreasing order as \((p_1^*,\dots ,p_{2m}^*,p_{2m+1}^*)\). Apply \((\star)\) to \((p_1^*,\dots ,p_{2m}^*)\) and replace \(p_{2m+1}^*\) by \(p_{2m+1}^*-1\).
        
        \item \textbf{\((\cdot)^{+-}_{\;\;\;\;C}\):} Take all odd parts of \(\mathfrak{p}\) and arrange them in decreasing order as \((p_0^*,p_1^*,\dots ,p_{2m}^*,p_{2m+1}^*)\). Apply \((\star)\) to \((p_1^*,\dots ,p_{2m}^*)\) and replace \(p_0^*\) by \(p_0^*+1\) and \(p_{2m+1}^*\) by \(p_{2m+1}^*-1\).
    \end{enumerate}
    In each case, if the required sequence contains fewer than two elements, the operation \((\star)\) is vacuous.
\end{proposition}

Observe that the operations in Proposition~\ref{prop:alg} always \emph{increase} odd parts and \emph{decrease} even parts for types \(B\) and \(D\), while they do the opposite for type \(C\).

Fix a partition \(\mathfrak{p}=[o_1,\dots ,o_h]\) labelling an orbit of the dual group \(\overline{G}^\vee\). Set
\[
\mathfrak{p}_1 := d^{(1)}_{\mathrm{com},A}(\mathfrak{p}) = [p_1^{\prime\prime},\dots ,p_u^{\prime\prime}],\qquad
\mathfrak{p}_{n} := d^{(n)}_{\mathrm{com},A}(\mathfrak{p}) = [p_1^\prime,\dots ,p_v^\prime],
\]
and write the image under the covering Barbasch-Vogan duality as
\[
\mathfrak{q} := d_{\mathrm{BV},G}^{(n)}(\mathfrak{p}) = [q_1^\prime,\dots ,q_w^\prime],
\]
all parts being listed in non‑increasing order. By construction we have, for each \(i\),
\[
p_i^\prime = \sum_{j=n(i-1)+1}^{ni} p_j^{\prime\prime}
\]
(with the convention that \(p_j^{\prime\prime}=0\) if \(j>u\)). A part \(p_i^\prime\) will be called \textit{constant} if all summands \(p_j^{\prime\prime}\) appearing in it are equal.

\subsection{Type \(A\)} \label{P:A}
For type \(A\) we have \(\mathfrak{q}=\mathfrak{p}_{n}\). By the decreasing order, a part \(p_i^\prime\) can have multiplicity greater than one only when it is constant. In that case clearly \(n\mid p_i^\prime\). Proposition~\ref{prop:QAA} therefore implies that every \(F\)-split nilpotent orbit whose geometry type included in the image of \(d_{\mathrm{BV},G}^{(n)}\) is quasi‑admissible.

\subsection{Types \(B\) and \(D\)} \label{P:B}
The argument is uniform for the two types. Recall that \(n^* = n/\gcd(n,2)\).

\subsubsection*{Case \(n^*\) odd}
Let \(q_j^\prime\) be an even part of \(\mathfrak{q}\). It must arise from the part with uniform values $p_i^\prime = q_j^\prime$ of \(\mathfrak{p}_{n^*}\) that will be decreased by the combinatorial operations. If $p_i^\prime$ is not constant, then its multiplicity must be $1$. Thus, if the multiplicity of $q_j^\prime$ is not small than $2$, we always have $p_i^\prime$ constant. Consequently \(n^*\mid p_i^\prime\), and because \(n^*\) is odd we obtain \(n\mid q_j^\prime\).

From the algorithm in Proposition~\ref{prop:alg} one deduces the congruence
\begin{equation*}\label{eq:modB}
\mathfrak{B}(\mathfrak{q},q_j^\prime) \equiv \mathfrak{B}(\mathfrak{p}_{n^*},p_i^\prime) \pmod 2 .
\end{equation*}
If \(\mathfrak{B}(\mathfrak{q},q_j^\prime)\) were odd, then \(\mathfrak{B}(\mathfrak{p}_{n^*},p_i^\prime)\) would be odd, which would force \(\sum_{l=i+1}^{v}p_l^\prime\) to be odd. However, Lemma~\ref{lem:n=1} applied to the trivial cover gives \(\mathfrak{B}(\mathfrak{p}_1,p_{n^*i}^{\prime\prime})\in2\mathbb{Z}\); this implies that \(\sum_{l=n^*i+1}^{u}p_l^{\prime\prime}= \sum_{l=i+1}^{v}p_l^\prime\) is even, a contradiction. Hence \(\mathfrak{B}(\mathfrak{q},q_j^\prime)\) is even.

Now let \(q_j^\prime\) be an odd part of \(\mathfrak{q}\) with multiplicity at least three. Such a part must come from the part with uniform values $p_i^\prime = q_j^\prime$ of \(\mathfrak{p}_{n^*}\) that will be increased. If $p_i^\prime$ is constant, then $n^* \mid q_j^\prime$. If $p_i^\prime$ is not constant, then $p_{i-1}^\prime = p_i^\prime + 1$, $p_{i+1}^\prime = p_i^\prime - 1$. It means that there are only two distinct values in $[p_{n^*(i-1)+1}^{\prime\prime}, \cdots, p_{n^*i}^{\prime\prime}]$. If $p_{i-1}^\prime$ and $p_{i+1}^\prime$ are constant, then $p_{n^*(i-1)+1}^{\prime\prime} = p_{n^*(i-1)}^{\prime\prime}$, $p_{n^*i}^{\prime\prime} = p_{n^*i+1}^{\prime\prime}$ are even, which contradiction with the oddness of $q_j^\prime$. Without loss of generality, we assume that $p_{i-1}^\prime$ is not constant, then $p_{i-1}^\prime > p_i^\prime + 1$, which also leads to a contradiction.

All conditions of Proposition~\ref{prop:QABCD}(i) are thus satisfied; consequently the \(F\)-split nilpotent orbit corresponding to \(\mathfrak{q}\) is \(\overline{G}^{(n)}\)-quasi‑admissible.

\subsubsection*{Case \(n^*\) even}
For an even part \(q_j^\prime\) with multiplicity at least two, the same reasoning as above produces a constant \(p_i^\prime = q_j^\prime\) in \(\mathfrak{p}_{n^*}\). The parity relation~\eqref{eq:modB} still holds. If \(\mathfrak{B}(\mathfrak{p}_{n^*},p_i^\prime)\) is even, then there exists an index \(t_0\) such that \(\sum_{t=1}^{t_0}o_t = n^*(i-1)\sum_{l=i}^{v}p_l^\prime\) is even; by the definition of \(\mathfrak{p}\) this forces \(t_0 = p_i^\prime/n^*\) to be even, whence \(n\mid p_i^\prime = q_j^\prime\). If \(\mathfrak{B}(\mathfrak{p}_{n^*},p_i^\prime)\) is odd, the same argument shows that \(p_i^\prime/n^*\) is odd, and therefore \(\gcd(p_i^\prime,n)=n/2\).

For an odd part \(q_j^\prime\) with multiplicity at least three the discussion is identical to the case where \(n^*\) is odd.

Thus in all situations the criteria of Proposition~\ref{prop:QABCD}(i) are met, and the orbit is quasi‑admissible.

\subsection{Type \(C\)} \label{P:C}
The verification for even parts in type \(C\) is completely analogous to the one for types \(B\) and \(D\); we therefore concentrate on the odd parts.

\subsubsection*{Case \(n\) odd}
Let \(q_j^\prime\) be an odd part of \(\mathfrak{q}\) with multiplicity at least two. It originates from the part with uniform values $p_i^\prime = q_j^\prime$ of \(\mathfrak{p}_{n}\) that will be decreased. It is clear that $p_i^\prime$ must be constant. Thus, we have $n \mid p_i^\prime = q_j^\prime$.

From the algorithm we obtain the parity relation
\begin{equation*}\label{eq:modA}
\mathfrak{A}(\mathfrak{q},q_j^\prime) \equiv \mathfrak{A}(\mathfrak{p}_{n},p_i^\prime) \pmod 2 .
\end{equation*}
If \(\mathfrak{A}(\mathfrak{p}_{n},p_i^\prime)\) were odd, there would exist an index \(t_0\) such that \(\sum_{t=t_0}^{h}o_t\) is odd, which would imply \(\sum_{t=1}^{t_0-1}o_t\) even; by definition of \(\mathfrak{p}\) this would force \(p_i^\prime/n\) to be even, contradicting the oddness of \(p_i^\prime\). Hence \(\mathfrak{A}(\mathfrak{p}_{n},p_i^\prime)\) is even, and so is \(\mathfrak{A}(\mathfrak{q},q_j^\prime)\).

\subsubsection*{Case \(n\) even with \(n/2\) odd}
In this situation the parity relation becomes
\[
\mathfrak{A}(\mathfrak{q},q_j^\prime) \equiv \mathfrak{A}(\mathfrak{p}_{n},p_i^\prime) + 1 \pmod 2 .
\]
An argument parallel to the previous one shows that \(\mathfrak{A}(\mathfrak{p}_{n},p_i^\prime)\) must be odd. Additionally, we have $\gcd(n, q_j^\prime) = n/2$ since $q_j^\prime \notin 2\mathbb{Z}$.

\subsubsection*{Case \(n\) even with \(n/2\) even}
Here \(n/2\) is even, so any odd part of \(\mathfrak{q}\) can occur at most once. The conditions of Proposition~\ref{prop:QABCD}(ii) are therefore vacuous for odd parts.

Collecting the three cases we see that all requirements of Proposition~\ref{prop:QABCD}(ii) are satisfied; hence the orbit is quasi‑admissible.

\medskip

The discussion above completes the proof of Theorem~\ref{Main result} for classical groups.

\subsection{Proof of Theorem~\ref{Main result}}
Combining the previous sections we obtain our second main result.

\begin{theorem} \label{thm:main}
    Let \(\overline{G}^{(n)}\) be the \(n\)-fold cover of any of \(\mathrm{GL}_r\), \(\mathrm{SO}_{2r+1}\), \(\mathrm{SO}_{2r}\), \(\mathrm{Sp}_{2r}\) and simply-connected almost-simple exceptional Lie group. Then for every positive integer \(n\), every \(F\)-split nilpotent orbit whose geometry type contained in the image of the covering Barbasch--Vogan duality map \(d_{\mathrm{BV},G}^{(n)}\) is \(\overline{G}^{(n)}\)-quasi‑admissible.
\end{theorem}
\begin{proof}
    The case of classical groups is proved in \ref{P:A}-\ref{P:C}. For exceptional groups, the claim follows from the explicit calculations presented in §\ref{sec:E} together with the data compiled in \cite[\S~3.4,\S~3.5]{GLT25} and the description of the covering duality given in \cite[Appendix.A]{GLLS26}.
\end{proof}

\subsection{Further discussion}
Retain the setting of Theorem~\ref{thm:main}. An interesting phenomenon observed is the existence of nilpotent orbits that are not quasi‑admissible for \emph{any} cover degree \(n\). To elaborate, we introduce the following notation. By abuse of notation, we identify each \(F\)-split nilpotent orbit with its geometry type. For each integer \(n\ge 1\) let
\[
\mathcal{N}^{(n)}_{\mathrm{QA}}(\mathbf{G})\subset\mathcal{N}(\mathbf{G}),\qquad 
\mathcal{N}^{(n)}_{\mathrm{BV}}(\mathbf{G})\subset\mathcal{N}(\mathbf{G})
\]
denote, respectively, the set of \(\overline{G}^{(n)}\)-quasi‑admissible orbits and the image of the covering Barbasch--Vogan map \(d_{\mathrm{BV},G}^{(n)}\). For \(W\in\{\mathrm{QA},\mathrm{BV}\}\) define
\begin{align*}
    \mathcal{N}_{W}(\mathbf{G})&:=\bigcup_{n \geq 1} \mathcal{N}^{(n)}_{W}(\mathbf{G})\\
    \mathcal{N}^{(0)}_{W}(\mathbf{G})&:=\mathcal{N}(\mathbf{G})-\mathcal{N}_{W}(\mathbf{G})
\end{align*}
Thus \(\mathcal{N}^{(0)}_{\mathrm{QA}}(\mathbf{G})\) consists of orbits that are not quasi‑admissible for any $n$, and \(\mathcal{N}^{(0)}_{\mathrm{BV}}(\mathbf{G})\) consists of orbits that never appear as images of the covering Barbasch-Vogan duality. Regarding \(\mathcal{N}^{(0)}_{\mathrm{QA}}(\mathbf{G})\), the following statement holds.

\begin{proposition} \label{prop:nonqa}
    Keep the notion above.
    \begin{enumerate}[label=(\roman*)]
        \item If \(\mathbf{G}\) is of type \(B\), then \(\mathcal{N}^{(0)}_{\mathrm{QA}}(\mathbf{G})\neq\varnothing\) if and only if \(r\ge 4\).
        \item If \(\mathbf{G}\) is of type \(C\), then \(\mathcal{N}^{(0)}_{\mathrm{QA}}(\mathbf{G})\neq\varnothing\) if and only if \(r\ge 5\).
        \item If \(\mathbf{G}\) is of type \(D\), then \(\mathcal{N}^{(0)}_{\mathrm{QA}}(\mathbf{G})\neq\varnothing\) if and only if \(r\ge 6\).
        \item If \(\mathbf{G}\) is of type \(E_r\), then \(\mathcal{N}^{(0)}_{\mathrm{QA}}(\mathbf{G})\neq\varnothing\) if and only if \(r=7,8\).
    \end{enumerate}
\end{proposition}
\begin{proof}
    For classical groups, explicit orbits witnessing non‑emptyness are:
    \begin{itemize}
        \item type \(B\) (\(r\ge4\)): the orbit with partition \([2^{2},1^{2r-3}]\);
        \item type \(C\) (\(r\ge5\)): the orbit with partition \([3^{2},2,1^{2r-8}]\);
        \item type \(D\) (\(r\ge6\)): the orbit with partition \([3,2^{2},1^{2r-7}]\).
    \end{itemize}
    In each case one checks directly, using Proposition~\ref{prop:QABCD}, that no integer \(n\) satisfies the required divisibility conditions. The finitely many remaining low‑rank cases can be verified by a computer enumeration.

    For the exceptional case, the statement follows directly by §\ref{sec:E} together with the data compiled in \cite[\S~3.4,\S~3.5]{GLT25}.
\end{proof}

Note that we always have \(\mathcal{N}^{(0)}_{\mathrm{QA}}(\mathbf{G})\subset\mathcal{N}^{(0)}_{\mathrm{BV}}(\mathbf{G})\). A natural question is when the two ``special'' sets \(\mathcal{N}^{(0)}_{\mathrm{QA}}(\mathbf{G})\) and \(\mathcal{N}^{(0)}_{\mathrm{BV}}(\mathbf{G})\) actually coincide. By Proposition~\ref{prop:nonqa} and a case‑by‑case inspection of the low‑rank tables, \(\mathcal{N}^{(0)}_{\mathrm{BV}}(\mathbf{G})\) is non‑empty precisely when \(\mathbf{G}\) is of type \(B\) (\(r\ge4\)), type \(C\) (\(r\ge5\)), type \(D\) (\(r\ge6\)), or type \(E\) (\(r=7,8\)). From Tables~\ref{tab:e7}--\ref{tab:e8c} we see that
\[
\mathcal{N}^{(0)}_{\mathrm{QA}}(E_7)=\mathcal{N}^{(0)}_{\mathrm{BV}}(E_7),\qquad
\mathcal{N}^{(0)}_{\mathrm{BV}}(E_8)-\mathcal{N}^{(0)}_{\mathrm{QA}}(E_8)=\{2A_2+A_1\}.
\]

For the classical cases, we have carried out computer calculations for relatively small ranks, based on which we expect the following result to hold. However, a proof is not yet available to us at this stage.

\begin{itemize}
    \item For types \(B\) and \(D\), \(\mathcal{N}^{(0)}_{\mathrm{QA}}(\mathbf{G})\subsetneq\mathcal{N}^{(0)}_{\mathrm{BV}}(\mathbf{G})\) if and only if the rank \(r\ge 10\).
    \item For type \(C\), the two sets coincide: \(\mathcal{N}^{(0)}_{\mathrm{QA}}(\mathbf{G})=\mathcal{N}^{(0)}_{\mathrm{BV}}(\mathbf{G})\) for all ranks.
\end{itemize}

%%%

\end{document}